\documentclass[11pt, reqno]{amsart}
\usepackage{amsmath, amssymb, amsthm, graphicx,marvosym}
\allowdisplaybreaks
\usepackage{cancel}
\usepackage{cite}
\usepackage[nobysame]{amsrefs}[11pts]
\usepackage{bm}
\usepackage{color}
\newtheorem{theorem}{Theorem}[section]
\newtheorem{definition}{Definition}[section]
\newtheorem{lemma}{Lemma}[section]
\newtheorem{remark}{Remark}[section]

\newtheorem{corollary}{Corollary}[section]
\numberwithin{equation}{section}
\numberwithin{figure}{section}

\makeatletter \@addtoreset{equation}{section} \makeatother

\def\tilde{\widetilde}

\newcommand{\fr}{\frac}
\newcommand{\ef}{\eqref}

\newcommand{\beq}{\begin{equation}}
\newcommand{\eeq}{\end{equation}}
\newcommand{\eps}{\varepsilon}

\newcommand{\dv}{{\mathrm {div}}}

\begin{document}

\title[Degenerate compressible Navier-Stokes equations]{Well-posedness    of regular solutions for     3-D full   compressible Navier-Stokes equations with  degenerate viscosities and  heat conductivity }

\author{Qin Duan}
\address[Q. Duan]{School     of Mathematical Sciences, Shenzhen  University, Shenzhen 518060, China.
}
\email{\tt qduan@szu.edu.cn}

\author{Zhouping Xin}
\address[Z.P. Xin]{The Institute of Mathematical Sciences and
Department of Mathematics, The Chinese University of Hong Kong, Shatin, N.T., Hong Kong.}
\email{\tt zpxin@ims.cuhk.edu.hk}

\author{Shengguo Zhu}
\address[S.G.  Zhu]{School of Mathematical Sciences, CMA-Shanghai and MOE-LSC,  Shanghai Jiao Tong University, Shanghai 200240, China.}
\email{\tt zhushengguo@sjtu.edu.cn}


\date{\today}

\begin{abstract}
For the degenerate viscous and heat conductive compressible fluids,  the momentum  equations and the energy equation  are    degenerate both in the time evolution and spatial  dissipation  when vacuum appears,  and then the physical entropy $S$ behaves singularly,
which make it challenging to study the  corresponding  existence of regular solutions with   high order regularities of  $S$ near the vacuum.
In this paper,   for the physically important case that    the   coefficients of viscosities and heat  conductivity  depend on  the absolute  temperature $\theta$ in a  power  law ($\theta^\nu$ with some $\nu>0$),    we establish the    local existence   of the unique regular solution with far field vacuum to  the   Cauchy problem of the 3-D  full  compressible  Navier-Stokes equations. 
In our analysis,    the corresponding  vacuum problem is formulated  in terms of the  density $\rho$, velocity $u$ and  $S$ instead of $(\rho, u,\theta)$, which makes it possible to  compare the degeneracy degrees of the time evolution and  spatial  dissipations  near the vacuum in terms of  $\rho's$  powers. However, for   heat conductive fluids, both a degenerate spatial dissipation  and a source term  related to $\triangle \rho^{\gamma-1}$ with   the adiabatic exponent $\gamma$, will appear  in the time evolution equation for  $S$, which 
makes it formidable to study the 
propagation  of  regularities of $S$. Fortunately,
we can choose  proper weights to control the dynamics  of $(\rho, u,S)$  by   introducing     an   enlarged reformulated system,
which  provides  an effective propagation  mechanism for $S's$  high order regularities near the vacuum.

\end{abstract}
\maketitle

\tableofcontents

\section{Introduction}

The motion of a  compressible viscous, heat-conductive, and Newtonian polytropic fluid occupying a spatial domain $\Omega \subset \mathbb{R}^3$  is governed by the following full compressible Navier-Stokes equations (\textbf{CNS}):
\begin{equation}\label{1}
\left\{\begin{aligned}
&\rho_t+\text{div}(\rho u)=0,\\[2pt]
&(\rho u)_t+\text{div}(\rho u\otimes u)+\nabla P=\text{div}\mathbb{T},\\[2pt]
&(\rho \mathcal{E})_t+\text{div}(\rho \mathcal{E} u+Pu)=\text{div}(u\mathbb{T})+\text{div}(\kappa\nabla\theta).
\end{aligned}\right.
\end{equation}
Here and throughout, $\rho\geq 0$ denotes the mass density, $u=(u^{(1)},u^{(2)},u^{(3)})^{\top}$ the fluid velocity, $P$ the pressure of the fluid,  $\theta$ the absolute temperature, $\mathcal{E}=e+\frac{1}{2}|u|^2$  the specific total energy,  $e$ the specific internal energy, $x=(x_1,x_2,x_3)^\top\in \Omega$  the Eulerian  spatial coordinate  and finally $t\geq 0$  the time coordinate. The  equation of state for polytropic fluids satisfies 
\begin{equation}\label{2}
    P=R\rho\theta=(\gamma-1)\rho e=Ae^{S/c_v}\rho^\gamma,\hspace{2mm} e=c_v\theta,\hspace{2mm}c_v=\frac{R}{\gamma-1},
\end{equation}
where $R$ is the gas constant, $A$ is a   positive constant, $c_v$ is  the specific heat at constant volume,  $\gamma>1$ is the adiabatic exponent and $S$ is the specific entropy. $\mathbb{T}$ is the viscous stress tensor given by
\begin{equation}\label{3}
\mathbb{T}=2\mu D(u)+\lambda \text{div}u \mathbb{I}_3,
\end{equation}
where $D(u)=\frac{\nabla u+(\nabla u)^\top}{2}$ is the deformation tensor, $\mathbb{I}_3$ is the $3\times 3$ identity matrix, $\mu$ is the shear viscosity coefficient, and $\lambda+\frac{2}{3}\mu$ is the bulk viscosity coefficient. $\kappa$ denotes the   coefficient of the heat  conductivity.

For rarefied gases,  the compressible Navier-Stokes equations can be derived from the Boltzmann equations; see Chapman-Cowling   \cite{chap}, Deng-Hani-Ma \cite{Dengyu},  Li-Qin \cite{tlt} and so on. Under some proper physical assumptions,   $(\mu,\lambda)$ and $\kappa$ are not constants but functions of the absolute temperature $\theta$ such as:
\begin{equation}
\label{eq:1.5g}
\begin{split}
\mu(\theta)=&r_1 \theta^{\frac{1}{2}}F(\theta),\quad  \lambda(\theta)=r_2 \theta^{\frac{1}{2}}F(\theta), \quad \kappa(\theta)=r_3 \theta^{\frac{1}{2}}F(\theta)
\end{split}
\end{equation}
for some   constants $r_i$ $(i=1,2,3)$.
  Actually for the cut-off inverse power force models, if the intermolecular potential varies as $r^{-\Upsilon}$,
 where $ r$ is intermolecular distance,  then 
 \begin{equation}
\label{eq:1.6g}
F(\theta)=\theta^{\varpi}\quad  \text{with}\quad   \varpi=\frac{2}{\Upsilon} \in [0,\infty)
\end{equation}
in \ef{eq:1.5g}.   In particular, for Maxwellian molecules,
$\Upsilon = 4$  and $\varpi=\frac{1}{2}$;
 for rigid elastic spherical molecules,
$\Upsilon=\infty$ and  $\varpi=0$; while for  ionized gas, $\Upsilon=1$ and $\varpi=2$.  

In the current  paper, we will consider the following case:
\begin{equation}\label{4}
\mu(\theta)=\alpha\theta^\nu,\hspace{2mm}\lambda(\theta)=\beta\theta^\nu,\hspace{2mm}\kappa(\theta)=\daleth\theta^\nu,
\end{equation}
where $(\alpha,\beta,\daleth,\nu)$ are all constants satisfying
\begin{equation}\label{5}
\alpha>0,\quad  2\alpha+3\beta\geq 0,\quad \daleth>0 \quad \text{and} \quad 0<\delta=(\gamma-1)\nu <1.
\end{equation}
 In terms of $(\rho,u,S)$, it follows from \eqref{2}-\ef{3} and \eqref{4} that   \eqref{1} can be rewritten as the following  system which does not explicitly contain negative powers of $\rho$:
\begin{equation}\label{8}
\left\{\begin{aligned}
\displaystyle
&\rho_t+\text{div}(\rho u)=0,\\
\displaystyle
& \underbrace{\rho(u_t+u\cdot \nabla u)}_{\circledast}+\nabla P= \underbrace{A^\nu R^{-\nu}\text{div}(\rho^{\delta}e^{\frac{S}{c_v}\nu}Q(u))}_{\Diamond},\\
\displaystyle
&\underbrace{P \big(S_t+u\cdot \nabla S\big)}_{\circledast}-\underbrace{\digamma \rho^{\delta+\gamma-1}e^{\frac{S}{c_v}\nu}\triangle e^{\frac{S}{c_v}}}_{\Diamond}\\
=& A^{\nu}R^{1-\nu}\rho^{\delta}e^{\frac{S}{c_v}\nu}H(u)+\underbrace{\digamma \rho^{\delta}e^{\frac{S}{c_v}(\nu+1)}\triangle \rho^{\gamma-1}}_{\star}+\Lambda (\rho, S),
\end{aligned}\right.
\end{equation}
where  $\digamma=R \daleth (AR^{-1})^{\nu+1}$, $\circledast$ denotes the degenerate time evolution, $\Diamond$ the degenerate dissipation,  $\star$ the  source term  involving second order   derivatives of  $\rho$, and 
\begin{equation}\label{QH}
\begin{split}
Q(u)=&\alpha(\nabla u+(\nabla u)^\top)+\beta\text{div}u\mathbb{I}_3,\ \   H(u)=\text{div}(uQ(u))-u\cdot \text{div}Q(u),\\[3pt]
\Lambda (\rho,S)=& 2\digamma \rho^{\delta}e^{\frac{S}{c_v}\nu}\nabla \rho^{\gamma-1} \cdot \nabla e^{\frac{S}{c_v}}+\digamma \nabla ( \rho^{\delta}e^{\frac{S}{c_v}\nu})\cdot \nabla ( e^{\frac{S}{c_v}}\rho^{\gamma-1}).
\end{split}
\end{equation}
Set $\Omega=\mathbb{R}^3$. We  will study  the local existence  of the regular solution $(\rho,u,S)$ with far-field vacuum  to the Cauchy problem    $\eqref{8}$  with \eqref{2}, \eqref{5}, \eqref{QH},  and  the following initial data and far field behavior for some constant $\bar{S}$:
\begin{align}
(\rho,u,S)|_{t=0}=(\rho_0(x)>0,u_0(x),S_0(x))\quad  &\text{for} \quad x\in\mathbb{R}^3, \label{6}\\
(\rho,u,S)(t,x)\rightarrow(0,0,\bar{S})\quad  \text{as} \quad |x|\rightarrow\infty\quad &\text{for}\quad t\geq 0.\label{7}
\end{align}

For simplicity, throughout this paper,
for any function space $X$, unless otherwise specified, $X=X(\mathbb{R}^3)$, and 
the following conventions are used:
\begin{equation*}\begin{split}
 & \|f\|_s=\|f\|_{H^s(\mathbb{R}^3)},\quad |f|_p=\|f\|_{L^p(\mathbb{R}^3)},\quad \|f\|_{m,p}=\|f\|_{W^{m,p}(\mathbb{R}^3)},\\
 &|f|_{C^k}=\|f\|_{C^k(\mathbb{R}^3)},  \quad \|f\|_{X_1 \cap X_2}=\|f\|_{X_1}+\|f\|_{X_2},\quad \int  f   =\int_{\mathbb{R}^3}  f \text{d}x, \\
 & D^{k,r}=\{f\in L^1_{loc}(\mathbb{R}^3): |f|_{D^{k,r}}=|\nabla^kf|_{r}<\infty\},\quad  |f|_{D^{k,r}}=\|f\|_{D^{k,r}(\mathbb{R}^3)},\\
&D^{1}_*=\{f\in L^6(\mathbb{R}^3):  |f|_{D^1_*}= |\nabla f|_{2}<\infty\},\quad  D^k=D^{k,2},   \quad  |f|_{D^{1}_*}=\|f\|_{D^{1}_*(\mathbb{R}^3)}.
\end{split}
\end{equation*}

In the case  that    $(\mu,\lambda,\kappa)$ are  constants,  there is a huge literature  on the well-posedness theory for the full  \textbf{CNS} \eqref{1}.
 When  the initial data is away from the vacuum, the local well-posedness of classical solutions to the Cauchy problem of  \eqref{1} follows from the standard symmetric hyperbolic-parabolic structure, cf. \cites{itaya1, itaya2, KA, nash, serrin,tani} and the references therein.  However, such an approach fails  in the presence of the vacuum due to some new difficulties, for example,  the degeneracy of the time evolution in the energy equation:
\begin{equation}\label{dege}
\displaystyle
 \underbrace{\rho (\mathcal{E}_t+u\cdot \nabla \mathcal{E})}_{\circledast}+\text{div}(Pu)=\text{div}(u\mathbb{T})+\text{div}(\kappa\nabla\theta).
\end{equation}
  One of the  main issues  is   to understand the dynamics  of $(u,\theta,S)$ near the vacuum.  Note that near the vacuum,   in the sense that the corresponding equations do not explicitly contain negative powers of $\rho$, the equation $\eqref{1}_3$ for $\theta$  degenerates only in the time evolution, while  the equation $\ef{8}_3$ for $S$
  degenerates  both in the time evolution and spatial  dissipation  even for the case $\nu=0$, which  makes the behavior of  $S$  more singular than that of $\theta$,  and the  study on  the   regularities of  $S$  challenging. Thus, most
of the well-posedness theories  on the full  \textbf{CNS} with vacuum state   developed in the existing literatures  are 
regardless of $S$.  It is worth pointing out that, in the presence of  vacuum, the full  \textbf{CNS} formulated in terms of $(\rho, u,\theta)$ is not \\  equivalent to the one formulated in terms of $(\rho, u,S)$, since the regularities of $(\rho, \theta)$ do not yield any information  for $S$ near the vacuum.   Actually, for general initial data containing vacuum, the local well-posedness of strong solutions to the Cauchy problem  of the three-dimensional (3-D) full  \textbf{CNS} was  obtained by Cho-Kim  \cite{CK}  in terms of $(\rho, u,\theta)$,  and the corresponding    global well-posedness theories   with  suitable small initial energy conditions have been established   by Huang-Li \cite{huangli} with non-vanishing $(\rho, \theta)$ at far fields, and Wen-Zhu \cite{wenzhu} with vanishing $(\rho, \theta)$ at far fields by extending  the corresponding studies on the isentropic case by Huang-Li-Xin \cite{HX1}. 
 It should be noticed that in the far-field vacuum setting considered in \cites{CK, wenzhu}, the velocity is generally controlled in homogeneous Sobolev spaces, with $\sqrt{\rho}u$, rather than $u$, belonging to $(L^\infty(0,T;L^2))$.
In fact, one can not expect that the strong solutions to the full \textbf{CNS} lie in the inhomogeneous Sobolev spaces, if $\rho_0$ has compact support or even decays to zero in  far fields rapidly, see Li–Wang–Xin \cite{ins} and Li-Xin \cite{lx4}. Moreover,  it follows from   Xin–Yan
\cite{zxy}  that the global solutions  in \cites{huangli,wenzhu} must have unbounded
$S$ if initially there is an isolated mass group surrounded by the vacuum region. However, when  $\rho_0$ vanishes only at far fields with a slow decay rate, recently in Li-Xin \cites{lx2,lx3}, for the Cauchy problem of the full \textbf{CNS}, it is shown that the uniform boundedness of $S$  and the $L^2$ regularity of $u$   can be propagated  within the solution's life span.   For specific pressure laws excluding \ef{2}, the
global  existence of so-called “variational” solutions with vacuum  in dimension $d\geq 2$
has been  established  by Feireisl in \cites{fu2,fu3} (see also Poul \cites{poul1} for  the \textbf{CNS}-Poisson system), where $\theta$  satisfies
  an inequality.  
We also  refer the readers to \cites{CH, 	danchin,    hoff,  jensen, lions, MPI, song,  KA2,  mat,  zx} and   the references therein  for some  related  progress.

In contrast to the fruitful development in  the classical case $\nu=0$ in \eqref{4}, the corresponding progress for the degenerate case $\nu>0$ in \eqref{4} is very limited due to the strong  degeneracy and  nonlinearity both in viscosity and heat  conductivity besides the degeneracy in the time evolution near the vacuum. 
Recently,  the degenerate  isentropic \textbf{CNS} (\textbf{DICNS}),  i.e.,   $\eqref{8}_1$-$\eqref{8}_2$  with $S(t,x)$  being constant and $\eqref{8}_3$ ignored, has received extensive attentions, in which  the viscosity vanishes  at vacuum: 
\begin{equation}\label{dengshang}
\displaystyle
 \underbrace{\rho(u_t+u\cdot \nabla u)}_{\circledast}+\nabla P= \underbrace{\text{div}(\rho^{\delta}Q(u))}_{\Diamond}.
\end{equation}
Via making use   of the B-D entropy  in  \cites{bd3,bd2}, some significant achievements   on    weak  solutions with vacuum for the \textbf{DICNS} and related models  have been obtained, cf. \cites{ bd, vassu, bd6,bvy, zhenhua, lz,vayu}. 
On the other hand, there are only a few results available for strong  solutions with finite energy. In particular,  
when  $0<\delta <1$, by introducing    an elaborate  elliptic approach on the singularly weighted regularity estimates for  $u$ 
and  a symmetric hyperbolic system with singularities for some   quantities involving $\rho$ and its derivatives, Xin-Zhu \cite{zz2} identifies a class of initial data admitting one unique 3-D local regular solution with far field vacuum to the  Cauchy problem of $\eqref{8}_1$ and \eqref{dengshang} in some inhomogeneous Sobolev  spaces, which has been extended to be  global-in-time ones  with large data in $\mathbb{R}$   by Cao-Li-Zhu \cite{clz}. The related progress for the cases $\delta\geq 1$ on  smooth solutions with vacuum can also be found in  \cites{zhu,sz3, sz333,zz}.   Since the coefficients of the time evolution and $Q(u)$
are  powers of $\rho$, it is easy  to compare  the order  of the degeneracy  of  these two operators near the vacuum, which enable one  to  select the dominant operator to control the behavior of  $u$  and lead  to the ``hyperbolic-strong singular elliptic" structure in \cites{clz, zz2} and  the  `` quasi-symmetric hyperbolic"--``degenerate elliptic"  structure in \cites{zhu,sz333, zz}. Some other related progress  can also  be found in \cites{ GG,  germain, hailiang,taiping}  and the references therein.

Since $e$, $\theta$ and $S$ are all fundamental states for viscous compressible fluids, it is of great importance to study their dynamics  for the full  \textbf{CNS}, which is a subtle and difficult problem in the presence of vacuum. 
Indeed, in the studies for the well-posedness of classical solutions with vacuum to the full  \textbf{CNS} \eqref{1}-\eqref{3} with degenerate viscosities and heat conductivity of the form  \eqref{eq:1.5g}-\eqref{eq:1.6g}, the structures of the coefficients for the time evolution  and the spatial  dissipation operators are different, and $S$ plays important roles but behaves more singularly  than $\theta$ near the vacuum, which cause substantial difficulties in the analysis and make it difficult to adapt the approaches for the isentropic case in \cites{clz,zhu, sz3,sz333,zz2}. It should be pointed out that due to the physical requirements on $\theta$ and $S$ near the vacuum, it is of more advantages to formulate the \textbf{CNS} \eqref{1}-\eqref{3} in terms of $(\rho, u,S)$ instead of $(\rho, u,\theta)$ in contrast to \cites{CK,huangli, wenzhu} as  illustrated below.
 Since $\theta=AR^{-1}\rho^{\gamma-1} e^{S/c_v}$
for $\rho>0$, one may rewrite  $\eqref{1}_2$-$\eqref{1}_3$ as $\eqref{8}_2$-$\eqref{8}_3$ which do  not  contain explicitly negative powers of $\rho$.
Thus, if  $S$ has  uniform boundedness and high enough regularities, then it is still possible to  compare  the orders of the degeneracy of the time evolution and the spatial  dissipation operators near the vacuum by the powers of $\rho$, and then to  choose proper weights to control the behaviors of the physical quantities. 
However, no matter for the case $\nu=0$ or the case $\nu>0$, due to the  degeneracy  in both the time evolution and the  spatial dissipation in  $\eqref{8}_3$, the physical entropy for  polytropic  gases behaves singularly near the  vacuum, and it is thus a challenge to study its regularities.    Indeed, even  for the case of constant viscosities and heat conductivity, i.e., $\nu=0$ in \eqref{4},  only the boundedness of $S$  has been achieved in Li-Xin \cites{lx2,lx3},  and yet the higher regularities  of $S$ near the vacuum have not been established in the existing literatures. Furthermore, it seems difficult to adapt the approach for  $\nu=0$ in \cites{lx2,lx3} to the case $\nu>0$ due to the stronger degeneracy in spatial dissipations. Recently,  for the case $\nu>0$ and  $\daleth=0$ in $\eqref{8}_3$,      we have shown  in  \cite{dxz} that  the following equation:
\begin{equation}\label{transport}
S_t+u\cdot \nabla S=A^{\nu-1}R^{1-\nu}\rho^{\delta-\gamma}e^{\frac{S}{c_v}(\nu-1)}H(u)
\end{equation}
can provide an effective propagation  mechanism for   regularities of $S$ in $D^1_*\cap D^3$ in 
short time, and the corresponding analysis  depends essentially  on the transport
structure of \eqref{transport}.
However, for the case  $\nu>0$ and $\daleth>0$, although the time evolution equation $\ef{8}_3$ for $S$ is a  degenerate parabolic equation rather  than a transport one, yet it contains nonlinear, degenerate and possibly singular  diffusion term $\digamma \rho^{\delta+\gamma-1}e^{\frac{S}{c_v}\nu}\triangle e^{\frac{S}{c_v}}$, and a complicated strongly  nonlinear and possibly singular  source term $\digamma \rho^{\delta}e^{\frac{S}{c_v}(\nu+1)}\triangle \rho^{\gamma-1}$, which makes it extremely difficult to adapt  the standard  techniques  for  
propagation  of  high order regularities  for parabolic equations 
to modify the techniques for  $\daleth=0$ in \cite{dxz}. Indeed, the argument for establishing  the uniformly high order regularities  of $S$ near the vacuum in short time  used in \cite{dxz} for the case $\daleth=0$ fails here also.

In order to overcome these difficulties, under the assumptions   \eqref{4}-\eqref{5},  we reformulate  the  equations   $\eqref{8}_2$-$\eqref{8}_3$  as
\begin{equation}\label{qiyi}
\left\{\begin{aligned}
&\quad u_t+u\cdot\nabla u +\frac{A\gamma}{\gamma-1}e^{\frac{S}{c_v}}\nabla\rho^{\gamma-1}+A\rho^{\gamma-1} \nabla e^{\frac{S}{c_v}}+\underbrace{A^\nu R^{-\nu}\rho^{\delta-1}e^{\frac{S}{c_v}\nu} Lu}_{\Box}\\
=&A^\nu R^{-\nu}\frac{\delta}{\delta-1}\nabla\rho^{\delta-1}\cdot Q(u)e^{\frac{S}{c_v}\nu} + \underbrace{A^\nu R^{-\nu}\rho^{\delta-1}\nabla e^{\frac{S}{c_v}\nu} \cdot Q(u)}_{\backsim}\\
& \quad \underbrace{\rho^{\frac{1-\delta}{2}}(S_t+u\cdot \nabla S)}_{\circledast}-\underbrace{\digamma A^{-1}\rho^{\frac{\delta-1}{2}}e^{\frac{S}{c_v}(\nu-1)}\triangle e^{\frac{S}{c_v}}}_{\Box}\\
=&A^{\nu-1}R^{1-\nu}\rho^{\frac{1+\delta-2\gamma}{2}}e^{\frac{S}{c_v}(\nu-1)}H(u)\\
&+\underbrace{\digamma A^{-1}\rho^{\frac{1+\delta-2\gamma}{2}}e^{\frac{S}{c_v}\nu}\triangle \rho^{\gamma-1}}_{\backsim}+A^{-1}\rho^{\frac{1-\delta-2\gamma}{2}}e^{-\frac{S}{c_v}}\Lambda(\rho,S),
\end{aligned}\right.
\end{equation}
where $\Box$ denotes the singular dissipation, $\backsim$  the strong singular source term, and  $L$   the Lam\'e operator defined by
$$Lu \triangleq -\alpha\triangle u-(\alpha+\beta)\nabla \mathtt{div}u.$$
It should be emphasized here that the key   in the reformulation  \eqref{qiyi} is the choice  of the degenerate weight $\rho^{\frac{1-\delta}{2}}$ in front of  $S_t+u\cdot \nabla S$ in the equation for the entropy, which is  inspired by the competition of different terms in the system \eqref{8} for weights in singular weighted energy estimates. In fact,  if $\rho$ decays to zero in the far field,   the right hand side of $\eqref{qiyi}_1$ contains a singularity $ \rho^{\delta-1}\nabla e^{\frac{S}{c_v}\nu} \cdot Q(u)$, which is expected to be in $L^2$  for the basic $L^2$ estimate on $u$. Then $\nabla e^{\frac{S}{c_v}\nu}$ should decay very fast in the far field  such that $\rho^{-\delta_*}\nabla e^{\frac{S}{c_v}\nu}\in L^{p_*}$ for some $\delta_*\in (0,1-\delta)$ and $p_*>1$, which can be obtained by dividing the entropy equation $\eqref{8}_3$ by $\rho^{\delta_o}$ for some $\delta_o>0$ (determined by $(\delta_*,p_*)$) on both sides and then carrying out  singular weighted energy estimates on $e^{\frac{S}{c_v}\nu}$. It can be checked that to estimate  the velocity   $u$, it is better to choose  $\delta_*$ or $\delta_o$  larger.  On the other hand, the  source term   $\digamma \rho^{\delta}e^{\frac{S}{c_v}(\nu+1)}\triangle \rho^{\gamma-1}$ in the entropy equation $\eqref{8}_3$ makes it difficult to choose large $\delta_o$. Indeed, if  $\delta_o$ is large enough, then $\rho^{-\delta_o}\digamma \rho^{\delta}e^{\frac{S}{c_v}(\nu+1)}\triangle \rho^{\gamma-1}\notin L^2$. In order to balance these opposite requirements, we finally choose $\delta_o=\frac{2\gamma+\delta-1}{2}$, which turns out to be the desired choice as will be seen later.
While  for the case $\daleth=0$ in $\eqref{8}_3$,  since the source  term   $\digamma \rho^{\delta}e^{\frac{S}{c_v}(\nu+1)}\triangle \rho^{\gamma-1}$ vanishes, then there is no such kind of 
competition for weights mentioned above. One can directly divide $P$ on both sides of $\eqref{8}_3$, and then get  the time evolution equation \eqref{transport} of $S$ when $\daleth=0$, thus  the uniform high order 
 regularities  of $S$ can be obtained by the classical arguments for  transport equations (see \cite{dxz} for details).

Based on   \eqref{qiyi}, to establish the existences of smooth solutions with far field vacuum to \eqref{8}, one still  would  encounter some  notable difficulties  as follows:
\begin{enumerate} 

\item  first,   in contrast to \eqref{transport} for the case $\daleth=0$,   the time evolution for  $S$ is  still degenerate here, and naturally  the corresponding estimates for $(S_t,S_{tt})$  are not as strong as the ones  in \cite{dxz}, which   will increase the difficulty to  obtain the desired  energy estimates on $u$.

\item second, even  in the case that the uniform boundedness of  $S$ can be obtained, the coefficient $\digamma A^{-1}\rho^{\frac{\delta-1}{2}}e^{\frac{S}{c_v}(\nu-1)}$ in front of the Laplace operator $ \triangle$ in  $\eqref{qiyi}_2$ will tend to $\infty$ as $\rho\rightarrow 0$ in the far field, which makes it  highly  non-trivial   to show that  $\rho^{\frac{\delta-1}{2}}e^{\frac{S}{c_v}(\nu-1)}\triangle e^{\frac{S}{c_v}}$ is well defined in some Sobolev space. Moreover, how to utilize the smoothing effect of this singular elliptic operator when $\rho$ loses its  strictly positive lower bound is also a tricky issue. It is obvious that there is no such kind of  singular dissipation operators  in the time evolution equation \eqref{transport} of $S$ when $\daleth=0$.

\item at last but more importantly, due to  the effect of  heat conductivity here,   the time evolution equation $\eqref{qiyi}_2$  for  $S$  contains some strong singularities such   as:
\begin{equation}\label{troubleterms}
\mathcal{S}_1=\rho^{\frac{1-\delta-2\gamma}{2}}e^{-\frac{S}{c_v}}\Lambda(\rho,S) \quad \text{and}\quad \mathcal{S}_2=\rho^{\frac{1+\delta-2\gamma}{2}}e^{\frac{S}{c_v}\nu}\triangle \rho^{\gamma-1}. 
\end{equation}
It is worth pointing out that the appearance of $\mathcal{S}_2$   makes it difficult to show     $S\in D^4$ for $t>0$. In fact, it follows from  $\eqref{qiyi}_2$ and the regularity theory of elliptic equations that  $| S|_{D^4}$ can   be controlled by  $|\mathcal{S}_2|_{D^2}$, which seems impossible  in the current $H^3$ framework. Then such kind of  singularities  becomes some of  the main obstacles to get  the  uniform   high order regularities for $S$, thus whose analysis  becomes extremely crucial. The dissipative term $\rho^{\frac{\delta-1}{2}}e^{\frac{S}{c_v}(\nu-1)}\triangle e^{\frac{S}{c_v}}$ does not provide a substantial contribution to the treatment  of $\mathcal{S}_2$. It needs to conduct a very detailed analysis of some quantities  related with the  derivatives of $\rho$. Here it should be noted that $\mathcal{S}_1$ and $\mathcal{S}_2$ do not appear in the time evolution equation \eqref{transport} of $S$ when $\daleth=0$.
\end{enumerate}

Therefore, the following quantities will play significant roles in our analysis: $$(\rho^{\gamma-1},\ \nabla \rho^{\delta-1},\ \rho^{\delta-1} Lu, \ e^{\frac{S}{c_v}},\ \rho^{\frac{\delta-1}{2}}\triangle e^{\frac{S}{c_v}} ).$$
 Due to  this observation, we first introduce a  class of solutions called regular solutions to  the Cauchy problem \eqref{8} with \eqref{2} and \eqref{QH}-\eqref{7} as follows.
\begin{definition} \label{def21} Let $T>0$  be a finite constant. The triple $(\rho,u,S)$  is called a regular solution to the Cauchy problem \eqref{8} with \eqref{2} and \eqref{QH}-\eqref{7}   in $[0,T]\times\mathbb{R}^3$  if $(\rho,u,S)$ solves this problem in the sense of distributions and{\rm:}
\begin{enumerate}
\item
$\rho>0,\hspace{2mm}\rho^{\gamma-1}\in C([0,T]; D^1_*\cap D^3),\hspace{2mm}\nabla\rho^{\delta-1}\in C([0,T]; L^\infty \cap D^{1,3}\cap D^2)$;\\[1pt]
$\nabla \rho^{\frac{3(\delta-1)}{4}}\in C([0,T]; D^1_*),\hspace{2mm}\nabla\rho^{\frac{3(\delta-1)}{8}}\in C([0,T]; L^4)$;
 \item $ u\in C([0,T];H^3)\cap L^2([0,T];D^4),\quad  \rho^{\frac{\delta-1}{2}}\nabla u \in C([0,T];L^2),$\\[1pt]
\quad $    \rho^{\delta-1}\nabla^2 u\in  L^\infty([0,T];H^1)\cap L^2([0,T];D^2)  $;
\item $S-\bar{S}\in C([0,T];D^1_*\cap D^3),\quad  e^{\frac{S}{c_v}}-e^{\frac{\bar{S}}{c_v}}\in C([0,T];D^1_*\cap D^3)$,\\[1pt]
\quad $\rho^{\frac{\delta-1}{4}}\nabla e^{\frac{S}{c_v}} \in L^\infty([0,T];L^2)$, \quad  $\rho^{\frac{\delta-1}{2}}\nabla^2 e^{\frac{S}{c_v}} \in L^\infty([0,T];H^1)$,\\[1pt]
\quad $\rho^{(\gamma-1)(1+\nu)} e^{\frac{S}{c_v}(\nu+1)} \in L^2([0,T];D^1\cap D^4)$.
\end{enumerate}

\end{definition}
\begin{remark}\label{rz1}
First, it follows from  Definition \ref{def21}  that $\nabla \rho^{\delta-1}\in L^\infty$, which implies  that the vacuum can occur  only   in the far field.

Second, denote by $m(t)$, $\mathbb{P}(t)$, $E_k(t)$, $E_p(t)$, and  $E(t)=E_k(t)+E_p(t)$ the total mass, momentum, total kinetic energy, the potential energy, and the total energy respectively. It then can be checked easily (see Lemma \ref{lemmak})
that regular solutions defined in  Definition \ref{def21}  satisfy the conservation of   $m(t)$, $\mathbb{P}(t)$ and $E(t)$ if the initial data allows finite total mass, momentum and total energy, which  is not  clear   for strong solutions   in the case of  constant viscosities and heat conductivity  obtained in \cites{CK, huangli,  wenzhu}, cf.  \cites{dxz, rozanova,zz2}.
In this sense, the definition of regular solutions above  is consistent with the physical background of the \textbf{CNS}.
\end{remark}

The regular solutions select  $(\rho, u,S)$  in a physically reasonable way when far field vacuum appears. To find  a regular solution to  \eqref{8}, we will solve  an enlarged system consisting of (up to leading order): a  transport equation for $\rho^{\gamma-1}$,  a singular parabolic  system  for $u$, a degenerate-singular parabolic equation for $e^{\frac{S}{c_v}}$, and a symmetric hyperbolic system  for $\nabla \rho^{\delta-1}$, which makes the  original problem  tractable. The first main result  can be stated as follows.
\begin{theorem} \label{th21}Let parameters $(\gamma,\delta=\nu(\gamma-1),\alpha,\beta, \daleth)$ satisfy
\begin{equation}\label{can1}
\gamma>1,\quad 0<\delta<1,\quad \gamma+\delta \leq 2,\quad\alpha>0,\quad2\alpha+3\beta\geq 0 \quad \text{and} \quad \daleth>0.
\end{equation}
Assume that  the initial data $(\rho_0,u_0,S_0)$ satisfy
\begin{equation}\label{2.7}\begin{aligned}
&\rho_0>0,\quad\rho_0^{\gamma-1}\in  D^1_*\cap D^3,\quad \nabla\rho_0^{\delta-1}\in L^q\cap D^{1,3},\ \ \rho_0^{\frac{\delta-1}{4}}\nabla^3 \rho_0^{\delta-1}\in L^2,\\ 
& \nabla\rho_0^{\frac{3(\delta-1)}{4}}\in D^1_*,\quad \nabla\rho_0^{\frac{3(\delta-1)}{8}}\in L^4,\quad u_0\in H^3, \quad S_0-\bar{S}\in D^1_*\cap D^3,
\end{aligned}
\end{equation}
for some $q\in (3,\infty)$,
and the following  initial  compatibility conditions:
\begin{equation}\label{2.8}\begin{aligned}
&\nabla u_0=\rho_0^{\frac{1-\delta}{2}}g_1,\quad Lu_0=\rho_0^{1-\delta}g_2,\quad \nabla(\rho_0^{\delta-1}Lu_0)=\rho_0^{\frac{1-\delta}{2}}g_3,\\
& \nabla e^{\frac{S_0}{c_v}}=\rho_0^\frac{1-\delta}{4}g_4,\quad \triangle e^{\frac{S_0}{c_v}}=\rho_0^\frac{3(1-\delta)}{4}g_5,\quad \nabla(\rho_0^{\frac{\delta-1}{2}}\triangle e^{\frac{S_0}{c_v}})=\rho_0^{\frac{3(1-\delta)}{4}}g_6,
\end{aligned}\end{equation}
for some  $(g_1,g_2,g_3,g_4,g_5,g_6)\in L^2$. Then there exist a time $T_*>0$ and a unique regular solution $(\rho,u,S)$ in $[0,T_*]\times \mathbb{R}^3$  to the Cauchy problem \eqref{8} with \eqref{2} and \eqref{QH}-\eqref{7} satisfying:
\begin{equation*}
\label{2.9}
\begin{aligned}
\displaystyle
& \nabla \rho^{\delta-1}\in  C([0,T_*];L^q),\quad   \rho^{\gamma-1}_t \in C([0,T_*];H^2),  \quad \rho^{\gamma-1}_{tt} \in C([0,T_*];L^2), \\[1pt]
\displaystyle
&(\nabla\rho^{\delta-1}_t,u_t)\in C([0,T_*]; H^1), \quad   
   t^{\frac{1}{2}}u_t\in  L^\infty([0,T_*];D^2)\cap L^2([0,T_*];D^3),\\[1pt]
   \displaystyle
& \big(\nabla\rho^{\delta-1}_{tt},\nabla\rho^{\gamma-1}_{tt},\rho^{\delta-1}\nabla^2 u_t, t^{\frac{1}{2}}\rho^{\delta-1}\nabla^3 u_t,u_{tt},t^{\frac{1}{2}} \rho^{\frac{\delta-1}{2}} \nabla u_{tt}\big) \in L^2([0,T_*];L^2),\\[1pt]
\displaystyle
& t^{\frac{1}{2}}u_{tt}\in L^\infty([0,T_*];L^2)\cap L^2([0,T_*];D^1_*),\\[1pt]
\displaystyle
&\big(t^{\frac{1}{2}}\rho^{\delta-1} \nabla^4 u,\rho^{\frac{\delta-1}{2}}\nabla u_t,t^{\frac{1}{2}} \rho^{\delta-1}\nabla^2 u_t,\rho^{\frac{\delta-1}{4}}\nabla S\big) \in L^\infty([0,T_*];L^2),\\[1pt]
\displaystyle
&\big(\rho^{\frac{\delta-1}{2}}\nabla^2 S,\rho^{\frac{\delta-1}{2}}\nabla^3 S,\rho^{\frac{1-\delta}{4}}S_{t},t^{\frac{1}{2}} \rho^{\frac{\delta-1}{2}}\nabla^2 S_t, t^{\frac{1}{2}}\rho^{\frac{1-\delta}{4}}S_{tt}\big) \in L^\infty([0,T_*];L^2),\\[1pt]
\displaystyle
&  S_t\in C([0,T_*];D^1_*)\cap L^2([0,T_*];D^2),\quad   \rho^{\frac{\delta-1}{4}}S_{t}\in L^\infty([0,T_*];D^1_*),\\[1pt]
\displaystyle
&\rho^{\frac{\delta-1}{2}} S_t\in L^2([0,T_*];D^2),\quad (\rho^{\frac{1-\delta}{4}}S_{tt}, t^{\frac{1}{2}} \rho^{\frac{\delta-1}{4}} \nabla S_{tt}) \in L^2([0,T_*];L^2).
\end{aligned}\end{equation*}
Moreover, it holds that 
\begin{enumerate}

\item $(\rho,u,S)$ preserves  the  total mass, momentum and total energy provided that    $m(0)<\infty$ is assumed additionally;
 \smallskip
 \item the corresponding $(\rho,u,\theta=AR^{-1}\rho^{\gamma-1} e^{S/c_v})$ is a classical solution to   the problem   \eqref{1}-\eqref{3} with \eqref{4} and \eqref{6}-\eqref{7} in $(0,T_*]\times \mathbb{R}^3$.

\end{enumerate}

\end{theorem}

\begin{remark} \label{kuaiman}$\eqref{2.7}$-$\eqref{2.8}$ identify a class of admissible initial data that provide unique solvability to   \eqref{8} with \eqref{2} and \eqref{QH}-\eqref{7}. Such initial  data include
\begin{equation}\label{example}
\rho_0(x)=\frac{1}{(1+|x|^2)^{\varkappa}},\quad u_0(x)\in C^3_0(\mathbb{R}^3),\quad S_0=\bar{S}+f(x),
\end{equation}
for any  $f(x)\in  C^3_0(\mathbb{R}^3)$, where for some fixed $q$,  $(\varkappa,\gamma,\delta)$ need to satisfy  
\begin{equation}\label{exampledege}
\hspace{2mm}\frac{1}{4(\gamma-1)}<\varkappa<\min\left\{\frac{1-3/q}{2(1-\delta)},\frac{1}{3(1-\delta)}\right\}\quad \text{and} \quad \frac{7}{4}+\frac{\delta}{4}<\gamma+\delta\leq 2.
\end{equation}

\end{remark}

\begin{remark}\label{rxiangrong}
 The compatibility conditions \eqref{2.8} are important  for the existence of  regular  solutions $(\rho, u,S)$ obtained in Theorem \ref{th21}.
Indeed,
\begin{itemize}
          \item $\nabla  u_0=\rho_0^{\frac{1-\delta}{2}}g_1$ \big(\rm{resp.}, $\nabla e^{\frac{S_0}{c_v}}=\rho_0^\frac{1-\delta}{4}g_4$\big) plays a key role in the derivation of $\rho^{\frac{\delta-1}{2}} \nabla u \in L^\infty([0,T_*];L^2)$ \big(\rm{resp.}, $\rho^{\frac{\delta-1}{4}} \nabla S \in L^\infty([0,T_*];L^2)$\big); \\
        \item  $Lu_0=\rho_0^{1-\delta}g_2$ \big(\rm{resp.}, $\triangle e^{\frac{S_0}{c_v}}=\rho_0^\frac{3(1-\delta)}{4}g_5$\big) is crucial  in the derivation of  $u_t \in L^\infty([0,T_*];L^2)$ \big(\rm{resp.}, $\rho^{\frac{1-\delta}{4}}S_t \in L^\infty([0,T_*];L^2)$),  which will be used in the estimate for  $|u|_{D^2}$ \big(\rm{resp.}, $|S|_{D^2}$);  \\
        \item and  $\nabla(\rho_0^{\delta-1}Lu_0)=\rho_0^{\frac{1-\delta}{2}}g_3$  \big(\rm{resp.},  $\nabla(\rho_0^{\frac{\delta-1}{2}}\triangle e^{\frac{S_0}{c_v}})=\rho_0^{\frac{3(1-\delta)}{4}}g_6$)
        is used  in the derivation of $\rho^{\frac{\delta-1}{2}}\nabla u_t \in  L^\infty([0,T_*];L^2)$  \big(\rm{resp.}, $\rho^{\frac{\delta-1}{4}}\nabla S_t \in  L^\infty([0,T_*];L^2)$),  which leads to some desired   estimate for  $|u|_{D^3}$  \big(\rm{resp.}, $|S|_{D^3}$).
\end{itemize}
\end{remark}

\begin{remark}
It should be pointed out  that due to the requirement $\gamma+\delta\leq 2$ on  $(\gamma, \delta=(\gamma-1)\nu)$ in \eqref{can1},  the existence theory established in Theorem \ref{th21} applies, in particular  to   the monatomic gas, for which, $(\gamma, \nu)=(\frac{5}{3},\frac{1}{2})$. 
\end{remark}

\begin{remark}\label{zhunbei2} Note that for the regular solution $(\rho,u,S)$ obtained in Theorem \ref{th21}, $u$  stays in the inhomogeneous Sobolev space $H^3$ instead of the homogenous one $D^1_*\cap D^2$ in \cites{CK,wenzhu}  for   flows with constant viscosity and heat conductivity coefficients.

In \cite{lx4}, it is shown that for the 1-D flow with  constant viscosities and heat conductivity, the specific entropy becomes not uniformly bounded immediately after the initial time, as long as the initial density decays to zero in the far field rapidly. 
 Compared with  the conclusions obtained in Theorem \ref{th21} and the discussion in Remark \ref{kuaiman},  there is a natural question   whether the conclusion mentioned above holds for  the  degenerate  system considered here.
Due to strong degeneracy near the vacuum in $\eqref{1}_2$-$\eqref{1}_3$,   such  questions are not easy and  will be discussed in our  future work.

\end{remark}

\begin{remark}
It is worth pointing out  that in the current  $H^3$ framework,  although  the  regularities of solution $(\rho,u,S)$ obtained in Theorem \ref{th21} can not ensure that it is  a classical one to the Cauchy problem \eqref{8} with \eqref{2} and \eqref{QH}-\eqref{7} due to the appearance of the  source term $\digamma \rho^{\delta}e^{\frac{S}{c_v}(\nu+1)}\triangle \rho^{\gamma-1}$, yet the corresponding $(\rho,u,\theta=AR^{-1}\rho^{\gamma-1} e^{S/c_v})$ solves    problem   \eqref{1}-\eqref{3} with \eqref{4} and \eqref{6}-\eqref{7} classically.
\end{remark}


Moreover, in contrast to the classical theory  \cites{KA2,  mat,  wenzhu}, we show that: if the laws of  conservation of  $(m(t),\mathbb{P}(t))$ are both satisfied,  and $\mathbb{P}(0)$ is non-zero, then  $|u(t,\cdot)|_{\infty}$  will not decay to zero as $t\rightarrow \infty$.   To this end, we need the following definition.
\begin{definition}\label{d2}
Let $T>0$ be a positive time.  For   the Cauchy problem  \eqref{1}-\eqref{3} with \eqref{4} and \eqref{6}-\eqref{7},  a classical solution     $(\rho,u,\theta)$  in $(0,T]\times \mathbb{R}^3$ is said to be in $ D(T)$ if $(\rho,u,\theta)$ satisfies the following conditions:
\begin{itemize}
 \item  $m(t)$,  $\mathbb{P}(t)$ and $E_k(t)$    all belong to  $L^\infty([0,T])$;
 \smallskip

 \item The total mass is conserved, i.e.,  $\frac{d}{dt}m(t)=0$ for  any $ t\in [0,T]$;
 \smallskip
 \item  The momentum is conserved, i.e.,   $\frac{d}{dt}\mathbb{P}(t)=0$  for  any $ t\in [0,T]$.

\end{itemize}
\end{definition}

Then one has:
\begin{theorem}\label{th25}
Assume that $m(0)>0$, $|\mathbb{P}(0)|>0$, and  $(\gamma, \mu,\lambda,\kappa)$ satisfy
\begin{equation}\label{canshu-22}
\gamma >  1, \quad \mu\geq 0, \quad 2\mu+3\lambda\geq 0,\quad \kappa\geq 0.
\end{equation}  Then for   the Cauchy problem   \eqref{1}-\eqref{3} with \eqref{4} and \eqref{6}-\eqref{7}, there is no classical solution $(\rho,u,\theta)\in D(\infty)$  with
\begin{equation}\label{eq:2.15}
\limsup_{t\rightarrow \infty} |u(t,\cdot)|_{\infty}=0.
\end{equation}

\end {theorem}

An immediate consequence  of   Theorems \ref{th21}-\ref{th25} is

\begin{corollary}\label{co23}
For   the Cauchy problem \eqref{8} with \eqref{2} and \eqref{QH}-\eqref{7}, if one assumes   $0<m(0)<\infty$ and $|\mathbb{P}(0)|>0$ additionally, then  there is no global  regular  solution $(\rho,u,S)$ with the regularities  in Theorem \ref{th21}   satisfying \eqref{eq:2.15}, which is similar to the corresponding result in \cite{dxz}.

\end {corollary}

\begin{remark}
Though  the main strategy for studying the regularities of the velocity  is  similar to that of the  non-heat conductive case \cite{dxz}, yet the presence of heat conductivity here causes essential difficulties for studying the 
propagation  of  regularities of the entropy as discussed after \eqref{qiyi}. Thus new formulations and observations are needed to accomplish the required estimates. 

\end{remark}

The rest of this paper is   organised as follows.   In \S 2, we outline the  main strategy of our  proof. 
\S 3 is devoted to proving  the  local-in-time well-posedness  theory stated in  Theorem \ref{th21}, which can be  achieved in five steps: 
\begin{enumerate}

 \item construct  approximate solutions away from the vacuum for
 specially designed linearized problems with an artificial viscosity
$\sqrt{\rho^{2\delta-2}+\epsilon^2}e^{\frac{S}{c_v}\nu} Lu$
in the  momentum equations, an artificial heat conductivity
$(\rho^{2\delta-2}+\epsilon^2)^{\frac{1}{4}}e^{\frac{S}{c_v}\nu}\triangle e^{\frac{S}{c_v}} $
in the  entropy equation, and 
$\inf_{x\in \mathbb{R}^3}\rho^{\gamma-1}_0=\frac{\gamma-1}{A\gamma}\eta$
for some positive constants $\epsilon>0$ and  $\eta>0$;
\item establish the uniform estimates independent of  $(\epsilon,\eta)$ for the above solutions;

\item  then pass to the limit $\epsilon\rightarrow 0$ to recover the solution
of the corresponding  linearized problem  away from the vacuum with only physical viscosities;

\item  prove the unique solvability away from the vacuum  of the reformulated nonlinear
problem through a standard iteration process;

\item  finally take  the limit $\eta\rightarrow 0$ to recover the solution
of the reformulated nonlinear problem  with  physical viscosities and far field vacuum.
\end{enumerate}
The proofs for   Theorem \ref{th25} and Corollary \ref{co23} 
 are given  in \S 4.
  Finally,  we  list some basic facts which have been  used frequently in this paper in the appendix.


\section{Main strategy}\label{s2}

This section is devoted to  sketching  the  main strategies of our analysis. 
The first  key idea  for  proving  our well-posedness theory   is to  study the  vacuum problem  in terms of  $(\rho,u,S)$, i.e., the system \eqref{8},  instead of $(\rho, u,\theta)$, which makes it possible to  compare  the orders of the degeneracy of the time evolution and the spatial  dissipation  near the vacuum in terms of  the powers of $\rho$. However, note the time evolution equation $\eqref{8}_3$ of $S$,
the double degenerate structure with the strong singular source terms related to $\triangle \rho^{\gamma-1}$  makes the vacuum problem we considered intricate such that it is still   hard to be solved. Thus, in order to investigate the propagation  mechanism for $S's$  high order regularities,  some elaborate analysis of the intrinsic degenerate-singular structures of \eqref{8} has been introduced here.

\subsection{An enlarged   reformulation} 
We first reformulate   \eqref{8}
into an enlarged  tractable system, and such a new reformulation includes several  crucial points, which will be   explained as follows.

\subsubsection{The choice  of the degenerate weight for the energy space of $S$}
 On the one hand,  since the density $\rho$ of the fluid vanishes in the far field,   it is necessary to find  proper weights to control the behaviors of $S$. Inspired by  the competition of different terms in the system \eqref{8} for weights in  weighted energy estimates, i.e.,
$$
\rho^{\delta}\nabla e^{\frac{S}{c_v}\input{}}\cdot Q(u) \quad \text{and} \quad \rho^{\delta}e^{\frac{S}{c_v}(\nu+1)}\triangle \rho^{\gamma-1},$$
we choose the degenerate weight $\rho^{\frac{1-\delta}{2}}$ in front of  $S_t+u\cdot \nabla S$ in the equation for $S$,
\begin{equation}\label{qiyi2s}
\begin{split}
\displaystyle
& \quad \underbrace{\rho^{\frac{1-\delta}{2}}(S_t+u\cdot \nabla S)}_{\circledast}-\underbrace{\digamma A^{-1}\rho^{\frac{\delta-1}{2}}e^{\frac{S}{c_v}(\nu-1)}\triangle e^{\frac{S}{c_v}}}_{\Box}\\
=&A^{\nu-1}R^{1-\nu}\rho^{\frac{1+\delta-2\gamma}{2}}e^{\frac{S}{c_v}(\nu-1)}H(u)\\
&+\underbrace{\digamma A^{-1}\rho^{\frac{1+\delta-2\gamma}{2}}e^{\frac{S}{c_v}\nu}\triangle \rho^{\gamma-1}}_{\backsim}+A^{-1}\rho^{\frac{1-\delta-2\gamma}{2}}e^{-\frac{S}{c_v}}\Lambda(\rho,S),
\end{split}
\end{equation}
which turns out to be the desired choice as will be verified  later. The main ideas on   such a weight  have been discussed  in the paragraph below \eqref{qiyi}.

\subsubsection{The   reformulation of the source term related to $\triangle \rho^{\gamma-1}$}
On the other hand, it is observed that there exist three kinds of singular  source terms related to the derivatives of $\rho$ in \eqref{qiyi}, i.e., $\nabla\rho^{\delta-1}$, $\rho^{\frac{1+\delta-2\gamma}{2}}\triangle \rho^{\gamma-1}$ and $\rho^{\frac{1-\delta-2\gamma}{2}}\Lambda(\rho,S)$.
Denote $\psi=\frac{\delta}{\delta-1}\nabla\rho^{\delta-1}$.   Then by  the  equation $\eqref{8}_1$, $\psi$ solves the following system,
\begin{equation}\label{qiyi3s}
\psi_t+\nabla (u\cdot \psi)+(\delta-1)\psi\text{div} u +\delta\rho^{\delta-1}\nabla \text{div} u=0.
\end{equation}
We expect that the desired information on $\rho^{\frac{1+\delta-2\gamma}{2}}\triangle \rho^{\gamma-1}$ can be provided  by  the estimates for $\psi$ that are given by \eqref{qiyi3s} and the singular weighted estimates for $u$. Then for this purpose,  we rewrite $\rho^{\frac{1+\delta-2\gamma}{2}}\triangle \rho^{\gamma-1}$
 as
\begin{equation}\label{qiyi4s}
\rho^{\frac{1+\delta-2\gamma}{2}}\triangle \rho^{\gamma-1}=\frac{\gamma-1}{\delta-1}\rho^{\frac{1-\delta}{2}}\text{div}\nabla \rho^{\delta-1}+\frac{(\gamma-1)(\gamma-\delta)}{(\delta-1)^2}\rho^{\frac{3}{2}(1-\delta)}\nabla \rho^{\delta-1}\cdot \nabla \rho^{\delta-1}.
\end{equation}
Similarly, the derivatives of $\rho$ in the singular source term $\rho^{\frac{1-\delta-2\gamma}{2}}e^{-\frac{S}{c_v}}\Lambda(\rho,S)$ which appears in \eqref{qiyi2s} can also be rewritten  in terms of $\psi$ and the positive powers of $\rho$.

\subsubsection{An  enlarged Cauchy   problem}
 Now, based on \eqref{qiyi2s}-\eqref{qiyi4s}, one can  reformulate   \eqref{8}
into an enlarged   system. Set  $\delta=(\gamma-1)\nu$. In terms of 
\begin{equation}\label{2.1}\begin{aligned}
&\phi=\frac{A\gamma}{\gamma-1}\rho^{\gamma-1},\quad u,\quad  l=e^{\frac{S}{c_v}},\quad \psi=\frac{\delta}{\delta-1}\nabla\rho^{\delta-1},\quad n=\rho^{2-\delta-\gamma},\\[4pt]
\end{aligned}\end{equation}
the problem  \eqref{8} with \eqref{2} and \eqref{QH}-\eqref{7} implies that 
\begin{equation}\label{2.3}\left\{\begin{aligned}
\displaystyle
&\  \phi_t+u\cdot\nabla\phi+(\gamma-1)\phi \text{div}u=0,\\[2pt]
\displaystyle
&\  u_t+u\cdot \nabla u+a_1\phi\nabla l+l\nabla\phi+a_2l^\nu\phi^{2\iota} Lu\\
\displaystyle
=&a_2\phi^{2\iota}\nabla l^\nu\cdot Q(u)+a_3l^\nu\psi  \cdot  Q(u),\\[2pt]
\displaystyle
&\  \phi^{-\iota}(l_t+u\cdot\nabla l)-a_4\phi^{\iota}l^\nu \triangle l\\[2pt]
= & a_5l^\nu n\phi^{3\iota}H(u)+a_6l^{\nu+1} \phi^{-\iota}\text{div} \psi+\Theta(\phi,l,\psi),\\[2pt]
\displaystyle
&\ \psi_t+\sum_{k=1}^3 A_k(u) \partial_k\psi+B(u)\psi+\delta a\phi^{2\iota}\nabla \text{div} u=0,
\end{aligned}\right.\end{equation}
where
\begin{equation}\label{2.2nnn}
\begin{split}
\Theta(\phi,l,\psi)=&a_{7}l^{\nu+1} \phi^{-3\iota}\psi\cdot \psi+a_8l^\nu  \phi^{-\iota}\nabla l\cdot   \psi+a_9l^{\nu-1} \phi^{\iota}\nabla l\cdot \nabla l,
\end{split}
\end{equation}
and 
\begin{equation}\label{2.2}
\begin{split}
a_1=&\frac{\gamma-1}{\gamma},\quad a_2=a\Big(\frac{A}{R}\Big)^\nu,\quad a_3=\Big(\fr{A}{R}\Big)^\nu,\quad a_4=\digamma \frac{a}{Ac_v},\\
a_5=&\fr{A^{\nu-1}a^{2}(\gamma-1)}{R^\nu},\quad a_6=\digamma \frac{(\gamma-1)}{A  c_v \delta},\quad a_7=\digamma \frac{ \gamma(\gamma-1)}{a Ac_v \delta^2},\\
a_8=&2\digamma \frac{1+\nu}{A  c_v \nu},\quad a_9=\digamma \frac{a \nu}{Ac_v},\quad 
\iota= \fr{\delta-1}{2(\gamma-1)},\quad a=\Big(\fr{A\gamma}{\gamma-1}\Big)^{\fr{1-\delta}{\gamma-1}},
\end{split}
\end{equation}
 $A_k(u)=(a^k_{ij})_{3\times 3}$ for  $i$, $j$, $k=1$, $2$, $3$,
are symmetric  with
$$a^k_{ij}=u^{(k)}\quad \text{for}\ i=j;\quad \text{otherwise}\  a^k_{ij}=0, $$
 and $B(u)=(\nabla u)^\top+(\delta-1)\text{div}u\mathbb{I}_3$.

The initial data for \eqref{2.3} are given by
\begin{equation}\label{2.4}
\begin{split}
&(\phi,u,l,\psi)|_{t=0}=(\phi_0,u_0,l_0,\psi_0)\\
=&\Big(\fr{A\gamma}{\gamma-1}\rho_0^{\gamma-1}(x),u_0(x),e^{S_0(x)/c_v},\fr{\delta}{\delta-1}\nabla\rho_0^{\delta-1}(x)\Big)\quad  \text{for} \quad x\in\mathbb{R}^3.
\end{split}
\end{equation}
$(\phi,u,l,\psi)$ are required  to satisfy the  far field behavior: for some constant  $\bar{l}>0$,
\begin{equation}\label{2.5}
(\phi,u,l,\psi)\rightarrow (0,0,\bar{l},0) \hspace{2mm} \text{as} \hspace{2mm}|x|\rightarrow \infty \hspace{2mm} \text{for} \quad t\geq 0.
\end{equation}

Note that  the enlarged system \eqref{2.3}   consists of (up to leading order) 
\begin{itemize}
 \item  one {\it scalar transport} equation $\eqref{2.3}_1$ for $\phi$;
 \smallskip
 \item one  {\it singular parabolic}  system  $\eqref{2.3}_2$ for the velocity $u$;
 \smallskip
 \item  one {\it degenerate (time evolution operator)-singular (elliptic operator) parabolic}  equation   $\eqref{2.3}_3$ with several singular source terms  for $l$;
 \smallskip
 \item  one {\it symmetric hyperbolic} system  $\eqref{2.3}_4$ but with several singular source terms for $\psi$,
\end{itemize}
such a  structure will enable us to  establish  the following main  theorem.

\begin{theorem}\label{3.1} Let $\ef{can1}$ hold. Assume that  the initial data $(\phi_0,u_0,l_0,\psi_0)$ satisfy:
\begin{equation}\label{a}\begin{aligned}
&\phi_0>0,\quad \phi_0\in D^1_*\cap D^3,\quad \nabla\phi_0^{\frac{3}{2}\iota}\in D^1_*,\quad \nabla\phi_0^{\frac{3}{4}\iota}\in L^4, \quad u_0\in H^3,\\[3pt]
&l_0-\bar{l}\in D^1_*\cap D^3,\quad \inf_{x\in \mathbb{R}^3} l_0>0, \quad \psi_0\in L^q \cap D^{1,3}, \ \ \phi_0^{\frac{1}{2}\iota}\nabla^2 \psi_0\in L^2,
\end{aligned}\end{equation}
for some $q\in (3,\infty)$, and the following  compatibility conditions:
\begin{equation}
\label{2.8*}\begin{aligned}
&\nabla u_0=\phi_0^{-\iota}g_1,\quad Lu_0=\phi_0^{-2\iota}g_2,\quad 
\nabla(\phi_0^{2\iota}Lu_0)=\phi_0^{-\iota}g_3,\\[3pt]
& \nabla l_0=\phi_0^{-\frac{\iota}{2}}g_4,\quad \triangle l_0=\phi_0^{-\frac{3}{2}\iota}g_5, \quad \nabla(\phi_0^{\iota} \triangle l_0)=\phi_0^{-\frac{3}{2}\iota}g_6,
\end{aligned}
\end{equation}
for some $(g_1,g_2,g_3,g_4,g_5,g_6)\in L^2$. Then there exist a time $T_*>0$ and a unique strong solution $(\phi,u,l,\psi=\fr{a\delta}{\delta-1}\nabla\phi^{2\iota})$  in $[0,T_*]\times \mathbb{R}^3$   to the Cauchy problem \ef{2.3}-\ef{2.5}, such that $\phi(t,x)>0$ in $[0,T_*]\times \mathbb{R}^3$, $\inf\limits_{(t,x)\in [0,T_*]\times \mathbb{R}^3} l>0$ and 
\begin{equation}
\begin{split}\label{b}
\displaystyle
&\phi\in C([0,T_*];D^1_*\cap D^3),\quad \nabla \phi^{\frac{3}{2}\iota}\in C([0,T_*];D^1_*),\quad \nabla \phi^{\frac{3}{4}\iota}\in C([0,T_*];L^4),\\[4pt]
\displaystyle
&\psi\in  C([0,T_*];L^q \cap D^{1,3}\cap D^2),\quad \phi_t \in C([0,T_*];H^2), \quad \psi_t\in C([0,T_*]; H^1),   \\[4pt]
\displaystyle
&\phi_{tt} \in C([0,T_*];L^2)\cap L^2([0,T_*];D^1_*), \quad  u\in C([0,T_*];H^3)\cap L^2([0,T_*];H^4), \\[4pt]
\displaystyle
&   u_t\in  C([0,T_*];H^1),\quad  \phi^{2\iota}\nabla^2 u\in L^\infty([0,T_*];H^1)\cap L^2([0,T_*];D^2),\\[4pt]
\displaystyle
&(\phi^\iota\nabla u, t^{\fr{1}{2}}\phi^{2\iota} \nabla^4 u,\phi^\iota\nabla u_t,t^{\frac{1}{2}} \phi^{2\iota}\nabla^2 u_t)\in L^\infty([0,T_*];L^2),\\[4pt]
\displaystyle
& (\psi_{tt},\phi^{2\iota}\nabla^2u_t,t^{\frac{1}{2}}\phi^{2\iota}\nabla^3 u_t,  u_{tt},t^{\frac{1}{2}} \phi^{\iota} \nabla u_{tt})\in L^2([0,T_*];L^2),\\[4pt]
\displaystyle
& t^{\frac{1}{2}}u_{tt}\in L^\infty([0,T_*];L^2)\cap L^2([0,T_*];D^1_*),\quad l-\bar{l}\in  C([0,T_*];D^1_* \cap D^3),\\[4pt]
\displaystyle
&(\phi^{\frac{\iota}{2}}\nabla l,\phi^{\iota}\nabla^2 l,\phi^{\iota}\nabla^3 l,\phi^{-\frac{\iota}{2}}l_{t}, t^{\frac{1}{2}} \phi^{\iota}\nabla^2 l_t,t^{\frac{1}{2}}\phi^{-\frac{\iota}{2}}l_{tt}) \in L^\infty([0,T_*];L^2),\\[4pt] 
\displaystyle
&  l_t\in C([0,T_*];D^1_*)\cap L^2([0,T_*];D^2),\quad   \phi^{\frac{\iota}{2}}l_{t}\in L^\infty([0,T_*];D^1_*),\\[4pt]
\displaystyle
&\phi^{\iota} l_t\in L^2([0,T_*];D^2),\quad    (\phi^{-\frac{\iota}{2}}l_{tt},t^{\frac{1}{2}} \phi^{\frac{\iota}{2}} \nabla l_{tt})\in L^2([0,T_*];L^2).
\end{split}
\end{equation}

\end{theorem}
\begin{remark}\label{strongsolution}
In Theorem \ref{3.1}, $(\phi,u,l,\psi=\fr{a\delta}{\delta-1}\nabla\phi^{2\iota})$  in $[0,T_*]\times \mathbb{R}^3$ is called a strong solution   to the Cauchy problem \ef{2.3}-\ef{2.5} if it satisfies \ef{2.3}-\ef{2.5} in the sense of distributions and  satisfies the  equations \eqref{2.3}-\ef{2.2} a.e. $(t,x)\in (0,T_*]\times \mathbb{R}^3$.
\end{remark}

\subsection{Energy estimates for the degenerate-singular structure}
Now we sketch the  main strategy on   how to obtain  closed energy estimates based on the degenerate-singular structure  \eqref{2.3} described above.

Note first that $\phi$ satisfies a   {\it scalar transport} equation $\eqref{2.3}_1$. Then  $\phi$  can be estimated  by  classical arguments. Second,   $u$  is governed  by a   {\it singular parabolic}   system  $\eqref{2.3}_2$. The argument on how to establish the a prior estimates in $H^3$ space for such a structure  has been introduced   in \cite{zz2} for degenerate isentropic CNS.

Next we show how to treat    $l$.
Note that    $l$  can be controlled by the following {\it degenerate-singular  parabolic}  equation:
\begin{equation*}
    \begin{split}
      &  \underbrace{\phi^{-\iota}(l_t+u\cdot\nabla l)}_{\circledast}-\underbrace{a_4\phi^{\iota}l^\nu \triangle l}_{\Box}
=  \underbrace{a_5l^\nu n\phi^{3\iota}H(u)+\Theta(\phi,l,\psi)}_{\backsim_1}+\underbrace{a_6l^{\nu+1} \phi^{-\iota}\text{div} \psi}_{\backsim_2},
    \end{split}
\end{equation*}
where  $\backsim_i$ ($i=1,2$) denotes the source term   with $i$-th order  derivatives of $\rho$, which may be   singular near the vacuum. It follows  from $\eqref{2.3}_3$ that
\begin{equation*}\begin{split}
-a_4\triangle (\phi^{\iota} (l-\bar{l}))
=&l^{-\nu}\mathcal{E}(\phi,u,l,\psi)-a_4F(\nabla \phi^{\iota},l-\bar{l})=V,
\end{split}
\end{equation*}
where
\begin{equation*}\begin{split}
\mathcal{E}(\phi,u,l,\psi)=&-\phi^{-\iota}( l_t+u\cdot\nabla l)
+a_5l^\nu n\phi^{3\iota}H(u)
+a_6l^{\nu+1} \phi^{-\iota} \text{div} \psi\\[4pt]
&+\Theta(\phi,l,\psi),\\[4pt]
F(\nabla \phi^{\iota},l-\bar{l})=&( l-\bar{l})\triangle \phi^{\iota}+2\nabla \phi^{\iota}\cdot\nabla l.
\end{split}
\end{equation*}
Then the standard elliptic regularity theory yields 
\begin{equation}\label{lll12}
\begin{split}
|\phi^{\iota}( l-\bar{l})|_{D^2}\leq  C|V|_{2}\quad \text{and}\quad |\phi^{\iota}( l-\bar{l})|_{D^3}\leq C|V|_{D^1},
\end{split}
\end{equation}
for some constant  $C>0$ independent of the lower bound of $\phi$ provided that
\begin{equation}\label{tuoyuanwuqiong}
\phi^{\iota}( l-\bar{l})\rightarrow 0\qquad \text{as} \qquad |x| \rightarrow \infty.
\end{equation}

Based on \eqref{lll12}, one has 
\begin{equation}\label{lll13}
\begin{split}
|\phi^{\iota}\nabla^2l|_{2}
\leq & C(|V|_{2}+|\nabla\phi^{\iota}  |_\infty|\nabla l|_2+|l-\bar{l}|_\infty | \nabla^2 \phi^{\iota} |_2),\\[4pt]
|\phi^{\iota}\nabla^3l|_2\leq & 
 C(|V|_{D^1}+|\nabla\phi^{\iota}  |_\infty|\nabla^2 l|_2+|\nabla^2\phi^{\iota}  |_2|\nabla l|_\infty+|\nabla^3 \phi^{\iota} |_2|l -\bar{l} |_\infty).
\end{split}
\end{equation}
The above  elliptic estimates \eqref{lll12} and \eqref{lll13}  are first established for non-vacuum approximations, where $\phi^{\iota}( l-\bar{l})$ vanishes at infinity due to $\inf_{x\in \mathbb{R}^3}\phi>0$, and the estimates are  independent of $\phi$’s lower bound.  Such estimates still hold for the flow with far-field vacuum by the weak lower semi-continuity of norms. This argument does not require us to verify whether \eqref{tuoyuanwuqiong} holds when $\phi\rightarrow 0$ as $|x|\rightarrow \infty$.

It should be noted here that this analysis does not yield    $l\in L^2([0,T_*];D^4)$ due to  the appearance of the term $a_6l^{\nu+1} \phi^{-\iota} \text{div} \psi$ in $\ef{2.3}_3$ or $\mathcal{E}$. $|l|_{D^4}$ can be  controlled by  $|l^{\nu+1} \phi^{-\iota} \text{div} \psi|_{D^2}$, which seems  impossible in the current $H^3$ framework. What we can show is that 
$\theta^{\nu+1}\in L^2([0,T_*];D^4)$, which is  enough to show that the solution  obtained here is just a classical one to the original system \eqref{1}.  

In order to close the estimates for $l$,  one still needs to deal with several singular source terms such as   
$ \phi^{-\iota} \text{div} \psi$,  $n\phi^{3\iota}H(u)$ and $\Theta(\phi,l,\psi)$. Note that  $\ef{2.3}_4$
implies that the subtle  term $\psi$  solves a  symmetric hyperbolic  system with a  singular source term $\delta a\phi^{2\iota}\nabla \text{div} u$, then $\psi$  can be controlled by this structure and   the corresponding singular weighted estimates for $u$. However, due to the high non-linearity and singularity,   only the estimates for $\psi$ itself  is far from enough. In order to  deal with these terms, some more delicate estimates are introduced here,  i.e.,
\begin{equation}
\nabla\phi^{\frac{3}{2}\iota}\in D^1_*,\quad \nabla\phi^{\frac{3}{4}\iota}\in L^4, \quad  \phi^{\frac{1}{2}\iota}\nabla^2 \psi\in L^2,
\end{equation}
which, along with some singularly weighted interpolation inequalities such as
\begin{equation*} 
|\phi^{\fr{3}{2}\iota}\nabla u_t|_3\leq C|\phi^{\iota}\nabla u_t|_2^{\fr{1}{2}}|\phi^{2\iota}\nabla u_t|_6^{\fr{1}{2}},
\end{equation*}
 provide effective ways to distribute the powers of the  singular weights  reasonably. Then we can finally establish the desired uniform high order regularity for $S$. 

\section{Local-in-time well-posedness with far field vacuum}
In this section,  the proofs for    Theorems \ref{th21} and \ref{3.1}  will be given.
\subsection{Linearization away from the vacuum  with artificial dissipations}

Let $T$ be some positive time. To solve the nonlinear problem  $\ef{2.3}$-$\ef{2.5}$, we start with the following  linearized problem for  $(\phi^{\epsilon,\eta}, u^{\epsilon,\eta},  l^{\epsilon,\eta}, h^{\epsilon,\eta})$  in $[0,T]\times \mathbb{R}^3$:
\begin{equation}\label{ln}\left\{\begin{aligned}
&\phi^{\epsilon,\eta}_t+v\cdot\nabla\phi^{\epsilon,\eta}+(\gamma-1)\phi^{\epsilon,\eta} \text{div}v=0,\\
&u^{\epsilon,\eta}_t+v\cdot \nabla v+a_1\phi^{\epsilon,\eta}\nabla l^{\epsilon,\eta}+l^{\epsilon,\eta}\nabla\phi^{\epsilon,\eta}+a_2l^\nu\sqrt {(h^{\epsilon,\eta})^2  +\epsilon^2} Lu^{\epsilon,\eta}\\
=&a_2g\nabla l^\nu\cdot Q(v)+a_3l^\nu \psi^{\epsilon,\eta} \cdot Q(v),\\
& ((h^{\epsilon,\eta})^2+\epsilon^2)^{-\frac{1}{4}}( l^{\epsilon,\eta}_t+v\cdot\nabla l^{\epsilon,\eta})-a_4 w^\nu ((h^{\epsilon,\eta})^2+\epsilon^2)^{\frac{1}{4}} \triangle l^{\epsilon,\eta}\\
= & a_5 w^\nu \tilde{n}^{\epsilon,\eta} g^{\frac{3}{2}}H(v)+a_6w^{\nu+1}((h^{\epsilon,\eta})^2+\epsilon^2)^{-\frac{1}{4}} \text{div} \psi^{\epsilon,\eta}+\Pi(l^{\epsilon,\eta},h^{\epsilon,\eta},w,g),\\[2pt]
&h^{\epsilon,\eta}_t+v\cdot \nabla h^{\epsilon,\eta}+(\delta-1)g\text{div}v=0,\\
&(\phi^{\epsilon,\eta},u^{\epsilon,\eta},l^{\epsilon,\eta},h^{\epsilon,\eta})|_{t=0}=(\phi^\eta_0,u^\eta_0,l^\eta_0,h^\eta_0)\\[1pt]
=&(\phi_0+\eta,u_0,l_0,(\phi_0+\eta)^{2\iota})\quad  \text{for} \quad x\in\mathbb{R}^3,\\
&(\phi^{\epsilon,\eta},u^{\epsilon,\eta},l^{\epsilon,\eta},h^{\epsilon,\eta})\rightarrow (\eta,0,\bar{l},\eta^{2\iota}) \quad \text{as} \hspace{2mm}|x|\rightarrow \infty \quad \rm for\quad t\geq 0,
\end{aligned}\right.\end{equation}
where $\epsilon\in (0,1]$ and $\eta \in (0,1]$ are any given  positive constants, 
\begin{equation}\label{2.11}
\begin{split}
& \psi^{\epsilon,\eta}=\fr{a\delta}{\delta-1}\nabla h^{\epsilon,\eta},\quad \tilde{n}^{\epsilon,\eta}=a^b ((h^{\epsilon,\eta})^2+\epsilon^2)^{\frac{b}{2}},\quad b=\frac{2-\delta-\gamma}{\delta-1}\leq 0,\\
&\Pi(l^{\epsilon,\eta},h^{\epsilon,\eta},w,g)=a_7w^{\nu+1} ((h^{\epsilon,\eta})^2+\epsilon^2)^{-\frac{3}{4}}\psi^{\epsilon,\eta}\cdot \psi^{\epsilon,\eta}\\
&\qquad \qquad \qquad \qquad +a_8w^\nu((h^{\epsilon,\eta})^2+\epsilon^2)^{-\frac{1}{4}} \nabla l^{\epsilon,\eta}\cdot   \psi^{\epsilon,\eta}  +a_9w^{\nu-1} \sqrt{g}\nabla w\cdot \nabla w,
 \end{split}
\end{equation}
$v=(v^{(1)},v^{(2)},v^{(3)})^\top \in\mathbb{R}^3$ is a given vector, $g$ and $w$ are given  real functions satisfying:
\begin{align}
&(v(0,x),g(0,x),w(0,x))=(u_0(x),h^\eta_0(x)=(\phi^\eta_0)^{2\iota}(x),l_0(x)),\nonumber\\
& \inf_{(t,x)\in [0,T]\times \mathbb{R}^3} g>0,\ 
g\in L^\infty\cap C([0,T]\times \mathbb{R}^3),\ \nabla g\in C([0,T];L^q\cap D^{1,3}\cap D^2),\nonumber\\[3pt]
&\nabla g^{\frac{3}{4}}\in C([0,T]; D^1_*),\hspace{2mm}\nabla g^{\frac{3}{8}}\in C([0,T]; L^4),\quad  g_t\in C([0,T];H^2),\nonumber\\
&(\nabla g_{tt},v_{tt},w_{tt})\in L^2([0,T];L^2),\quad  v\in C([0,T];H^3)\cap L^2([0,T];H^4),\nonumber\\
& t^{\fr{1}{2}}v\in L^\infty([0,T];D^4),\quad v_t\in C([0,T];H^1)\cap L^2([0,T];D^2),  \label{4.1*}\\
&   t^{\fr{1}{2}}v_t\in L^\infty([0,T];D^2)\cap L^2([0,T];D^3),\quad 
 w-\bar{l}\in  C([0,T];D^1_* \cap D^3), \nonumber\\
&t^{\fr{1}{2}}(v_{tt},w_{tt})\in L^\infty([0,T];L^2)\cap L^2([0,T];D^1_*), \quad \inf_{(t,x)\in [0,T]\times \mathbb{R}^3} w>0,\nonumber\\ 
&w_t\in C([0,T];D^1_*)\cap L^2([0,T];D^2),\quad   t^{ \frac{1}{2} } w_t\in L^\infty([0,T];D^2).\nonumber
\end{align}
We note that,  for any fixed $\eta\in(0,1]$,  
$$(\phi^\eta_0,u^\eta_0,l^\eta_0,h^\eta_0)=(\phi_0+\eta,u_0,l_0,(\phi_0+\eta)^{2\iota}),$$ with $(\phi_0,u_0,l_0,h_0)$  satisfying  
	\eqref{a}-\eqref{2.8*}  and $\psi_0=\frac{a\delta}{\delta-1}\nabla \phi^{2\iota}_0$, there exists a constant $c_0>2$ independent of $(\epsilon,\eta)$ such that
\begin{equation}\label{2.14}\begin{aligned}
2+\eta+\bar{l}+\|\phi^\eta_0-\eta\|_{D^1_*\cap D^3}+\|u^\eta_0\|_{3}+\|\nabla h^\eta_0\|_{L^q\cap D^{1,3}\cap D^2}&\\
+|(h^\eta_0)^{\frac{1}{4}}\nabla^3h^\eta_0|_2
+|\nabla (h^\eta_0)^{\frac{3}{4}}|_{D^1_*}
+|\nabla (h^\eta_0)^{\frac{3}{8}}|_{4}+|(h^\eta_0)^{-1}|_\infty+|g^\eta_1|_2+|g^\eta_2|_2&\\
+|g^\eta_3|_2
+|g^\eta_4|_2+|g^\eta_5|_2
+|g^\eta_6|_2+\|l^\eta_0-\bar{l}\|_{D^1_*\cap D^3}+|(l^\eta_0)^{-1}|_\infty&\leq  c_0,
\end{aligned}\end{equation}
where 
\begin{equation*}
\begin{split}
&g^\eta_1=(\phi^\eta_0)^{\iota}\nabla u^\eta_0,\quad g^\eta_2=(\phi^\eta_0)^{2\iota}Lu^\eta_0,\quad 
g^\eta_3=(\phi^\eta_0)^{\iota}\nabla((\phi^\eta_0)^{2\iota}Lu^\eta_0),\\ &g_4^\eta=(\phi^\eta_0)^{\frac{\iota}{2}}\nabla l^\eta_0,\quad g_5^\eta=(\phi^\eta_0)^{\frac{3}{2}\iota}\triangle l^\eta_0,\quad g_6^\eta=(\phi^\eta_0)^{\frac{3}{2}\iota}\nabla ((\phi^\eta_0)^{\iota}\triangle l^\eta_0).
\end{split}
\end{equation*}

\begin{remark}\label{etawuguan}
From the definition of $\phi_0^\eta,u_0^\eta,l_0^\eta,h_0^\eta,g_1^\eta,g_2^\eta,g_3^\eta,g_4^\eta,g_5^\eta,g_6^\eta$, $h_0=\phi^{2\iota}_0$, and  \eqref{a}-\eqref{2.8*}, one can deduce that $c_0$ is independent of $(\epsilon,\eta)$ as follows.
\begin{align*}
|g^\eta_1|_{2}&=\big|\frac{\phi_0^{-\iota}}{(\phi_0+\eta)^{-\iota}}g_1\big|_{2}\leq |g_1|_{2},\ \ |g^\eta_2|_{2}=\big|\frac{\phi_0^{-2\iota}}{(\phi_0+\eta)^{-2\iota}}g_2\big|_{2}\leq |g_2|_{2},\\
|g^\eta_3|_{2}
&=\big|\frac{\phi_0^{-3\iota}}{(\phi_0+\eta)^{-3\iota}}(g_3-\frac{\eta\nabla\phi_0^{2\iota}}{\phi_0+\eta}\phi_0^\iota Lu_0)\big|_{2}\\
&\leq C(|g_3|_{2}+|{\nabla \phi_0^{2\iota}}|_{\infty}|\phi_0|^{-\iota}_{\infty}|g_2|_{2}),\\
|g^\eta_4|_{L^2}&=\big|\frac{\phi_0^{-\frac{\iota}{2}}}{(\phi_0+\eta)^{-\frac{\iota}{2}}}g_4\big|_{2}\leq |g_4|_{2},\ \ 
|g^\eta_5|_{L^2}=\big|\frac{\phi_0^{-\frac{3\iota}{2}}}{(\phi_0+\eta)^{-\frac{3\iota}{2}}}g_5\big|_{2}\leq |g_5|_{2},\\
|g^\eta_6|_{2}
&=\big|\frac{\phi_0^{-\frac{5\iota}{2}}}{(\phi_0+\eta)^{-\frac{5\iota}{2}}}(g_6-\frac{\eta\nabla\phi_0^{\iota}}{\phi_0+\eta}\phi_0^\frac{3\iota}{2}\triangle l_0)\big|_{2}\\
&\leq C(|g_6|_{2}+|{\nabla \phi_0^{\iota}}|_{L^\infty}|g_5|_{2}),\\
\notag
|\nabla h_0^\eta|_{q}&=|\frac{\phi_0^{-2\iota+1}}{(\phi_0+\eta)^{-2\iota+1}}\nabla h_0|_{q}
\leq |\nabla h_0|_{q}, \\
|\nabla^2 h_0^\eta|_3&\leq |\nabla^2 h_0|_{3}+C|h_0^{-1}\nabla h_0\otimes\nabla h_0|_3
\leq |\nabla^2 h_0|_{3}+C|h_0^{-\frac{1}{2}}|_\infty|\nabla h_0^{\frac{3}{4}}|_6^2, \\
|\nabla (h_0^\eta)^{\frac{3}{4}}|_{D^1_*}&=|\frac{\phi_0^{-\frac{3}{2}\iota+1}}{(\phi_0+\eta)^{-\frac{3}{2}\iota+1}}\nabla h_0^{\frac{3}{4}}|_{D^1_*}\leq |\nabla h_0^{\frac{3}{4}}|_{D^1_*}+C|\nabla h_0^{\frac{3}{8}}|^2_{4},\\
|\nabla (h_0^\eta)^{\frac{3}{8}}|_4&=|\frac{\phi_0^{-\frac{3}{4}\iota+1}}{(\phi_0+\eta)^{-\frac{3}{4}\iota+1}}\nabla h_0^{\frac{3}{8}}|_4\leq |\nabla h_0^{\frac{3}{8}}|_4,\\
|(h_0^\eta)^{\frac{1}{4}}\nabla^3 h_0^\eta|_2&\leq |h_0^{\frac{1}{4}}\nabla^3h_0|_2+C(|h_0^{-\frac{1}{2}}|_\infty|\nabla h_0^{\frac{3}{4}}|_6|\nabla^2h_0|_3+|h_0^{-\frac{1}{2}}|_\infty|\nabla h_0|_\infty|\nabla h_0^{\frac{3}{8}}|_4^2),
\end{align*}
\end{remark}
where $C>0$ is a constant  independent of $(\epsilon,\eta)$ and the initial data.

First, we  establish the local   well-posedness of strong solutions to  problem \eqref{ln}.
\begin{lemma}\label{ls}
Let $\eta, \epsilon\in (0,1]$, $\ef{can1}$ hold,  the initial data $(\phi_0,u_0,l_0,h_0)$ satisfy 
	\eqref{a}-\eqref{2.8*}, and \eqref{2.14} hold. 
Then  there exist some $T^*\in(0,T]$ and constants 
$c_0\leq c_1\leq c_2\leq c_3\leq c_4\leq c_5$ depending  only on $(A, R, c_v, \alpha, \beta, \gamma, \delta, \daleth, c_0)$, such that the following conclusions hold: if $(v,g,w)$ satisfy \eqref{4.1*} and 
\begin{align}\label{2.16}
\sup_{0\leq t\leq T^*}(\|\nabla g\|^2_{L^\infty\cap L^q\cap D^{1,3}\cap D^2}+|\nabla g^{\frac{3}{4}}|^2_{ D^1_*}+|\nabla g^{\frac{3}{8}}|^2_4)(t)\leq & c_1^2,\nonumber\\
  \inf_{[0,T^*]\times \mathbb{R}^3} w(t,x)\geq c^{-1}_1,\quad \inf_{[0,T^*]\times \mathbb{R}^3} g(t,x)\geq & c^{-1}_1,\nonumber\\
    \sup_{0\leq t\leq T^*}(|w|^2_\infty+|v|^2_\infty+|\sqrt{g}\nabla v|^2_2+\|v\|^2_{1})(t)+\int^{T^*}_0(|v|^2_{D^2}+|v_t|^2_2) \text{d}t \leq & c_1^2,\nonumber\\
\sup_{0\leq t\leq T^*}|g^{\frac{1}{4}}\nabla w(t)|^2_2+\int^{T^*}_0(|g^{-\frac{1}{4}}w_t|^2_2+|\sqrt{g}\nabla^2 w|^2_{2}) \text{d}t \leq & c_2^2,\nonumber\\
\sup_{0\leq t\leq T^*}(|g^{-\frac{1}{4}}w_t|^2_2+|\sqrt{g}\nabla^2 w|^2_2)(t)+\int^{T^*}_0(|g^{\frac{1}{4}}\nabla w_t|^2_2+|\sqrt{g}\nabla^3 w|^2_2) \text{d}t \leq & c_2^2,\nonumber\\
\sup_{0\leq t\leq T^*}(|g^{\frac{1}{4}}\nabla w_t|^2_2+|\sqrt{g}\nabla^3 w|^2_2)(t)+\int^{T^*}_0(|g^{-\frac{1}{4}} w_{tt}|^2_2+|\sqrt{g}\nabla^2 w_t|^2_2) \text{d}t \leq & c_2^2,\nonumber\\
\text{ess}\sup_{0\leq t\leq T^*}t(|\sqrt{g} \nabla^2w_t|^2_2+|g^{-\frac{1}{4}}w_{tt}|^2_2)(t) +
	\int^{T^*}_0t|g^{\frac{1}{4}} w_{tt}|_{D^1_*}^2\text{d}t  \leq & c^2_2,\nonumber\\
\sup_{0\leq t\leq T^*}(|v|^2_{D^2}+|v_t|^2_2+|g\nabla^2v|^2_2)(t)+\int^{T^*}_0(|v|_{D^3}^2+|v_t|^2_{D^1_*}) \text{d}t \leq & c_3^2,\\
\sup_{0\leq t\leq T^*}(|v|^2_{D^3}+|\sqrt{g}\nabla v_t|^2_2+|g_t|^2_{D^1_*})(t)+\int^{T^*}_0(|v|^2_{D^4}+|v_t|^2_{D^2}+|v_{tt}|^2_2)\text{d}t  \leq & c_4^2,\nonumber\\
\sup_{0\leq t\leq T^*}(|g\nabla^2v|^2_{D^1_*}+|g_t|^2_\infty)(t)+\int^{T^*}_0(|(g\nabla^2v)_t|^2_2+|g\nabla^2v|^2_{D^2})\text{d}t\leq & c_4^2,\nonumber\\
\text{ess}\sup_{0\leq t\leq T^*}t(|v|^2_{D^4}+|g\nabla^2v_t|^2_2)(t)+\int^{T^*}_0|g_{tt}|^2_{D^1_*}\text{d}t\leq & c^2_5,\nonumber\\ 
\text{ess}\sup_{0\leq t\leq T^*}t|v_{tt}(t)|^2_2+\int^{T^*}_0t(|v_{tt}|^2_{D^1_*}+|\sqrt{g}v_{tt}|_{D^1_*}^2+|v_t|^2_{D^3})\text{d}t\leq & c_5^2,\nonumber
\end{align}
then   there exists a unique strong solution $(\phi^{\epsilon,\eta},u^{\epsilon,\eta},l^{\epsilon,\eta},h^{\epsilon,\eta})$ to problem $\ef{ln}$  in $[0,T^*]\times \mathbb{R}^3$  such that
	\begin{equation}\label{2.13}\begin{aligned}
	&(\phi^{\epsilon,\eta}-\eta,l^{\epsilon,\eta}-\bar{l})\in C([0,T^*];D^1_*\cap D^3), \quad  (\phi^{\epsilon,\eta}_t,\nabla h^{\epsilon,\eta},h^{\epsilon,\eta}_t)\in C([0,T^*];H^2), \\[3pt]
	& h^{\epsilon,\eta}\in L^\infty\cap C([0,T^*]\times \mathbb{R}^3),\quad   	u^{\epsilon,\eta}\in C([0,T^*];H^3)\cap L^2([0,T^*];H^4),\\[3pt]
	&( u^{\epsilon,\eta}_t,l^{\epsilon,\eta}_t)\in C([0,T^*];H^1)\cap L^2([0,T^*];D^2),\quad  (u^{\epsilon,\eta}_{tt},l^{\epsilon,\eta}_{tt})\in L^2([0,T^*];L^2),\\[3pt]
	& t^{\fr{1}{2}}u^{\epsilon,\eta}\in L^\infty([0,T^*];D^4), \quad t^{\fr{1}{2}}u^{\epsilon,\eta}_t\in L^\infty([0,T^*];D^2)\cap L^2([0,T^*];D^3),\\[3pt]
 & t^{\fr{1}{2}}(u^{\epsilon,\eta}_{tt}, l^{\epsilon,\eta}_{tt})\in L^\infty([0,T^*];L^2)\cap L^2([0,T^*];D^1_*),\quad   t^{ \frac{1}{2} } l^{\epsilon,\eta}_t\in L^\infty([0,T^*];D^2).
	\end{aligned}\end{equation}
Moreover, the following uniform estimates hold: for $0\leq t\leq T^*$,
\begin{align}\label{key1kk}
\|(\phi^{\epsilon,\eta}-\eta)(t)\|^2_{D^1_*\cap D^3}+\|\phi^{\epsilon,\eta}_t(t)\|^2_2+
|\phi^{\epsilon,\eta}_{tt}(t)|^2_2+\int^t_0\|\phi^{\epsilon,\eta}_{ss}\|^2_1\text{d}s\leq &c^{2}_5,\nonumber\\
\|\psi^{\epsilon,\eta}(t)\|^2_{L^q\cap D^{1,3}\cap D^2}\leq c_1,\hspace{2mm}|\psi^{\epsilon,\eta}_t(t)|_2\leq c^{2}_4,\quad 
|h^{\epsilon,\eta}_t(t)|_\infty^2\leq & c^2_4, \nonumber\\ 
h^{\epsilon,\eta}(t,x)>\fr{1}{2c_0}, \ \fr{2}{3}\eta^{-2\iota}<\varphi^{\epsilon,\eta}, \ 
|\psi^{\epsilon,\eta}_t(t)|^2_{D^1_*}+\int^t_0(|\psi^{\epsilon,\eta}_{ss}|^2_2+|h^{\epsilon,\eta}_{ss}|^2_6)\text{d}s\leq & c^{2}_5,\nonumber\\ 
\|\tilde{n}^{\epsilon,\eta}(t)\|_{ L^\infty\cap D^{1,q}\cap D^{1,4}\cap D^{1,6} \cap D^{2} \cap D^3}+|\nabla (h^{\epsilon,\eta})^\frac{3}{8}|_4+|(h^{\epsilon,\eta})^{-\frac{1}{4}}\nabla^2 h^{\epsilon,\eta}(t)|_2\leq& c_1,\nonumber\\
|\nabla (h^{\epsilon,\eta})^\frac{3}{4}|_{D^1_*}\leq c_1,\ \ |\tilde{n}^{\epsilon,\eta}_t(t)|_2\leq &  c^{2}_4,\nonumber\\ |\tilde{n}^{\epsilon,\eta}_t(t)|_\infty+	|\nabla \tilde{n}^{\epsilon,\eta}_t(t)|_2+ |\nabla \tilde{n}^{\epsilon,\eta}_t(t)|_6+ |\tilde{n}^{\epsilon,\eta}_{tt}(t)|_2\leq& c^{2}_5,\nonumber\\
|u^{\epsilon,\eta}|^2_\infty+|\sqrt{h^{\epsilon,\eta}}\nabla u^{\epsilon,\eta}(t)|^2_2+\|u^{\epsilon,\eta}(t)\|^2_1+\int^t_0\big(\|\nabla u^{\epsilon,\eta}\|^2_1+|u^{\epsilon,\eta}_s|^2_2\big)\text{d}s\leq & c^{2}_1,\nonumber\\
	|\nabla l^{\epsilon,\eta}(t)|^2_2+|(h^{\epsilon,\eta})^{\fr{1}{4}}\nabla l^{\epsilon,\eta}(t)|^2_2\leq &c^2_2,\nonumber\\
\int^t_0(|(h^{\epsilon,\eta})^{-\fr{1}{4}}l^{\epsilon,\eta}_s|_2^2+|\sqrt{h^{\epsilon,\eta}} \nabla^2l^{\epsilon,\eta}|^2_2+|\nabla^2l^{\epsilon,\eta}|^2_2)\text{d}s\leq & c^{2}_2,\nonumber\\
 |(h^{\epsilon,\eta})^{-\frac{1}{4}}l^{\epsilon,\eta}_{t}(t)|^2_2+|\sqrt{h^{\epsilon,\eta}}\nabla^2l^{\epsilon,\eta}(t)|^2_2\leq & c^2_2,\nonumber\\
 \int^t_0(|(h^{\epsilon,\eta})^{\frac{1}{4}} \nabla l^{\epsilon,\eta}_s|^2_2+|\sqrt{h^{\epsilon,\eta}}\nabla^3l^{\epsilon,\eta} |_2^2)\text{d}s
   \leq &c^{2}_2,\\
    |(h^{\epsilon,\eta})^{\fr{1}{4}}\nabla l^{\epsilon,\eta}_t(t)|^2_2+
  |\sqrt{h^{\epsilon,\eta}}\nabla^3l^{\epsilon,\eta} (t)|_2^2\leq& c^2_2,\nonumber\\
   \int^t_0(|(h^{\epsilon,\eta})^{-\fr{1}{4}} l^{\epsilon,\eta}_{ss}|_2^2+|\sqrt{h^{\epsilon,\eta}}\nabla^2 l^{\epsilon,\eta}_s|_2^2)\text{d}s\leq &c^{2}_2,\nonumber\\
  	t|l^{\epsilon,\eta}_t(t)|^2_{D^2}+t|\sqrt{h^{\epsilon,\eta}} \nabla^2l^{\epsilon,\eta}_t(t)|^2_2+t|(h^{\epsilon,\eta})^{-\frac{1}{4}}l^{\epsilon,\eta}_{tt}(t)|^2_2  \leq &  c^{2}_2,\nonumber\\
\int^t_0s(|l^{\epsilon,\eta}_{ss}|_{D^1_*}^2+|h^{\frac{1}{4}} l^{\epsilon,\eta}_{ss}|_{D^1_*}^2)\text{d}s  \leq  c_1^2,\quad  c_1^{-1}\leq l^{\epsilon,\eta}(t,x)\leq & c_1,\nonumber\\
(|u^{\epsilon,\eta}|_{D^2}^2+|h^{\epsilon,\eta}\nabla^2u^{\epsilon,\eta}|^2_2+|u^{\epsilon,\eta}_t|^2_2)(t)\leq &c^2_3,\nonumber\\
\int^t_0(|u^{\epsilon,\eta}|^2_{D^3}+|h^{\epsilon,\eta}\nabla^2u^{\epsilon,\eta}|_{D^1_*}^2+|u^{\epsilon,\eta}_s|^2_{D^1_*})\text{d}s\leq & c^{2}_3,\nonumber\\
(|u^{\epsilon,\eta}_t|^2_{D^1_*}+|\sqrt{h^{\epsilon,\eta}}\nabla u^{\epsilon,\eta}_t|^2_2+|u^{\epsilon,\eta}|^2_{D^3}+|h^{\epsilon,\eta}\nabla^2u^{\epsilon,\eta}|^2_{D^1_*})(t)+\int^t_0|u^{\epsilon,\eta}_s|^2_{D^2}\text{d}s\leq &  c^{2}_4,\nonumber\\
\int^t_0(|u^{\epsilon,\eta}_{ss}|^2_2+|u^{\epsilon,\eta}|^2_{D^4}+|h^{\epsilon,\eta}\nabla^2u^{\epsilon,\eta}|^2_{D^2}+|(h^{\epsilon,\eta}\nabla^2u^{\epsilon,\eta})_s|^2_2)\text{d}s\leq & c^{2}_4,\nonumber\\
t|u^{\epsilon,\eta}_t(t)|^2_{D^2}+t|h^{\epsilon,\eta}\nabla^2u^{\epsilon,\eta}_t(t)|^2_2+t|u^{\epsilon,\eta}_{tt}(t)|^2_2+t|u^{\epsilon,\eta}(t)|^2_{D^4} \leq & c^{2}_5,\nonumber\\
    \int^t_0s(|u^{\epsilon,\eta}_{ss}|_{D^1_*}^2+|h\nabla^3 u^{\epsilon,\eta}_s|^2_{2}+|\sqrt{h^{\epsilon,\eta}} u^{\epsilon,\eta}_{ss}|_{D^1_*}^2)\text{d}s  \leq & c^{2}_5\nonumber.
\end{align}

\end{lemma}
 
Next, we will give the detail proof for  Lemma \ref{ls} by deriving  the uniform a priori estimates independent of $(\epsilon,\eta)$ for the  solution $(\phi^{\epsilon,\eta},u^{\epsilon,\eta},l^{\epsilon,\eta},h^{\epsilon,\eta})$ to   $\ef{ln}.$

\subsection{Proof for Lemma \ref{ln}}  
The proof for Lemma \ref{ln} is divided into  four parts.
In  \S 3.2,  $ M(c)\geq 1$  denotes  a generic continuous   increasing function of $c$, and $C\geq 1$  denotes  a generic  constant. Both may    depend  on the fixed constants $(A, R, c_v, \alpha, \beta, \gamma, \delta, \daleth,T)$, and  $M(c)$ also  depends on  $c$. They are independent of $(\epsilon,\eta)$ and  may vary from  line to line. Here  we will first give the global existence of 
$(\phi^{\epsilon,\eta},l^{\epsilon,\eta},h^{\epsilon,\eta})$, then establish the desired uniform estimates for $(\phi^{\epsilon,\eta},l^{\epsilon,\eta},h^{\epsilon,\eta})$ in some uniform short time, which  yields the desired local existence of $u^{\epsilon,\eta}$. At last, we derive the desired uniform estimates for $u^{\epsilon,\eta}$ in an uniform short time.
\subsubsection{Basic setting for  the uniform a priori estimates} \label{basicset}
First, according to the classical well-posedness theory for the linear transport equation and parabolic equations (see \cite{oar}), one has that for any $T>0$,  there exists a unique strong solution $(\phi^{\epsilon,\eta},l^{\epsilon,\eta},h^{\epsilon,\eta})$ to the following  sub-problem  of  $\ef{ln}$  in $[0,T]\times \mathbb{R}^3$:
\begin{equation}\label{sub-ln}\left\{\begin{aligned}
&\phi^{\epsilon,\eta}_t+v\cdot\nabla\phi^{\epsilon,\eta}+(\gamma-1)\phi^{\epsilon,\eta} \text{div}v=0,\\
& ((h^{\epsilon,\eta})^2+\epsilon^2)^{-\frac{1}{4}}( l^{\epsilon,\eta}_t+v\cdot\nabla l^{\epsilon,\eta})-a_4 w^\nu ((h^{\epsilon,\eta})^2+\epsilon^2)^{\frac{1}{4}} \triangle l^{\epsilon,\eta}\\
= & a_5 w^\nu \tilde{n}^{\epsilon,\eta} g^{\frac{3}{2}}H(v)+a_6w^{\nu+1}((h^{\epsilon,\eta})^2+\epsilon^2)^{-\frac{1}{4}} \text{div} \psi^{\epsilon,\eta}+\Pi(l^{\epsilon,\eta},h^{\epsilon,\eta},w,g),\\[2pt]
&h^{\epsilon,\eta}_t+v\cdot \nabla h^{\epsilon,\eta}+(\delta-1)g\text{div}v=0,\\
&(\phi^{\epsilon,\eta},l^{\epsilon,\eta},h^{\epsilon,\eta})|_{t=0}=(\phi^\eta_0,l^\eta_0,h^\eta_0)\\[1pt]
=&(\phi_0+\eta,l_0,(\phi_0+\eta)^{2\iota})\quad  \text{for} \quad x\in\mathbb{R}^3,\\
&(\phi^{\epsilon,\eta},l^{\epsilon,\eta},h^{\epsilon,\eta})\rightarrow (\eta,\bar{l},\eta^{2\iota}) \quad \text{as} \hspace{2mm}|x|\rightarrow \infty \quad \rm for\quad t\geq 0.
\end{aligned}\right.\end{equation}
Moreover, this solution satisfies the regularities shown in \eqref{2.13}.

\begin{remark}\label{r1}
First, it follows from  the definition of $g^\eta_2$ and $\phi^\eta_0>\eta$  that 
	\begin{equation}\label{incom}\left\{\begin{aligned}
	&L((\phi^\eta_0)^{2\iota}u^\eta_0)=g^\eta_2-\fr{\delta-1}{a\delta}G(\psi^\eta_0,u^\eta_0),\\
&(\phi^\eta_0)^{2\iota}u^\eta_0\longrightarrow 0 \ \ \text{as}\ \ |x|\longrightarrow \infty,\\
	\end{aligned}\right.\end{equation}
	where $\psi^\eta_0=\frac{a\delta}{\delta-1}\nabla (\phi^\eta_0)^{2\iota}=\frac{a\delta}{\delta-1}\nabla h^\eta_0$ and 
\begin{equation}\label{Gdingyi}
G(\psi,u)=\alpha\psi\cdot\nabla u+\alpha \text{div}(u\otimes\psi)+(\alpha+\beta)(\psi\text{div} u+\nabla(\psi\cdot u)).
\end{equation}
Then  the standard elliptic theory and $\ef{2.14}$ yield  that
\begin{equation}\label{incc}
\begin{split}
|(\phi^\eta_0)^{2\iota}u^\eta_0|_{D^2}\leq & C(|g^\eta_2|_2+|G(\psi^\eta_0,u^\eta_0)|_2)\leq M(c_0),\\
|(\phi^\eta_0)^{2\iota}\nabla^2u^\eta_0|_2\leq & C(|(\phi^\eta_0)^{2\iota}u^\eta_0|_{D^2}+|\nabla\psi^\eta_0|_3|u^\eta_0|_6+|\psi^\eta_0|_\infty|\nabla u^\eta_0|_2)\leq M(c_0).
\end{split}
\end{equation}
Due to $\nabla^2\phi^{2\iota}_0\in L^3$ and $\ef{incc}$, it holds that 
\begin{equation}\label{incc1}
|(\phi^\eta_0)^\iota\nabla^2\phi^\eta_0|_2+|(\phi^\eta_0)^\iota\nabla(\psi^\eta_0\cdot Q(u^\eta_0))|_2\leq M(c_0),
\end{equation}
where one has used the fact that
\begin{equation*}
\begin{split}
|\phi_0^\iota\nabla^2\phi_0 |_2\leq & C(|\phi_0|_6|\phi_0|^{-\iota}_\infty|\nabla^2\phi^{2\iota}_0|_3+|\nabla \phi^\iota_0|_6|\nabla \phi_0|_3)\leq  M(c_0),\\
|(\phi^\eta_0)^\iota\nabla^2\phi^\eta_0|_2=& \Big|\phi_0^\iota\nabla^2\phi_0 \frac{\phi_0^{-\iota}}{(\phi_0+\eta)^{-\iota}}\Big|_2\leq  |\phi_0^\iota\nabla^2\phi_0 |_2\leq M(c_0).
\end{split}
\end{equation*}

Second,  the initial  compatibility condition 
$$\nabla((\phi^\eta_0)^{2\iota}Lu^\eta_0)=(\phi^\eta_0)^{-\iota}g^\eta_3\in L^2$$ 
implies  formally that 
\begin{equation}\label{incom1}\left\{\begin{aligned}
	&L((\phi^\eta_0)^{2\iota}u^\eta_0)=\triangle^{-1}\text{div}((\phi^\eta_0)^{-\iota}g^\eta_3)-\fr{\delta-1}{a\delta}G(\psi^\eta_0,u^\eta_0),\\
&(\phi^\eta_0)^{2\iota}u^\eta_0\longrightarrow 0 \ \ \text{as}\ \ |x|\longrightarrow \infty.\\
	\end{aligned}\right.\end{equation}
Thus the standard  elliptic theory yields
\begin{equation}\label{incc*}
\begin{split}
|(\phi^\eta_0)^{2\iota}u^\eta_0|_{D^3}\leq & C(|\phi^\eta_0)^{-\iota}g^\eta_3|_2+|G(\psi^\eta_0,u^\eta_0)|_{D^1})\leq M(c_0)<\infty,\\
|(\phi^\eta_0)^{2\iota}\nabla^3u^\eta_0|_2\leq & C(|(\phi^\eta_0)^{2\iota}u^\eta_0|_{D^3}+|\nabla\psi^\eta_0|_3|\nabla u^\eta_0|_6\\
&+|\psi^\eta_0|_\infty|\nabla^2 u^\eta_0|_2+|\nabla^2\psi_0^\eta|_2|u_0^\eta|_\infty)\leq M(c_0).
\end{split}
\end{equation}

Third, from the initial  compatibility condition  
$$(\phi^\eta_0)^{3\iota}L\nabla u^\eta_0=g^\eta_3-(\phi^\eta_0)^{\iota}Lu^\eta_0\otimes\nabla(\phi^\eta_0)^{2\iota},$$ 
one has
\begin{equation*}
\begin{split}
 L\big((\phi^\eta_0)^{3\iota}\nabla u^\eta_0\big)&=(\phi^\eta_0)^{3\iota}L\nabla u^\eta_0+[L,(\phi^\eta_0)^{3\iota}]\nabla u_0\\
 &=g^\eta_3-(\phi^\eta_0)^{\iota}Lu^\eta_0\otimes\nabla(\phi^\eta_0)^{2\iota}+[L,(\phi^\eta_0)^{3\iota}]\nabla u_0,
 \end{split}
\end{equation*}
here $[L,(\phi^\eta_0)^{3\iota}]\nabla u_0$  is a commutator. Then the standard elliptic theory yields
\begin{equation}\label{incc**}
\begin{split}
|(\phi^\eta_0)^{3\iota}\nabla u^\eta_0|_{D^2}\leq & C(|g^\eta_3|_2+|(\phi^\eta_0)^{2\iota}\nabla^2 u^\eta_0|_2|\psi^\eta_0|_\infty(|h_0^\eta)^{-\frac{1}{2}}|_\infty)\\
&+|\psi^\eta_0|^2_\infty(|h_0^\eta)^{-\frac{1}{2}}|_\infty|\nabla u_0^\eta|_2+|\nabla^2h_0^\eta|_3|(\phi^\eta_0)^{\iota}\nabla u^\eta_0|_6\leq M(c_0),\\
|(\phi^\eta_0)^{3\iota}\nabla^3u^\eta_0|_2\leq & C(|(\phi^\eta_0)^{3\iota}\nabla u^\eta_0|_{D^2}+|\psi^\eta_0|_\infty||(h_0^\eta)^{-\frac{1}{2}}|_\infty|(\phi^\eta_0)^{2\iota}\nabla^2 u^\eta_0|_2\\
&+|\psi^\eta_0|^2_\infty||(h_0^\eta)^{-\frac{1}{2}}|_\infty|\nabla u^\eta_0|_2
+|\nabla^2h_0^\eta|_3|(\phi^\eta_0)^{\iota}\nabla u^\eta_0|_6\leq M(c_0).
\end{split}
\end{equation}

Similarly,   the definition of $g^\eta_5$ and  $\phi^\eta_0>\eta$ imply  that 
\begin{equation}\label{shangchuzhi1}\left\{\begin{aligned}
&\triangle ((\phi^\eta_0)^{\frac{3}{2}\iota}(l^\eta_0-\bar{l}))=g^\eta_5+2\nabla (\phi^\eta_0)^{\frac{3}{2}\iota} \cdot \nabla l^\eta_0+(l^\eta_0-\bar{l})\triangle (\phi^\eta_0)^{\frac{3}{2}\iota},\\
&(\phi^\eta_0)^{\frac{3}{2}\iota}(l^\eta_0-\bar{l})\longrightarrow 0 \ \ \text{as}\ \ |x|\longrightarrow \infty,
	\end{aligned}\right.\end{equation}
which, together with  $\ef{2.14}$, yields that 
\begin{equation}\label{shangchuzhi2}
\begin{split}
|(\phi^\eta_0)^{\frac{3}{2}\iota}(l^\eta_0-\bar{l})|_{D^2}\leq & C(|g^\eta_5|_2+|(\phi^\eta_0)^{-\frac{1}{2}\iota}|_\infty|\psi^\eta_0|_\infty|\nabla l^\eta_0|_2\\
&+|l^\eta_0-\bar{l}|_\infty|\nabla^2 (h^\eta_0)^{\frac{3}{4}}|_2)\leq M(c_0)<\infty,\\
|(\phi^\eta_0)^{\frac{3}{2}\iota}\nabla^2 l^\eta_0|_2\leq & C(|(\phi^\eta_0)^{\frac{3}{2}\iota}(l^\eta_0-\bar{l})|_{D^2}+|(\phi^\eta_0)^{-\frac{1}{2}\iota}|_\infty|\psi^\eta_0|_\infty|\nabla l^\eta_0|_2\\
&+|l^\eta_0-\bar{l}|_\infty|\nabla^2 (h^\eta_0)^{\frac{3}{4}}|_2)\leq M(c_0)<\infty.
\end{split}
\end{equation}

At last,   the initial  compatibility condition 
$$\nabla ((\phi^\eta_0)^{\iota}\triangle l^\eta_0)=(\phi^\eta_0)^{-\frac{3}{2}\iota}g^\eta_6\in L^2$$ 
implies  formally that 
\begin{equation}\label{shangchuzhi33}\left\{\begin{aligned}
	&\triangle((\phi^\eta_0)^{\frac{5}{2}\iota}(l^\eta_0-\bar{l}))=\triangle^{-1}\text{div}\big(g^\eta_6+\nabla (\phi^\eta_0)^{\frac{3}{2}\iota}\cdot (\phi^\eta_0)^{\iota}\triangle l^\eta_0\big)\\
 &\qquad +2\nabla (\phi^\eta_0)^{\frac{5}{2}\iota }\cdot \nabla l^\eta_0+(l^\eta_0-\bar{l})\triangle (\phi^\eta_0)^{\frac{5}{2}\iota},\\
&(\phi^\eta_0)^{\frac{5}{2}\iota}(l^\eta_0-\bar{l})\longrightarrow 0 \ \ \text{as}\ \ |x|\longrightarrow \infty,\\
	\end{aligned}\right.\end{equation}
which  yields
\begin{equation}\label{shangchuzhi44}
\begin{split}
|(\phi^\eta_0)^{\frac{5}{2}\iota}(l^\eta_0-\bar{l}))|_{D^3}\leq & C(|g^\eta_6|_2+|\psi^\eta_0|_\infty|(h^\eta_0)^{-1}|^{\frac{1}{2}}_\infty|g^\eta_5|_2+\aleph)\leq M(c_0),\\
|(\phi^\eta_0)^{\frac{5}{2}\iota}\nabla^3l^\eta_0|_2\leq & C(|(\phi^\eta_0)^{\frac{5}{2}\iota}(l^\eta_0-\bar{l}))|_{D^3}+\aleph)\leq M(c_0),
\end{split}
\end{equation}
where 
\begin{equation*}
\begin{split}
\aleph=&|\psi^\eta_0|_\infty|(\phi^\eta_0)^{\frac{1}{2}\iota}\nabla^2 l^\eta_0|_2
+|\nabla l^\eta_0|_3(|\nabla (\phi^\eta_0)^{\frac{7}{4}\iota}|_\infty|\nabla (h^\eta_0)^{\frac{3}{8}}|_6+|(h^\eta_0)^{\frac{1}{4}}\nabla^2h^\eta_0|_6)\\
&+|l^\eta_0-\bar{l}|_\infty(|(h^\eta_0)^{\frac{1}{4}}\nabla^3h^\eta_0|_2+|\nabla (\phi^\eta_0)^{\frac{1}{2}\iota}|_6|\nabla^2 (\phi^\eta_0)^{2\iota}|_3+|\psi^\eta_0|_\infty|\nabla (\phi^\eta_0)^{\frac{1}{4}\iota}|^2_4).
\end{split}
\end{equation*}

	\end{remark}

In  the rest of \S  3.2, without  ambiguity,
we simply  drop the superscripts $\epsilon$ and $\eta$ in 
$(\phi^\eta_0,u^\eta_0,l^\eta_0,h^\eta_0,\psi^\eta_0)$, 
$(\phi^{\epsilon,\eta},u^{\epsilon,\eta},l^{\epsilon,\eta},h^{\epsilon,\eta},\psi^{\epsilon,\eta})$, and $(g_1^\eta,g_2^\eta,g_3^\eta,g_4^\eta,g_5^\eta,g_6^\eta)$.
Based on the above basic setting, now we start to derive the desired uniform a priori estimates.
\subsubsection{A priori estimates for quantities related to $\rho$.} 

In the rest of \S 3.2, let $\epsilon \in (0,1]$,   $(\phi,l, h)$ be the  solution to $\ef{sub-ln}$ in $[0,T]\times\mathbb{R}^3$ obtained in \S \ref{basicset}, and set 
 $$\varphi=h^{-1},\ \ \tilde{n}=a^b (h^2+\epsilon^2)^{\frac{b}{2}}, \ \ 
\xi=\nabla h^{\frac{3}{4}}, \ \  \zeta= \nabla h^{\frac{3}{8}}.$$
First, it follows from  $\psi=\frac{a\delta}{\delta-1}\nabla h$  and   $\ef{ln}_4$ that 
\begin{equation}\label{psieq}
\psi_t+\sum_{k=1}^3 A_k(v) \partial_k\psi+B^*(v)\psi+a\delta ( g\nabla \text{div} v +\nabla g \text{div} v ) =0,
\end{equation}
with  $B^*(v)=(\nabla v)^\top$ and $A_k(v)$ defined in \eqref{2.3}.
Then   the standard energy estimates  and characteristic method  for hyperbolic equations,  one  can obtain the following   estimates on  $(\phi,\psi)$, which are  similar to      Lemmas 3.2-3.3 in  \cites{dxz,zz2}.

\begin{lemma}\label{phiphi}
 For $T_1=\min\{T^*,(1+Cc_4)^{-6}\}$ and  $ t\in [0,T_1]$,   	it holds that 
\begin{equation}\label{phi1}\begin{aligned}
	&\|\phi(t)-\eta\|_{D^1_*\cap D^3}\leq Cc_0,\quad|\phi_t(t)|_2\leq Cc_0c_1,\quad|\phi_t(t)|_{D^1_*}\leq Cc_0c_3,\\
	&|\phi_t(t)|_{D^2}\leq Cc_0c_4,\quad|\phi_{tt}(t)|_2\leq Cc_4^3,\quad\int^t_0\|\phi_{ss}\|^2_1 {\rm d}s\leq Cc_0^2c_4^2.
	\end{aligned}\end{equation}
\end{lemma}

\begin{lemma}\label{psi} 	\rm {For} $t\in [0, T_1]$ and  $q>3$, it holds that 
	\begin{equation}\label{2.24}\begin{split}
	&|\psi(t)|_\infty^2+\|\psi(t)\|^2_{L^q\cap D^{1,3}\cap D^2}\leq Cc_0^2,\hspace{2mm}|\psi_t(t)|_2\leq Cc_3^2,\\
	&|h_t(t)|^2_\infty\leq Cc_3^3c_4,\quad |\psi_t(t)|^2_{D^1_*}+\int^t_0(|\psi_{ss}|^2_2+|h_{ss}|^2_6){\rm{\text{d} } }s\leq Cc_4^4.
	\end{split}\end{equation}

\end{lemma}

Second, due to the presence of heat conductivity,  the  estimates  in Lemmas \ref{phiphi}-\ref{psi} are  far from  enough for   dealing with    terms such as 
$
a_6w^{\nu+1}(h^2+\epsilon^2)^{-\frac{1}{4}} \text{div} \psi
$,
some more  estimates related to $\rho$  are carried out  as follows. 

\begin{lemma}\label{gh}
It holds that for $(t,x)\in [0,T_1]\times \mathbb{R}^3$,
\begin{equation}\label{g/h}
\begin{split}
& |g \text{div} v|_\infty\leq Cc^{\frac{3}{2}}_3c^{\frac{1}{2}}_4,\quad   \fr{2}{3}\eta^{-2\iota}<\varphi(t,x)\leq  2c_0,\quad 
	h(t,x)>\fr{1}{2c_0},\quad \\
&\tilde{C}^{-1}\leq gh^{-1}(t,x)\leq   \tilde{C},\quad 
|h^{-\frac{1}{2}}h_{t}|_2\leq M(c_0)c_1,\quad |h^{-\frac{1}{2}}h_{tt}|_2\leq M(c_0)c_4^2.\end{split}
\end{equation}
 where $\tilde{C}$ is a suitable constant independent of $(\epsilon,\eta)$ and $c_i$ $(i=1,2,...,5)$.
\end{lemma}

\begin{proof}
First, note that 
\begin{equation}\label{2.34a}
\varphi_t+v\cdot \nabla \varphi-(\delta-1)g\varphi^2\text{div} v=0.
\end{equation}
Let $X(t;x)$ be  the particle path defined by
\begin{equation}\begin{cases}
\label{2.34b}
\frac{d}{dt}X(t;x)=v(t,X(t; x)),\quad  0\leq t\leq T;\\[8pt]
X(0;x)=x, \quad \qquad \qquad \ \ \quad   x\in \mathbb{R}^3.
\end{cases}
\end{equation}
Then 
\begin{equation}\label{2.34c}
\begin{split}
\displaystyle
\varphi(t,X(t;x))=\varphi_0(x)\Big(1+(1-\delta)\varphi_0(x)\int_0^t g\text{div}v (s,X(s;x))\text{d}s\Big)^{-1},
\end{split}
\end{equation}
which,  along with  the fact that 
\begin{equation}\label{2.26d}\begin{split}
|g\text{div} v|_\infty\leq & C|g\text{div} v|^{\frac{1}{2}}_{D^1}|g\text{div} v|^{\frac{1}{2}}_{D^2}
\leq  C\big(|\nabla g|_\infty|\nabla v|_2+|g\nabla^2 v|_2\big)^{\frac{1}{2}}\\
&\cdot \big(|\nabla^2 g|_3|\nabla v|_6+|\nabla g|_\infty|\nabla^2 v|_2+|g\nabla^2v|_{D^1_*}\big)^{\frac{1}{2}}
\leq  Cc^{\frac{3}{2}}_3c^{\frac{1}{2}}_4,
\end{split}
\end{equation}
implies that 
\begin{equation}\label{2.34d}
\begin{split}
\displaystyle
\frac{2}{3}\eta^{-2\iota}<\varphi(t,x)<2|\varphi_0|_\infty\leq 2c_0\quad \text{for} \quad  [t,x]\in [0,T_1]\times \mathbb{R}^3.
\end{split}
\end{equation}

 Second, set $gh^{-1}=y(t,x)$. Then 
\begin{equation}\label{ghe}
y_t+yh^{-1}h_t=g_t\varphi; \quad y(0,x)=1.
\end{equation}
Thus 
\begin{equation}\label{y}
y(t,x)=\exp\big(-\int^t_0h_s h^{-1}\text{d}s\big)\big(1+\int^t_0g_s\varphi\exp\big(\int^s_0 h_\tau h^{-1}\text{d}\tau\big)\text{d}s\big),
\end{equation}
which, along with  Lemma \ref{psi}, \ef{2.16} and \ef{2.34d}, yields $\tilde{C}^{-1}\leq gh^{-1}(t,x)\leq  \tilde{C}.$\\

At last, according to  $\ef{ln}_4$, \ef{2.16} and Lemma \ref{psi},  one has 
\begin{equation*}\label{}
\begin{split}
|h^{-\frac{1}{2}}h_{t}|_2&\leq C(|h^{-\frac{1}{2}}v\nabla h|_2+|h^{-\frac{1}{2}}gdivv|_2)\leq M(c_0)c_1,\\
|h^{-\frac{1}{2}}h_{tt}|_2&\leq C(|h^{-\frac{1}{2}}v_t\nabla h|_2+|h^{-\frac{1}{2}}v\nabla h_t|_2+|h^{-\frac{1}{2}}g_tdivv|_2+|h^{-\frac{1}{2}}gdivv_t|_2)\\
&\leq  M(c_0)(|v_t|_2+|v|_\infty|\psi_t|_2+|g_t|_\infty|\nabla v|_2+|\sqrt{g}\nabla v_t|_2)\\
&\leq  M(c_0)c_4^2,
\end{split}
\end{equation*}
where  one has used the fact that 
 \begin{equation}\label{vH2}
\|v\|_2\leq \|u_0\|_2+t^\frac{1}{2}(\int^t_0 \|v_s\|_2^2ds)^{\frac{1}{2}}\leq Cc_0+Cc_4t^{\frac{1}{2}}\leq Cc_0,\ \ \text{for} \ \ t\leq T_1.
\end{equation}
The proof of Lemma \ref{gh} is complete.
\end{proof}

\begin{lemma}\label{varphi}\rm {For} $t\in [0, T_1]$ and  $q>3$, it holds that 
	\begin{equation}\label{2.34}\begin{aligned}
|\xi(t)|_{D^1_*}+|\zeta(t)|_4+|h^{-\frac{1}{4}}\nabla^2 h(t)|_2\leq& M(c_0),\\
  \|\tilde{n}(t)\|_{ L^\infty\cap D^{1,q}\cap D^{1,4}\cap D^{1,6} \cap D^{2} \cap D^3}\leq  M(c_0),\quad |\tilde{n}_t(t)|_2\leq & M(c_0)c_1,\\
|\tilde{n}_t(t)|_\infty+	|\nabla \tilde{n}_t(t)|_2+ |\nabla \tilde{n}_t(t)|_6\leq  M(c_0)c_4^2,\ \  |\tilde{n}_{tt}(t)|_2\leq& M(c_0)c_4^3.	
	\end{aligned}\end{equation}
\end{lemma}
\begin{proof} 
We start with estimate on $\xi$.  $\ef{ln}_4$ implies  that 
\begin{equation}\label{2.23c}
\xi_t+\sum_{k=1}^3 A_k(v) \partial_k\xi+B^*(v)\xi+\fr{3}{4}(\delta-1)\nabla (h^{-\fr{1}{4}}g\text{div}v)=0,
\end{equation}
where  $B^*(v)=(\nabla v)^\top$ and $A_k(v)$ is defined in \eqref{2.3}.

Set $\varsigma=(\varsigma_1,\varsigma_2,\varsigma_3)^\top$ ($ |\varsigma|=1$ and $\varsigma_i=0,1$). Applying  $\partial_{x}^{\varsigma} $ to $\ef{2.23c}$, multiplying by $2\partial_{x}^{\varsigma} \xi$ and then integrating over $\mathbb{R}^3$, one can get
\begin{equation*}\begin{split}
\frac{d}{dt}|\partial_{x}^{\varsigma}  \xi|^2_2
\leq& C(|\nabla v|_\infty+|g\text{div} v|_\infty|\varphi|_\infty)|\nabla \xi|_2|\partial_{x}^{\varsigma}\xi|_2+C|\nabla^2 v|_3|\xi|_6|\partial_{x}^{\varsigma}\xi|_2\\
&+C\big(|\varphi|^{\fr{5}{4}}_\infty(|gh^{-1}|_\infty|\psi|^2_\infty|\nabla v|_2+|\nabla g|_\infty|\nabla v|_2|\psi|_\infty+|g\nabla^2 v|_2|\psi|_\infty) \\
& +|\varphi|^{\fr{1}{4}}_\infty(|\nabla^2g|_3|\nabla v|_6+|\nabla g|_\infty|\nabla^2v|_2+|g\nabla^3v|_2)\big)|\partial_{x}^{\varsigma}\xi|_2,
\end{split}
\end{equation*}
which, along with  $\ef{2.16}$ and \eqref{2.26d},  
Lemmas \ref{psi}-\ref{gh}  and Gronwall's inequality, yields that  
\begin{equation}\label{2.23e2}
|\nabla \xi(t) |_2\leq Cc_0\quad \text{for} \quad 0\leq t\leq T_1.
\end{equation}

Similarly,   $\ef{ln}_4$ implies  that 
\begin{equation}\label{2.23c1}
\zeta_t+\sum_{k=1}^3 A_k(v) \partial_k\zeta+B^*(v)\zeta+\fr{3}{8}(\delta-1)\nabla (h^{-\fr{5}{8}}g\text{div}v)=0.
\end{equation}
Then multiplying $\ef{2.23c1}$ by $4|\zeta|^2\zeta$ and integrating with respect to $x$ over $\mathbb{R}^3$ yield
\begin{equation*}
\begin{split}
\fr{d}{dt}|\zeta|^4_4 \leq &C|\nabla v|_\infty|\zeta|^4_4+C\big(|g\nabla^2v|_4+|\nabla g|_\infty|\nabla v|_4 +|gh^{-1}|_\infty|\nabla v|_4|\psi|_\infty\big)|\varphi|^{\fr{5}{8}}_\infty|\zeta|^3_4,
\end{split}
\end{equation*}
which, along with  $\ef{2.16}$,  Lemmas \ref{psi}-\ref{gh}  and Gronwall's inequality, yields that 
\begin{equation}\label{2.23e1}
|\zeta(t)|_4\leq Cc_0\quad \text{for} \quad 0\leq t\leq T_1.
\end{equation}

It follows from the fact that 
$$
h^{-\frac{1}{4}}\nabla^2 h=\frac{4}{3}\nabla \xi+\frac{16}{9}\zeta\otimes \zeta,
$$
\eqref{2.23e2} and  \eqref{2.23e1}  that 
\begin{equation}\label{2.bbbn}
|h^{-\frac{1}{4}}\nabla^2 h(t)|_2\leq C(|\nabla \xi(t) |_2+|\zeta(t)|^2_4) \leq  M(c_0)\quad \text{for} \quad 0\leq t\leq T_1.
\end{equation}

Finally, for the estimate on   
$\tilde{n}=a^b(h^2+\epsilon^2)^{\frac{b}{2}}$,  it follows from 
Lemmas \ref{phiphi}-\ref{gh}, \eqref{2.23e2} and  \eqref{2.23e1}-\eqref{2.bbbn} that 
 \begin{align*}
|\tilde{n}|_\infty&\leq C|(h^2+\epsilon^2)^{\frac{b}{2}}|_{\infty} \leq C|h^b|_\infty\leq M(c_0), \\
|\nabla\tilde{n}|_q&\leq C|(h^2+\epsilon^2)^{\frac{b}{2}-1}h\nabla h|_q
\leq C|h^{b-1}\psi|_q \leq  M(c_0),\\
|\nabla\tilde{n}|_4&\leq C|(h^2+\epsilon^2)^{\frac{b}{2}-1}h\nabla h|_4
\leq C|h^{b-\frac{3}{8}}\zeta|_4 \leq  M(c_0),\\
|\nabla\tilde{n}|_6&\leq C|(h^2+\epsilon^2)^{\frac{b}{2}-1}h\nabla h|_6
\leq C|h^{b-\frac{3}{4}}\xi|_6 \leq  M(c_0),\\
|\nabla^2\tilde{n}|_2&\leq C(|(h^2+\epsilon^2)^{\frac{b}{2}-1}(|h\nabla^2 h|+|\nabla h\cdot \nabla h|)|_2+|(h^2+\epsilon^2)^{\frac{b}{2}-2}(h\nabla h)^2|_2)\\
&\leq  C|h^{b-2}(|h\nabla^2 h|+|\nabla h\cdot \nabla h|)|_2\\
&\leq C(|h^{b-\frac{3}{4}}|_\infty|h^{-\frac{1}{4}}\nabla^2 h|_2+|h^{b-\frac{3}{4}}|_\infty|\nabla h^{\frac{3}{8}}|^2_4)\leq M(c_0),\\
|\nabla^3\tilde{n}|_2&\leq C(|(h^2+\epsilon^2)^{\frac{b}{2}-1}(|h\nabla^3 h|+|\nabla h\cdot \nabla^2 h|)|_2\\
&\qquad +|(h^2+\epsilon^2)^{\frac{b}{2}-2}(|h\nabla h\cdot\nabla h \cdot \nabla h|+|h^2\nabla h\cdot \nabla^2h|)|_2)\\
&\qquad +|(h^2+\epsilon^2)^{\frac{b}{2}-3}h^3\nabla h\cdot\nabla h \cdot \nabla h|_2)\\
&\leq  C(|h^{b-1}\nabla^3 h|_2+|h^{b-2}\nabla h \cdot \nabla^2 h|_2+|h^{b-3}\nabla h \cdot \nabla h \cdot \nabla h|_2)\\
&\leq C(|h^{b-1}|_\infty|\nabla^2 \psi|_2   +|h^{b-\frac{7}{4}}|_\infty|\psi|_\infty(|h^{-\frac{1}{4}}\nabla^2 h|_2+|\nabla h^{\frac{3}{8}}|^2_4))\leq M(c_0),\\
|\tilde{n}_t|_\infty&\leq C(|(h^2+\epsilon^2)^{\frac{b}{2}-1}hh_t|_\infty\leq C|h^{b-1}h_t|_\infty\leq M(c_0)c_4^2,\\
|\tilde{n}_t|_2&\leq C(|(h^2+\epsilon^2)^{\frac{b}{2}-1}hh_t|_2\leq C|h^{b-1}h_t|_2   \leq C|h^{b-\frac{1}{2}}|_\infty | h^{-\frac{1}{2}}h_t|_2\leq M(c_0)c_1,\\
|\nabla \tilde{n}_t|_2&\leq C(|(h^2+\epsilon^2)^{\frac{b}{2}-1}(|h\nabla h_t|+|\nabla h h_t|)|_2+|(h^2+\epsilon^2)^{\frac{b}{2}-2}h^2\nabla h h_t|_2)\\
&\leq C(|h^{b-2}(|h\nabla h_t|+|\nabla h h_t|)|_2)\\
&\leq C(|h^{b-1}|_\infty|\psi_t|_2+|h^{b-\frac{3}{2}}|_\infty|\psi|_\infty|h^{-\frac{1}{2}}h_t|_2)\leq M(c_0)c_4^2,\\
|\nabla \tilde{n}_t|_6&\leq C(|h^{b-2}(|h\nabla h_t|+|\nabla h h_t|)|_6)\\
&\leq C(|h^{b-1}|_\infty|\psi_t|_6+|h^{b-2}|_\infty|\psi|_\infty|h_t|_6)\leq M(c_0)c_4^2,\\
|\tilde{n}_{tt}|_2&\leq C(|(h^2+\epsilon^2)^{\frac{b}{2}-1}(|hh_{tt}|+ |h_t^2|)|_2+|(h^2+\epsilon^2)^{\frac{b}{2}-2}(h h_t)^2|_2))\\
&\leq C(|h^{b-\frac{1}{2}}|_\infty|h^{-\frac{1}{2}}h_{tt}|_2+|h^{b-\frac{3}{2}}|_\infty|h_t|_\infty|h^{-\frac{1}{2}}h_t|_2)\leq M(c_0)c_4^3.
\end{align*}


\end{proof}

\subsubsection{A priori estimates for $l$}Now we can  estimate  $l$  by making full use of  the {\it degenerate-singular parabolic}  structure introduced in $\eqref{qiyi}_2$. Recall that 
\begin{equation*}
\begin{split}
H(v)=2\alpha \sum_{i=1}^{3} (\partial_{i}v^{i})^{2}+\beta(\text{div}v)^2+\alpha \sum_{i\neq j}^{3} (\partial_{i}v^{j})^2+2\alpha\sum_{i>j} (\partial_{i}v^{j})(\partial_{j}v^{i}).
\end{split}
\end{equation*}
\begin{lemma}\label{l}For  $T_2=\min\{T_1,(1+Cc_4)^{-12-2\nu}\}$ and $t\in [0,T_2]$,  it holds that  
\begin{equation}\label{lt}
	\begin{aligned}
	|\nabla l(t)|^2_2+|h^{\fr{1}{4}}\nabla l(t)|^2_2+\int^t_0|w^{-\fr{\nu}{2}}h^{-\fr{1}{4}}l_s|_2^2\text{d}s\leq & M(c_0),\\
\int^t_0(|h^{-\fr{1}{4}}l_s|_2^2+|\sqrt{h} \nabla^2l|^2_2+|\nabla^2l|^2_2)\text{d}s\leq & M(c_0)c^{3\nu}_1.
	\end{aligned}
\end{equation}
\end{lemma}
\begin{proof} 
First, it follows from the fact that $\epsilon\in (0,1]$ and   Lemma \ref{gh} that 
\begin{equation}\label{epsilonzhuanhua}
K^{-1}_0|h^{-\fr{1}{4}}f|^2_2\leq |(h^2+\epsilon^2)^{-\fr{1}{8}}f|^2_2\leq  |h^{-\fr{1}{4}}f|^2_2,
\end{equation}
for general $f$ satisfying $h^{-\fr{1}{4}}f\in L^2$, where the constant $K_0=(1+4c^2_0)^{\frac{1}{4}}$.

Second, it follows from  $\ef{ln}_3$ that
\begin{equation}\label{lt0}\begin{split}
&w^{-\nu}(h^2+\epsilon^2)^{-\frac{1}{4}}( l_t+v\cdot\nabla l)-a_4  (h^2+\epsilon^2)^{\frac{1}{4}} \triangle l\\
= & a_5 \tilde{n}g^{\frac{3}{2}}H(v)+a_6w (h^2+\epsilon^2)^{-\frac{1}{4}} \text{div} \psi+w^{-\nu}\Pi(l,h,w,g).
\end{split}
\end{equation}

Multiplying \ef{lt0} by $l_t$ and integrating over $\mathbb{R}^3$, one can obtain by   integration by parts, \text{H\"older's} inequality, \eqref{2.16}, Lemmas \ref{psi}-\ref{varphi} and  $\text{Young's}$ inequality that
\begin{align}\label{gradinl}
&\fr{a_4}{2}\frac{d}{dt}|(h^2+\epsilon^2)^{\frac{1}{8}}\nabla l|^2_2+|w^{-\fr{\nu}{2}}(h^2+\epsilon^2)^{-\fr{1}{8}}l_t|^2_2\nonumber\\
=&-a_4\int\nabla(h^2+\epsilon^2)^{\frac{1}{4}}\cdot \nabla ll_t-\int w^{-\nu}(h^2+\epsilon^2)^{-\fr{1}{4}}v\cdot\nabla ll_t+a_5\int \tilde{n} g^{\fr{3}{2}}H(v) l_t\nonumber\\
&+a_6\int w (h^2+\epsilon^2)^{-\frac{1}{4}}\text{div}\psi l_t+\int w^{-\nu}\Pi(l,h,w,g)l_t+\fr{1}{4}\int a_4\fr{h h_t}{(h^2+\epsilon^2)^{\frac{3}{4}}}|\nabla l |^2\nonumber\\
\leq &C(|w^{\fr{\nu}{2}}|_\infty|\psi|_\infty
+|w^{-\fr{\nu}{2}}|_\infty|v|_\infty)|\varphi|^{\fr{1}{2}}_\infty
|h^{\frac{1}{4}}\nabla l|_2|w^{-\fr{\nu}{2}}(h^2+\epsilon^2)^{-\fr{1}{8}}l_t|_2\\
&+CK_0\big(|w^{\fr{\nu}{2}}|_\infty|\tilde{n}|_\infty|g^{\fr{3}{4}}\nabla v|_3|g\nabla v|_6|hg^{-1}|_{\infty}^{\fr{1}{4}}+|w^{1+\fr{\nu}{2}}|_\infty(|h^{-\fr{1}{4}}\nabla^2h|_2+|\nabla h^{\fr{3}{8}}|_4^2)\nonumber\\
&+|w^{-1+\fr{\nu}{2}}|_\infty|hg^{-1}|^{\fr{1}{4}}_\infty|g^{\fr{1}{4}}\nabla w|_3|\sqrt{g}\nabla w|_6\big)|w^{-\fr{\nu}{2}}(h^2+\epsilon^2)^{-\fr{1}{8}}l_t|_2\nonumber\\
&+C|\varphi|_\infty|h_t|_\infty|(h^2+\epsilon^2)^{\frac{1}{8}}\nabla l|^2_2\nonumber\\
\leq& M(c_0)c_1^{\nu}c_4^2|(h^2+\epsilon^2)^{\frac{1}{8}}\nabla l|^2_2+M(c_0)c_1^{2+\nu}c_4^{10}+\frac{1}{2}|w^{-\fr{\nu}{2}}(h^2+\epsilon^2)^{-\fr{1}{8}}l_t|_2^2,\nonumber
\end{align}
where one has used the fact that  
\begin{equation}\label{keyvbvbnm}
\begin{split}
|g^{\fr{3}{4}}\nabla v|_3\leq C|\sqrt{g}\nabla v|_2^{\fr{1}{2}}|g\nabla v|_6^{\fr{1}{2}}.
\end{split}
\end{equation}
Then it follows from \eqref{gradinl},  \ef{2.14}, Gronwall's inequality and \eqref{epsilonzhuanhua}, yields that for $0\leq t\leq \min\{T_1,(1+Cc_4)^{-12-\nu}\}$,
\begin{equation*}
\begin{split}
|(h^2+\epsilon^2)^{\frac{1}{8}}\nabla l|^2_2+|h^{\fr{1}{4}}\nabla l|^2_2+\int^t_0(|w^{-\fr{\nu}{2}}(h^2+\epsilon^2)^{-\fr{1}{8}}l_s|^2_2+|w^{-\fr{\nu}{2}}h^{-\fr{1}{4}}l_s|_2^2)\text{d}s\leq & M(c_0).
\end{split}
\end{equation*}
This, together with \eqref{2.16} and  Lemma \ref{gh},  leads to  
\begin{equation}\label{lt2nn}
\begin{split}
|\nabla l|^2_2\leq  M(c_0), \quad  \int^t_0|h^{-\fr{1}{4}}l_s|_2^2\text{d}s\leq  M(c_0)c^\nu_1.
\end{split}
\end{equation}

On the other hand, $\ef{lt0}$ implies  that
\begin{equation}\label{le2}\begin{split}
-a_4\triangle ((h^2+\epsilon^2)^{\frac{1}{4}} (l-\bar{l}))&=-a_4(h^2+\epsilon^2)^{\frac{1}{4}} \triangle l-a_4F(\nabla(h^2+\epsilon^2)^{\frac{1}{4}},l-\bar{l})\\
&=w^{-\nu}\mathcal{E}-a_4F(\nabla(h^2+\epsilon^2)^{\frac{1}{4}},l-\bar{l}),\\
\end{split}
\end{equation}
where
\begin{equation}\begin{split}
\mathcal{E}=&\mathcal{E}(h,v,l,\psi,\tilde{n},g,w)=-(h^2+\epsilon^2)^{-\frac{1}{4}}( l_t+v\cdot\nabla l)
+a_5w^\nu \tilde{n}g^{\frac{3}{2}}H(v)
\\
&+a_6w^{\nu+1} (h^2+\epsilon^2)^{-\frac{1}{4}} \text{div} \psi+\Pi(l,h,w,g),\\
F=&F(\nabla(h^2+\epsilon^2)^{\frac{1}{4}},l-\bar{l})=(l-\bar{l})\triangle (h^2+\epsilon^2)^{\frac{1}{4}} +2\nabla(h^2+\epsilon^2)^{\frac{1}{4}}\cdot\nabla l.
\end{split}
\end{equation}
To derive  the $L^2$ estimate of $\nabla^2 l$ from \ef{le2}, one starts with  the $L^2$ estimates of
\begin{equation*}
 (\mathcal{E},F=F(\nabla(h^2+\epsilon^2)^{\frac{1}{4}},l-\bar{l}))   .
\end{equation*}
It follows from  \ef{2.14},   \ef{2.16},  \ef{2.26d},  \ef{vH2},    \ef{lt2nn} and  Lemmas \ref{psi}-\ref{varphi} that 
\begin{equation}\label{nablal}\begin{split}
|\mathcal{E}|_2\leq& C(|\varphi|^{\fr{1}{4}}_\infty|h^{-\fr{1}{4}}l_t|_2+|\varphi|^{\fr{1}{2}}_\infty|v|_\infty|\nabla l|_2+|w^\nu|_\infty|\tilde{n}|_\infty|g^{\fr{3}{2}}\nabla v\cdot \nabla v|_2\\
&+|w^{\nu+1}|_\infty|\varphi|^{\fr{1}{4}}_\infty|h^{-\fr{1}{4}}\nabla^2h|_2+|w^{\nu+1}|_\infty|\varphi|^{\fr{1}{4}}_\infty|\nabla h^{\fr{3}{8}}|_4^2\\
&+|w^\nu|_\infty|\varphi|^{\fr{1}{2}}_\infty|\psi|_\infty|\nabla l|_2+|w^{\nu-1}|_\infty|\sqrt{g}\nabla w\cdot \nabla w|_2)\\
\leq &M(c_0)(c^{\nu+1}_1+|h^{-\fr{1}{4}}l_t|_2),\\
|F|_2\leq &C \big(|\nabla^2 (h^2+\epsilon^2)^{\frac{1}{4}}|_3|l-\bar{l}|_6
+|\varphi|^{\fr{1}{2}}_\infty|\psi|_\infty|\nabla l|_2\big)
\leq M(c_0),
\end{split}
\end{equation}
where one has used   the facts that  for $0\leq t\leq \min\{T_1,(1+Cc_4)^{-12-\nu}\}$,
\begin{align}
    \label{jichuxinxivbnm}
 |g^{\fr{3}{2}}\nabla v\cdot \nabla v|_2
\leq& |\sqrt{g}(0,x)\nabla u_0|_3|g(0,x)\nabla u_0|_6+t^{\fr{1}{2}}\Big(\int^t_0|(g^{\fr{3}{2}}\nabla v\cdot \nabla v)_s|_2^2\text{d}s\Big)^{\fr{1}{2}}\nonumber\\
\leq&  |\sqrt{g}(0,x)\nabla u_0|_3|g(0,x)\nabla u_0|_6\nonumber\\
&+Ct^{\fr{1}{2}}\Big(\int^t_0\big(|g_s|^2_\infty|\sqrt{g}\nabla v|^2_2|\nabla v|^2_\infty+|\sqrt{g}\nabla v_s|^2_2|g\nabla v|^2_\infty\big)\text{d}s\Big)^{\fr{1}{2}}\nonumber\\
\leq &M(c_0)(1+c_4^{3}t)
\leq M(c_0), \nonumber \\
|g^{\frac{1}{2}}\nabla w\cdot \nabla w|_2
\leq& |\sqrt{g}(0,x)\nabla l_0|_6|\nabla l_0|_3+t^{\fr{1}{2}}\Big(\int^t_0|(\sqrt{g}\nabla w\cdot \nabla w)_s|_2^2\text{d}s\Big)^{\fr{1}{2}}\\
\leq& |\sqrt{h_0}\nabla l_0|_6|\nabla l_0|_3\nonumber\\
&+Ct^{\fr{1}{2}}\Big(\int^t_0\big(|(\sqrt{g})_s|^2_\infty|\nabla w|^2_6|\nabla w|^2_3+|g^{\frac{1}{4}}\nabla w_s|^2_2|g^{\frac{1}{4}}\nabla w|^2_\infty\big)\text{d}s\Big)^{\fr{1}{2}}\nonumber\\
\leq & M(c_0)(1+c_4^{5}t)
\leq M(c_0),\nonumber\\
|\nabla^2 (h^2+\epsilon^2)^{\frac{1}{4}}|_3\leq & C(|\varphi|^{\frac{1}{2}}_\infty|\nabla\psi |_3+|\varphi|_\infty|\nabla h^{\fr{3}{4}}|_6^2)\leq M(c_0).\nonumber
    \end{align}

Then it follows from \eqref{le2}-\eqref{nablal},  Lemma \ref{zhenok} and  Lemmas \ref{psi}-\ref{varphi}  that for 
$0\leq t\leq \min\{T_1,(1+Cc_4)^{-12-\nu}\}$,
\begin{equation}\label{nabla2hl}\begin{split}
|(h^2+\epsilon^2)^{\frac{1}{4}} (l-\bar{l})|_{D^2}\leq& C(|w^{-\nu}\mathcal{E}|_2+|F(\nabla(h^2+\epsilon^2)^{\frac{1}{4}},l-\bar{l})|_2)\\
\leq &M(c_0)(c^{2\nu+1}_1+c^{\nu}_1|h^{-\fr{1}{4}}l_t|_2),\\
|(h^2+\epsilon^2)^{\frac{1}{4}} \nabla^2l|_2\leq &C(|(h^2+\epsilon^2)^{\frac{1}{4}} (l-\bar{l})|_{D^2}+|\nabla^2 (h^2+\epsilon^2)^{\frac{1}{4}}(l-\bar{l})|_2\\
&+|\nabla l\cdot \nabla (h^2+\epsilon^2)^{\frac{1}{4}} |_2)\\
\leq &C(|(h^2+\epsilon^2)^{\frac{1}{4}} (l-\bar{l})|_{D^2}
+|\nabla^2 (h^2+\epsilon^2)^{\frac{1}{4}}|_3|l-\bar{l}|_6\\
&+|\varphi|^{\fr{1}{2}}_\infty|\psi|_\infty|\nabla l|_2)
\leq M(c_0)(c^{2\nu+1}_1+c^{\nu}_1|h^{-\fr{1}{4}}l_t|_2).
\end{split}
\end{equation}

Consequently, this, together with  $\ef{lt2nn}$ and Lemma \ref{varphi}, shows   that $\eqref{lt}_2$ holds
for $0\leq t\leq T_2=\min\{T_1,(1+Cc_4)^{-12-2\nu}\}$.
\end{proof}

\begin{lemma}\label{l2}For $T_3=\min\{T_2,(1+Cc_4)^{-20-8\nu}\}$ and $t\in [0,T_3]$,  it holds that   	\begin{equation}\label{ltt0}\begin{split}
   |w^{-\fr{\nu}{2}}h^{-\frac{1}{4}}l_{t}(t)|^2_2+\int^t_0(|h^{\frac{1}{4}} \nabla l_s|^2_2+|\nabla l_s|^2_2)\text{d}s
   \leq &M(c_0),\\
 |h^{-\frac{1}{4}}l_{t}(t)|_2+|\sqrt{h}\nabla^2l(t)|_2+|l(t)|_{D^2}
\leq& M(c_0)c_1^{2\nu+1},\\
	\int^t_0(|\sqrt{h}\nabla^3l |_2^2+|\sqrt{h}\nabla^2 l|^2_{D^1_*}+|l|^2_{D^3})\text{d}s
\leq& M(c_0)c_1^{2\nu+2}.
	\end{split}\end{equation}
\end{lemma}
\begin{proof}
Applying $\partial_t$ to $\ef{ln}_3$ yields
\begin{align}\label{ltt}
&(h^2+\epsilon^2)^{-\frac{1}{4}}l_{tt}-a_4w^{\nu}  (h^2+\epsilon^2)^{\frac{1}{4}} \triangle l_t\nonumber\\
=&
-((h^2+\epsilon^2)^{-\frac{1}{4}})_tl_t-((h^2+\epsilon^2)^{-\frac{1}{4}}v\cdot\nabla l)_t
+a_4(w^{\nu}) _t(h^2+\epsilon^2)^{\frac{1}{4}} \triangle l\\
&+a_4w^{\nu} ((h^2+\epsilon^2)^{\frac{1}{4}})_t \triangle l+a_5( w^{\nu} \tilde{n}g^{\frac{3}{2}}H(v))_t
+a_6(w^{\nu+1} (h^2+\epsilon^2)^{-\frac{1}{4}} \text{div} \psi)_t+\Pi(l,h,w,g)_t.\nonumber
\end{align}
Multiplying \ef{ltt} by $w^{-\nu}l_t$, integrating over $\mathbb{R}^3$ and integration by parts lead to 
\begin{align}\label{ltt1}
&\fr{1}{2}\fr{d}{dt}|w^{-\fr{\nu}{2}}(h^2+\epsilon^2)^{-\frac{1}{8}}l_{t}|^2_2+a_4|(h^2+\epsilon^2)^{\frac{1}{8}} \nabla l_t|^2_2\nonumber\\
=&
-\int\big((h^2+\epsilon^2)^{-\frac{1}{4}})_tl_t+((h^2+\epsilon^2)^{-\frac{1}{4}}v\cdot\nabla l)_t\big)w^{-\nu}l_t\nonumber\\
&+a_4\int(w^{\nu} _t(h^2+\epsilon^2)^{\frac{1}{4}} \triangle l+
w^{\nu} ((h^2+\epsilon^2)^{\frac{1}{4}})_t \triangle l)w^{-\nu}l_t\\
&+\int\big(a_5( w^{\nu} \tilde{n}g^{\frac{3}{2}}H(v))_t+a_6(w^{\nu+1} (h^2+\epsilon^2)^{-\frac{1}{4}} \text{div} \psi)_t+\Pi(l,h,w,g)_t\big)w^{-\nu}l_t\nonumber\\
&-a_4\int\nabla(h^2+\epsilon^2)^{\frac{1}{4}} \nabla l_t l_t+\fr{1}{2}\int(w^{-\nu}(h^2+\epsilon^2)^{-\frac{1}{4}})_t|l_{t}|^2=\sum^{6}_{i=1}J_i,\nonumber
\end{align}
where $J_i$, $i=1,2,\cdots,6$, are given and estimated as follows:
\begin{align*}
J_1=&-\int\big((h^2+\epsilon^2)^{-\frac{1}{4}})_tl_t+((h^2+\epsilon^2)^{-\frac{1}{4}}v\cdot\nabla l)_t\big)w^{-\nu}l_t\\
\leq &C|\varphi|_\infty|h_t|_\infty|w^{-\fr{\nu}{2}}(h^2+\epsilon^2)^{-\frac{1}{8}}l_{t}|^2_2+C|w^{-\fr{\nu}{2}}|_\infty(|\varphi|_\infty^{\fr{5}{4}}|h_t|_\infty|v|_\infty|\nabla l|_2\\
&+|\varphi|_\infty^{\fr{1}{4}}|v_t|_3|\nabla l|_6+|\varphi|_\infty^{\fr{1}{2}}|v|_\infty|(h^2+\epsilon^2)^{\fr{1}{8}}\nabla l_t|_2)|w^{-\fr{\nu}{2}}(h^2+\epsilon^2)^{-\frac{1}{8}}l_{t}|_2,\\
J_2=&a_4\int((w^{\nu}) _t(h^2+\epsilon^2)^{\frac{1}{4}} \triangle l+
w^{\nu} ((h^2+\epsilon^2)^{\frac{1}{4}})_t \triangle l)w^{-\nu}l_t\\
\leq& CK^{\frac{3}{2}}_0(|w^{-1+\fr{\nu}{2}}|_\infty|hg^{-1}|^\fr{1}{4}_\infty|g^{\frac{1}{4}} w_t|_\infty|\sqrt{h} \nabla^2 l|_2\\
&+|w^{\fr{\nu}{2}}|_\infty|\varphi|_\infty^{\fr{3}{4}}|h_t|_\infty|\sqrt{h}\nabla^2 l|_2)|w^{-\fr{\nu}{2}}(h^2+\epsilon^2)^{-\frac{1}{8}}l_{t}|_2,\\
J_3=&a_5\int( w^{\nu} \tilde{n}g^{\frac{3}{2}}H(v))_tw^{-\nu}l_t\\
\leq& CK^{\frac{1}{2}}_0|hg^{-1}|_{\infty}^{\fr{1}{4}}\Big((|w^{-1+\fr{\nu}{2}}|_\infty|\tilde{n}|_\infty|w_t|_6|g\nabla v|_\infty+|w^{\fr{\nu}{2}}|_\infty|\tilde{n}_t|_\infty|g\nabla v|_6)|g^{\fr{3}{4}}\nabla v|_3\\
&+|w^{\fr{\nu}{2}}|_\infty|\tilde{n}|_\infty(|g_t|_6|g^{\fr{3}{4}}\nabla v|_3|\nabla v|_\infty+|g\nabla v|_6|g^{\fr{3}{4}}\nabla v_t|_3)\Big)|w^{-\fr{\nu}{2}}(h^2+\epsilon^2)^{-\frac{1}{8}}l_{t}|_2,\\
J_4=&a_6\int(w^{\nu+1} (h^2+\epsilon^2)^{-\frac{1}{4}} \text{div} \psi)_tw^{-\nu}l_t\\
\leq & C\big(|w^{\fr{\nu}{2}}|_\infty|\varphi|_\infty^{\fr{1}{4}}|\nabla\psi|_3|w_t|_6+|w^{1+\fr{\nu}{2}}|_\infty(|\varphi|_\infty|h^{-\fr{1}{4}}\nabla^2 h|_2|h_t|_\infty\\
&+|\varphi|_\infty^{\fr{1}{4}}|\nabla\psi_t|_2)\big)|w^{-\fr{\nu}{2}}(h^2+\epsilon^2)^{-\frac{1}{8}}l_{t}|_2,\\
J_5=&\int \Pi(l,h,w,g)_t w^{-\nu}l_t\\
\leq& C\big(|w^{\fr{\nu}{2}}|_\infty|\varphi|_\infty|hg^{-1}|_\infty^{\fr{1}{4}}|g^{\fr{1}{4}}w_t|_6|\nabla h^{\fr{3}{4}}|_6^2\\
&+|w^{1+\fr{\nu}{2}}|_\infty|\varphi|_\infty(|h_t|_\infty|\nabla h^{\fr{3}{8}}|_4^2+|\nabla h^{\fr{3}{4}}|_6|\psi_t|_3)\big)|w^{-\fr{\nu}{2}}(h^2+\epsilon^2)^{-\frac{1}{8}}l_{t}|_2\\
&+CK^{\frac{1}{2}}_0\big(|w^{-1+\fr{\nu}{2}}|_\infty|\nabla l|_6|w_t|_6|\nabla h^{\fr{3}{4}}|_6+|w^{\fr{\nu}{2}}|_\infty(|\varphi|_\infty^{\fr{3}{2}}|\psi|_\infty|h^{\fr{1}{4}}\nabla l|_2|h_t|_\infty\\
&+|\varphi|_\infty^{\fr{1}{2}}|\psi|_\infty|(h^2+\epsilon^2)^{\fr{1}{8}}\nabla l_t|_2
+|\varphi|_\infty^{\fr{3}{4}}|\psi_t|_3|\sqrt{h}\nabla l|_6)\\
&+|w^{-2+\fr{\nu}{2}}|_\infty|hg^{-1}|_\infty^{\fr{1}{4}}|g^{\fr{1}{4}}\nabla w|_3|\sqrt{g}\nabla w|_\infty|w_t|_6\\
&+|w^{-1+\fr{\nu}{2}}|_\infty|hg^{-1}|_\infty^{\fr{1}{4}}(|g^{-1}|_\infty|g_t|_\infty |\sqrt{g}\nabla w|_\infty|g^{\fr{1}{4}}\nabla w|_2\\
&+|\sqrt{g}\nabla w|_\infty|g^{\fr{1}{4}}\nabla w_t|_2)\big)|w^{-\fr{\nu}{2}}(h^2+\epsilon^2)^{-\frac{1}{8}}l_{t}|_2,\\
J_6=&-a_4\int\nabla(h^2+\epsilon^2)^{\frac{1}{4}}\cdot \nabla l_t l_t+\fr{1}{2}\int(w^{-\nu}(h^2+\epsilon^2)^{-\frac{1}{4}})_t|l_{t}|^2\\
\leq &C(|w^{\fr{\nu}{2}}|_\infty|\varphi|_\infty^{\fr{1}{2}}|\psi|_\infty|(h^2+\epsilon^2)^{\fr{1}{8}}\nabla l_t|_2\\
&+|w^{-\fr{\nu}{2}-1}|_\infty|\varphi|^{\fr{1}{2}}_\infty|w_t|_3|(h^2+\epsilon^2)^{\fr{1}{8}}l_t|_6)|w^{-\fr{\nu}{2}}(h^2+\epsilon^2)^{-\frac{1}{8}}l_{t}|_2\\
&+C|\varphi|_\infty|h_t|_\infty|w^{-\fr{\nu}{2}}(h^2+\epsilon^2)^{-\frac{1}{8}}l_{t}|^2_2.
\end{align*}
Integrating \ef{ltt1} over $(\tau,t)$ ($\tau\in (0,t)$), according to     \eqref{2.16}, Lemmas \ref{psi}-\ref{l},  \eqref{keyvbvbnm}, \eqref{nabla2hl},  the facts that 
\begin{equation} \label{hepsilon}
(h^2+\epsilon^2)^{\frac{1}{4}}h^{-\frac{1}{2}}\leq  K_0=(1+4c^2_0)^{\frac{1}{4}},
\end{equation}
\begin{equation} \label{34vt}
|g^{\fr{3}{4}}\nabla v_t|_3\leq C|\sqrt{g}\nabla v_t|_2^{\fr{1}{2}}|g\nabla v_t|_6^{\fr{1}{2}},
\end{equation}
and \eqref{epsilonzhuanhua},  one can obtain that  for $0\leq t\leq T_2$,
\begin{align}
\label{ltt3}
   &K^{-1}_0|w^{-\fr{\nu}{2}}h^{-\frac{1}{4}}l_{t}(t)|^2_2+a_4\int^t_\tau|h^{\frac{1}{4}} \nabla l_s|^2_2\text{d}s\nonumber\\
   \leq & |w^{-\fr{\nu}{2}}h^{-\frac{1}{4}}l_{t}(\tau)|^2_2+M(c_0)c_4^{4\nu+8}\int^t_\tau(|\sqrt{g}\nabla^2 w_s|^{\fr{1}{2}}_2+1)|w^{-\fr{\nu}{2}}h^{-\frac{1}{4}}l_{s}|^2_2\text{d}s\\
   &+M(c_0)c_4^{10+\nu}(t+t^{\fr{1}{2}}).\nonumber
\end{align}

Due to $\ef{ln}_3$, one gets
\begin{equation}\label{ltt4}\begin{split}
 |w^{-\fr{\nu}{2}}(h^2+\epsilon^2)^{-\frac{1}{8}}l_{t}(\tau)|_2\leq&\big(| w^{-\fr{\nu}{2}}(h^2+\epsilon^2)^{\frac{1}{8}}(-(h^2+\epsilon^2)^{-\fr{1}{4}}v\cdot \nabla l\\
 &+a_4 w^{\nu} (h^2+\epsilon^2)^{\frac{1}{4}} \triangle l+ a_5w^{\nu} \tilde{n} g^{\frac{3}{2}}H(v)\\
 &+a_6w^{\nu+1} (h^2+\epsilon^2)^{-\frac{1}{4}}\text{div} \psi+\Pi(l,h,w,g))|_2\big)(\tau),
\end{split}
\end{equation}
which, together with \eqref{epsilonzhuanhua},  the  regularities of $(v,g,w,h,\phi,l)$ shown in \eqref{4.1*}-\eqref{2.13}, \ef{2.14} and Remark  \ref{r1}, implies  that
\begin{equation*}\label{ltt5}\begin{split}
 &\limsup_{\tau\rightarrow 0}|w^{-\fr{\nu}{2}}h^{-\frac{1}{4}}l_t(\tau)|_2 \leq \limsup_{\tau\rightarrow 0}K^{\frac{1}{2}}_0|w^{-\fr{\nu}{2}}(h^2+\epsilon^2)^{-\frac{1}{8}}l_{t}(\tau)|_2\\
 \leq &CK_0\big(|l^{-\fr{\nu}{2}}_0h^{-\fr{1}{4}}_0 u_0\cdot \nabla l_0|_2+ |l^{\fr{\nu}{2}}_0h^{\fr{1}{4}}_0(h^2_0+\epsilon^2)^{\frac{1}{4}} \triangle l_0|_2\\
 &+|l^{\fr{\nu}{2}}_0h^{\fr{1}{4}}_0h_0^b h^{\frac{3}{2}}_0H(u_0)|_2+|l^{1+\fr{\nu}{2}}_0h^{-\frac{1}{4}}_0 \text{div} \psi_0|_2+|l^{-\fr{\nu}{2}}_0h^{\fr{1}{4}}_0\Pi(l_0,h_0,w_0,g_0)|_2\big)\\
 \leq &CK_0|l^{\fr{\nu}{2}}_0|_\infty(|l^{-\nu}_0|_\infty|\phi_0^{-\fr{\iota}{2}}|_\infty|u_0|_\infty|\nabla l_0|_2+|g_5|_2+\epsilon^{\fr{1}{2}}|\phi_0^{-\iota}|_\infty|g_5|_2\\
 &+|\phi_0^{2b\iota }|_\infty|\phi_0^{\fr{3}{2}\iota}\nabla u_0|_3|\phi_0^{2\iota}\nabla u_0|_6)+CK_0|l^{1+\fr{\nu}{2}}_0|_\infty(|\nabla^2 \phi_0^{\fr{3}{2}\iota} |_2+|\nabla\phi_0^{\fr{3}{4}\iota}|_4^2)\\
&+CK_0|l^{\fr{\nu}{2}}_0|_\infty|\nabla\phi_0^{\fr{3}{2}\iota}|_6|\nabla l_0|_3+CK_0|l^{\fr{\nu}{2}-1}_0|_\infty|\phi_0^{\fr{3}{2}\iota}\nabla l_0|_6|\nabla l_0|_3
 \leq M(c_0).
 \end{split}
\end{equation*}

Letting $\tau\rightarrow 0$ in $\ef{ltt3}$, it follows from  Gronwall's inequality  that for $0\leq t\leq \min\{T_2,(1+Cc_4)^{-20-6\nu}\}$,
\begin{equation}\label{ltt6}
\begin{split}
   & |w^{-\fr{\nu}{2}}h^{-\frac{1}{4}}l_{t}(t)|^2_2+\int^t_0(|h^{\frac{1}{4}} \nabla l_s|^2_2+|\nabla l_s|^2_2)\text{d}s
   \leq M(c_0),
\end{split}
\end{equation}
which,  along with \ef{nabla2hl}, yields that for $0\leq t\leq \min\{T_2,(1+Cc_4)^{-20-6\nu}\}$,
\begin{equation}\label{ltt7*}
|(h^2+\epsilon^2)^{\fr{1}{4}}(l-\bar{l})(t)|_{D^2}+|\sqrt{h}\nabla^2l(t)|_2+|l(t)|_{D^2}\leq M(c_0)c_1^{2\nu+1}.
\end{equation}

Next, to derive  the $L^2$-estimates of $\nabla^3 l$, one considers the  $L^2$-estimates of
\begin{equation*}
\big(\nabla\mathcal{E},\nabla\tilde{F}=\nabla F(\nabla(h^2+\epsilon^2)^{\fr{1}{4}},l-\bar{l})\big).
\end{equation*}
Using  Lemmas \ref{psi}-\ref{l}, \ef{2.16},\ef{vH2} and \ef{jichuxinxivbnm},
one can get that 
\begin{align}\label{ltt7}
  |\mathcal{E}|_{D^1_*}\leq & C\big(|\varphi|_\infty^{\fr{1}{2}}|\nabla l_t|_2+|\varphi|_\infty^{\fr{5}{4}}|\psi|_\infty|h^{-\fr{1}{4}}l_t|_2+|\varphi|_\infty^{\fr{3}{2}}|\psi|_\infty|v|_3|\nabla l|_6\nonumber\\
  &+|\varphi|_\infty^{\fr{1}{2}}(|v|_\infty|\nabla^2 l|_2+|\nabla v|_3|\nabla l|_6)+|w^{\nu-1}|_\infty|\tilde{n}|_\infty|\nabla w|_6|g^{\fr{3}{2}}\nabla v\cdot \nabla v|_3\nonumber\\
  &+|w^\nu|_\infty(|\nabla \tilde{n}|_\infty|g^{\fr{3}{2}}\nabla v\cdot \nabla v|_2+|\tilde{n}|_\infty|\nabla g|_\infty|\varphi|_\infty^{\fr{1}{2}}|\nabla v|_3|g\nabla v|_6\nonumber\\
  &+|\tilde{n}|_\infty|g^{\frac{3}{2}}\nabla v\cdot \nabla^2v|_2+|\varphi|_\infty^{\fr{1}{2}}|\nabla w|_6|\nabla\psi|_3)\nonumber\\
&+|w^{\nu+1}|_\infty(|\varphi|_\infty^{\fr{5}{4}}|\nabla h^{\fr{3}{4}}|_6|\nabla\psi|_3
  +|\varphi|_\infty^{\fr{1}{2}}|\nabla^2\psi|_2)\nonumber\\
&+|w^\nu|_\infty|\varphi|_\infty^{\fr{3}{2}}|\psi|_\infty^2|\nabla w|_2+|w^{\nu+1}|_\infty(|\varphi|_\infty^{\fr{5}{4}}|\psi|_\infty|\nabla  h^{\fr{3}{8}}|_4^2\nonumber\\
  &+|\varphi|_\infty^{\fr{5}{4}}|\psi|_\infty|h^{-\fr{1}{4}}\nabla^2h|_2)+|w^{\nu-1}|_\infty|\varphi|_\infty^{\fr{1}{2}}|\psi|_\infty|\nabla l|_6|\nabla w|_3\\
&+|w^\nu|_\infty(|\varphi|_\infty^{\fr{3}{2}}|\psi|_\infty^2|\nabla l|_2+|\varphi|_\infty^{\fr{1}{2}}|\psi|_\infty|\nabla^2 l|_2\nonumber\\
  &+|\varphi|_\infty^{\fr{1}{2}}|\nabla \psi|_3|\nabla l|_6)+|w^{\nu-2}|_\infty|\sqrt{g}\nabla w|_6|\nabla w|_6^2\nonumber\\
  &+|w^{\nu-1}|_\infty(|g^{-\fr{1}{2}}|_\infty|\nabla g|_\infty|\nabla w|_4^2+|\sqrt{g}\nabla w\cdot \nabla^2 w|_2)\big)\nonumber\\
  &\leq M(c_0)(|\nabla l_t|_2+c_1^{3\nu+2}),\nonumber\\
|\tilde{F}|_{D^1_*}\leq & C(|\nabla(h^2+\epsilon^2)^{\fr{1}{4}}|_\infty|\nabla^2l|_2+|\nabla^2(h^2+\epsilon^2)^{\fr{1}{4}}|_3|\nabla l|_6\nonumber\\
&+|\nabla^3(h^2+\epsilon^2)^{\fr{1}{4}}|_2|l-\bar{l}|_\infty)
\leq M(c_0)c_1^{2\nu+1},\nonumber
\end{align}
where one has used  \ef{2.14} and  the facts that, for $0\leq  t\leq \min\{T_2,(1+Cc_4)^{-20-6\nu}\}$,
\begin{align}\label{ltt9vbnmm}
|g\nabla v|_6
\leq & |h_0\nabla u_0|_6+t^{\fr{1}{2}}\Big(\int^t_0|(g\nabla v)_s|_6^2\text{d}s\Big)^{\fr{1}{2}}\nonumber\\
\leq& C(|h_0\nabla^2 u_0|_2+|\psi_0|_\infty|\nabla u_0|_2)\nonumber\\
&+t^{\fr{1}{2}}\Big(\int^t_0(|g_s|_\infty|\nabla v|_6+|g \nabla^2 v_s|_2+|\nabla g|_\infty|\nabla v_s|_2)^2\text{d}s\Big)^{\fr{1}{2}}\nonumber\\
\leq &M(c_0)(1+c_4^2(t+t^{\frac{1}{2}}))
\leq M(c_0),\nonumber\\
|g^{\fr{3}{2}}\nabla v\cdot \nabla^2 v|_2
\leq& |\sqrt{h_0}\nabla u_0|_3|h_0\nabla^2 u_0|_6+t^{\fr{1}{2}}\Big(\int^t_0|(g^{\fr{3}{2}}\nabla v\cdot \nabla^2 v)_s|_2^2\text{d}s\Big)^{\fr{1}{2}}\nonumber\\
\leq&  |\sqrt{h_0}\nabla u_0|_3|h_0\nabla^2 u_0|_6+Ct^{\fr{1}{2}}\Big(\int^t_0\big(|g_s|^2_\infty|\sqrt{g}\nabla v|^2_\infty|\nabla^2 v|^2_2\nonumber \\
&+|g\nabla v_s|^2_6|\sqrt{g}\nabla^2 v|^2_3+|g\nabla^2 v_s|^2_2|\sqrt{g}\nabla v|^2_\infty\big)\text{d}s\Big)^{\fr{1}{2}} \nonumber\\
\leq &M(c_0)(1+c_4^{4}(t+t^{\frac{1}{2}}))
\leq M(c_0),\nonumber   \\
|\sqrt{g}\nabla w|_6
\leq & |\sqrt{h_0}\nabla l_0|_6+t^{\fr{1}{2}}\Big(\int^t_0|(\sqrt{g}\nabla w)_s|_6^2\text{d}s\Big)^{\fr{1}{2}}\nonumber \\
\leq& C(|\sqrt{h_0}\nabla^2 l_0|_2+|(h_0)^{-1}|^{\frac{1}{2}}_\infty|\psi_0|_\infty|\nabla l_0|_2)\\
&+t^{\fr{1}{2}}\Big(\int^t_0(|g^{-1}|^{\frac{1}{2}}_\infty|g_s|_\infty|\nabla w|_6 \nonumber\\
&+|\sqrt{g} \nabla^2 w_s|_2+|g^{-1}|^{\frac{1}{2}}_\infty|\nabla g|_\infty|\nabla w_s|_2)^2\text{d}s\Big)^{\fr{1}{2}}\nonumber \\
\leq &M(c_0)(1+c_4^3(t+t^{\frac{1}{2}}))
\leq M(c_0),\nonumber \\
|g^{\fr{3}{2}}\nabla v\cdot \nabla v|_3\leq & C|g^{\fr{3}{2}}\nabla v\cdot \nabla v|^{\fr{1}{2}}_2|g^{\fr{3}{2}}\nabla v\cdot \nabla v|^{\fr{1}{2}}_6\nonumber \\
\leq & M(c_0)(|g^{\fr{3}{2}}\nabla v\cdot \nabla v|_2+|\nabla g|_\infty|\varphi|^{\fr{1}{2}}_\infty|\nabla v|_3|g\nabla v|_6\nonumber \\
&+|g^{\fr{3}{2}}\nabla v\cdot \nabla^2 v|_2)^{\fr{1}{2}}\leq M(c_0)c_1,\nonumber \\
|\sqrt{g}\nabla w\cdot \nabla^2 w|_2
\leq& |\nabla l_0|_\infty|\sqrt{h_0}\nabla^2 l_0|_2+t^{\fr{1}{2}}\Big(\int^t_0|(\sqrt{g}\nabla w\cdot \nabla^2 w)_s|_2^2\text{d}s\Big)^{\fr{1}{2}}\nonumber \\
\leq&  |\nabla l_0|_\infty|\sqrt{h_0}\nabla^2 l_0|_2+Ct^{\fr{1}{2}}\Big(\int^t_0\big(|g_s|^2_\infty|g^{-1}|_\infty|\nabla w|^2_\infty|\nabla^2 w|^2_2\nonumber \\
&+|\sqrt{g}\nabla w_s|^2_6|\nabla^2 w|^2_3+|\sqrt{g}\nabla^2 w_s|^2_2|\nabla w|^2_\infty\big)\text{d}s\Big)^{\fr{1}{2}}\nonumber \\
\leq &M(c_0)(1+c_4^{4}(t+t^{\frac{1}{2}}))
\leq M(c_0),\nonumber \\
\|w\|_{D^1\cap D^2}
\leq & \|l_0\|_{D^1\cap D^2}+t^{\fr{1}{2}}\Big(\int^t_0\|w_s\|^2_{D^1\cap D^2}\text{d}s\Big)^{\fr{1}{2}}\nonumber \\
\leq &M(c_0)(1+c_4^2(t+t^{\frac{1}{2}}))
\leq M(c_0),\nonumber \\                                           
|\nabla^3(h^2+\epsilon^2)^{\fr{1}{4}}|_2\leq & C(|\varphi|_\infty^{\fr{1}{2}}|\nabla^3h|_2+|\varphi|_\infty^{\fr{5}{4}}|\psi|_\infty|h^{-\fr{1}{4}}\nabla^2 h|_2+|\varphi|_\infty^{\fr{7}{4}}|\nabla h^{\fr{3}{4}}|_6^3) \nonumber\\
\leq & M(c_0).\nonumber
\end{align}

It follows from \ef{le2}, \ef{ltt7*}-\ef{ltt7},  Lemma \ref{zhenok} and Lemmas \ref{psi}-\ref{l} that
\begin{equation}\label{ltt10}\begin{split}
|(h^2+\epsilon^2)^{\fr{1}{4}} (l-\bar{l})(t)|_{D^3}\leq & C(|w^{-\nu}\mathcal{E}|_{D^1_*}+|F(\nabla(h^2+\epsilon^2)^{\fr{1}{4}},l-\bar{l})|_{D^1_*})\\
\leq & C (|w^{-\nu}|_\infty|\mathcal{E}|_{D^1_*}+|\nabla w^{-\nu}|_3|\mathcal{E}|_{6}+|\tilde{F}|_{D^1_*})\\
\leq &M(c_0)(c^{\nu+1}_1|\nabla l_t|_2+c_1^{4\nu+3}),\\
|(h^2+\epsilon^2)^{\fr{1}{4}}\nabla^3 l(t)|_2\leq & C(|(h^2+\epsilon^2)^{\fr{1}{4}} (l(t)-\bar{l})|_{D^3}+|\nabla(h^2+\epsilon^2)^{\fr{1}{4}}|_\infty|\nabla^2l|_2\\
&+|\nabla^2(h^2+\epsilon^2)^{\fr{1}{4}}|_3|\nabla l|_6
+|\nabla^3(h^2+\epsilon^2)^{\fr{1}{4}}|_2|l-\bar{l}|_\infty)\\
\leq &M(c_0)(c^{\nu+1}_1|\nabla l_t|_2+c_1^{4\nu+3}),
\end{split}
\end{equation}
which, along with  \ef{ltt6} and Lemma \ref{gh}, yields  that, for $0\leq t\leq T_3$,
\begin{equation}\label{ltt11}\begin{split}
\int^t_0(|\sqrt{h}\nabla^3 l|_2^2+|\sqrt{h}\nabla^2 l|^2_{D^1_*}+|l|^2_{D^3})\text{d}s
\leq M(c_0)c_1^{2\nu+2}.
\end{split}
\end{equation}

The proof of Lemma \ref{l2} is completed.
 \end{proof}

\begin{lemma}\label{l3}For $T_4=\min\{T_3,(1+Cc_4)^{-40-10\nu}\}$ and $t\in [0,T_4]$,  it holds that   
\begin{equation}\label{ltt2b}\begin{split}
	|h^{\fr{1}{4}}\nabla l_t(t)|^2_2+|\nabla l_{t}(t) |_2^2+\int^t_0|w^{-\fr{\nu}{2}}h^{-\fr{1}{4}} l_{ss}|_2^2ds\leq & M(c_0),\\
  |\sqrt{h}\nabla^3l (t)|_2+|\sqrt{h}\nabla^2 l(t)|_{D^1_*}+|l(t)|_{D^3}\leq &M(c_0)c_1^{4\nu+3},\\
  \int^t_0(|\sqrt{h}\nabla^2 l_s|_2^2+|\nabla^2 l_s|_2^2)\text{d}s
\leq & M(c_0)c^{3\nu}_1.
\end{split}\end{equation}
\end{lemma}
\begin{proof} Multiplying \ef{ltt} by $w^{-\nu}l_{tt}$ and integrating over $\mathbb{R}^3$ yield  
\begin{equation}\label{nabla3l}\begin{split}
&\fr{a_4}{2}\fr{d}{dt}|(h^2+\epsilon^2)^{\frac{1}{8}} \nabla l_t|^2_2+|w^{-\fr{\nu}{2}}(h^2+\epsilon^2)^{-\frac{1}{8}}l_{tt}|^2_2
=\sum^{12}_{i=7}J_i,
\end{split}
\end{equation}
where $J_i$, $i=7,8,\cdots,12$, are given and estimated as follows:
\begin{align}\label{nabla3l2}
J_7=&-\int\big(((h^2+\epsilon^2)^{-\frac{1}{4}})_tl_t+((h^2+\epsilon^2)^{-\frac{1}{4}}v\cdot\nabla l)_t\big)w^{-\nu}l_{tt}\nonumber\\
\leq &C|w^{-\fr{\nu}{2}}|_\infty(|\varphi|_\infty|h_t|_\infty|h^{-\frac{1}{4}}l_{t}|_2+|\varphi|_\infty^{\fr{3}{2}}|h_t|_\infty|v|_\infty|h^{\fr{1}{4}}\nabla l|_2\nonumber\\
&+|\varphi|_\infty^{\fr{1}{4}}|v_t|_3|\nabla l|_6+|\varphi|_\infty^{\fr{1}{2}}|v|_\infty|h^{\fr{1}{4}}\nabla l_t|_2)|w^{-\fr{\nu}{2}}(h^2+\epsilon^2)^{-\frac{1}{8}}l_{tt}|_2,\nonumber\\
J_8=&a_4\int((w^{\nu})_t(h^2+\epsilon^2)^{\frac{1}{4}} \triangle l+
w^{\nu} ((h^2+\epsilon^2)^{\frac{1}{4}})_t \triangle l)w^{-\nu}l_{tt}\nonumber\\
\leq& CK^{\frac{1}{2}}_0|hg^{-1}|_\infty^{\fr{1}{4}}(|w^{-1+\fr{\nu}{2}}|_\infty|g^{\frac{1}{4}} w_t|_\infty|(h^2+\epsilon^2)^{\frac{1}{4}} \nabla^2 l|_2\nonumber\\
&+|w^{\fr{\nu}{2}}|_\infty|\varphi|_\infty^{\fr{3}{4}}|h_t|_\infty|\sqrt{h}\nabla^2 l|_2)|w^{-\fr{\nu}{2}}(h^2+\epsilon^2)^{-\frac{1}{8}}l_{tt}|_2,\nonumber\\
J_9=&a_5\int( w^{\nu} \tilde{n}g^{\frac{3}{2}}H(v))_tw^{-\nu}l_{tt}\nonumber\\
\leq& CK^{\frac{1}{2}}_0|hg^{-1}|_{\infty}^{\fr{1}{4}}\Big((|w^{-1+\fr{\nu}{2}}|_\infty|\tilde{n}|_\infty|w_t|_6+|w^{\fr{\nu}{2}}|_\infty|\tilde{n}_t|_6)|g\nabla v|_\infty|g^{\fr{3}{4}}\nabla v|_3\nonumber\\
&+|w^{\fr{\nu}{2}}|_\infty|\tilde{n}|_\infty(|\nabla v|_\infty|g_t|_6|g^{\fr{3}{4}}\nabla v|_3+|g\nabla v|_6|g^{\fr{3}{4}}\nabla v_t|_3)\Big)|w^{-\fr{\nu}{2}}(h^2+\epsilon^2)^{-\frac{1}{8}}l_{tt}|_2,\nonumber\\
J_{10}=&a_6\int(w^{\nu+1} (h^2+\epsilon^2)^{-\frac{1}{4}} \text{div} \psi)_tw^{-\nu}l_{tt}\nonumber\\
\leq & C(|w^{\fr{\nu}{2}}|_\infty|\varphi|_\infty^{\fr{1}{4}}|\nabla\psi|_3|w_t|_6+|w^{1+\fr{\nu}{2}}|_\infty|\varphi|_\infty|h^{-\fr{1}{4}}\nabla^2 h|_2|h_t|_\infty\\
&+|w^{1+\fr{\nu}{2}}|_\infty|\varphi|_\infty^{\fr{1}{4}}|\nabla\psi_t|_2)|w^{-\fr{\nu}{2}}(h^2+\epsilon^2)^{-\frac{1}{8}}l_{tt}|_2,\nonumber\\
J_{11}=&\int \Pi(l,h,w,g)_t w^{-\nu}l_{tt}\nonumber\\
\leq& CK^{\frac{1}{2}}_0\big(|w^{\fr{\nu}{2}}|_\infty|\varphi|_\infty|hg^{-1}|_\infty^{\fr{1}{4}}|g^{\fr{1}{4}}w_t|_6|\nabla h^{\fr{3}{4}}|_6^2\nonumber\\
&+|w^{1+\fr{\nu}{2}}|_\infty|\varphi|_\infty(|h_t|_\infty|\nabla h^{\fr{3}{8}}|_4^2+|\nabla h^{\fr{3}{4}}|_6|\psi_t|_3)
+|w^{-1+\fr{\nu}{2}}|_\infty|\nabla l|_6|w_t|_6|\nabla h^{\fr{3}{4}}|_6\nonumber\\
&+|w^{\fr{\nu}{2}}|_\infty(|\varphi|_\infty^{\fr{3}{2}}|\psi|_\infty|(h^2+\epsilon^2)^{\fr{1}{8}}\nabla l|_2|h_t|_\infty+|\varphi|_\infty^{\fr{1}{2}}|\psi|_\infty|h^{\fr{1}{4}}\nabla l_t|_2
+|\varphi|_\infty^{\fr{1}{4}}|\psi_t|_3|\nabla l|_6)\nonumber\\
&+|w^{-2+\fr{\nu}{2}}|_\infty|hg^{-1}|_\infty^{\fr{1}{4}}|g^{\fr{1}{4}}\nabla w|_3|\sqrt{g}\nabla w|_\infty|w_t|_6\nonumber\\
&+|w^{-1+\fr{\nu}{2}}|_\infty|hg^{-1}|_\infty^{\fr{1}{4}}(|g^{-1}|_\infty|g_t|_\infty |\sqrt{g}\nabla w|_\infty|g^{\fr{1}{4}}\nabla w|_2\nonumber\\
&+|\sqrt{g}\nabla w|_\infty|g^{\fr{1}{4}}\nabla w_t|_2)\big)|w^{-\fr{\nu}{2}}(h^2+\epsilon^2)^{-\frac{1}{8}}l_{tt}|_2,\nonumber\\
J_{12}=&-a_4\int\nabla(h^2+\epsilon^2)^{\frac{1}{4}} \cdot\nabla l_t l_{tt}+\fr{a_4}{2}\int((h^2+\epsilon^2)^{\frac{1}{4}})_t|\nabla l_{t}|^2\nonumber\\
\leq &C|w^{\fr{\nu}{2}}|_\infty|\varphi|_\infty^{\fr{1}{2}}|\psi|_\infty|(h^2+\epsilon^2)^{\fr{1}{8}}\nabla l_{t}|_2
|w^{-\fr{\nu}{2}}(h^2+\epsilon^2)^{-\frac{1}{8}}l_{tt}|_2\nonumber\\
&+C|\varphi|_\infty|h_t|_\infty|(h^2+\epsilon^2)^{\fr{1}{8}}\nabla l_{t}|^2_2.\nonumber
\end{align}
Integrating \ef{nabla3l} over $(\tau,t)$ and combining \ef{nabla3l2} yield  that for $0\leq t\leq T_3$,
\begin{equation}\label{nabla3l3}\begin{split}
&|h^{\fr{1}{4}}\nabla l_{t} (t)|^2_2+\frac{1}{a_4 K_0}\int^t_\tau|w^{-\fr{\nu}{2}}h^{-\fr{1}{4}}l_{ss}|_2^2\text{d}s\\
\leq & K_0|h^{\fr{1}{4}}\nabla l_{t} (\tau)|^2_2+M(c_0)c_4^{\nu+4}\int^t_0|h^{\fr{1}{4}}\nabla l_{s} |^2_2\text{d}s
+M(c_0)c_4^{5\nu+20}(t+t^{\fr{1}{2}}).
\end{split}
\end{equation}
where one has used the interpolation inequality \eqref{34vt}.   Due to $\ef{ln}_3$, one gets
\begin{equation}\label{nabla3l4}\begin{split}
|(h^2+\epsilon^2)^{\fr{1}{8}}\nabla l_{t} (\tau)|_2\leq &|(h^2+\epsilon^2)^{\fr{1}{8}}\nabla\big(-v\cdot\nabla l\\
&+(h^2+\epsilon^2)^{\fr{1}{4}} \big(a_4w^\nu(h^2+\epsilon^2)^{\frac{1}{4}}\triangle l+a_5w^\nu \tilde{n}g^{\fr{3}{2}}H(v)\\
&+a_6w^{\nu+1}(h^2+\epsilon^2)^{-\fr{1}{4}}\text{div}\psi+\Pi\big)\big)|_2(\tau).
\end{split}
\end{equation}
By \ef{4.1*}, \ef{2.14}, \ef{incc}-\ef{incc1}, \ef{incc*}-\ef{incc**}, \ef{shangchuzhi2}, \ef{shangchuzhi44}, and the  regularities of $(v,g,w,h,\phi,l)$ shown in \eqref{4.1*}-\eqref{2.13}, one has 
\begin{align}\label{iniallt}
&\lim\sup_{\tau\rightarrow 0}|h^{\fr{1}{4}}\nabla l_{t} (\tau)|_2\leq \lim\sup_{\tau\rightarrow 0} |(h^2+\epsilon^2)^{\fr{1}{8}}\nabla l_{t} (\tau)|_2   \nonumber\\
\leq &CK_0^{\frac{1}{2}}(|h_0^{\fr{1}{4}}\nabla(u_0\cdot\nabla l_0)|_2+|h_0^{\fr{1}{4}}\nabla(l^\nu_0(h_0^2+\epsilon^2)^{\frac{1}{2}}\triangle l_0)|_2\nonumber\\
&+|h_0^{\fr{1}{4}}\nabla((h_0^2+\epsilon^2)^\frac{1}{4}l^\nu_0(h_0^2+\epsilon^2)^\frac{b}{2}h_0^\frac{3}{2}H(u_0)|_2\nonumber\\
&+|h_0^{\fr{1}{4}}\nabla(l^{\nu+1}_0\text{div}\psi_0)|_2+|h_0^{\fr{1}{4}}\nabla((h_0^2+\epsilon^2)^\frac{1}{4}\Pi_0)|_2)\nonumber\\
\leq &CK_0^{\frac{5}{2}}\big(|u_0|_{\infty}|\phi_0^{\fr{\iota}{2}}\nabla^2l_0|_2+|\nabla u_0|_\infty|\phi_0^{\fr{\iota}{2}}\nabla l_0|_2+|l_0^\nu|_\infty|\phi_0^{\fr{5}{2}\iota}\nabla^3 l_0|_2\nonumber\\
&
+|\nabla l_0^\nu|_3|\phi_0^{\fr{5}{2}\iota}\nabla^2 l_0|_6
+|l_0^\nu|_\infty|\psi_0|_\infty|\phi_0^{-\iota}|_\infty|\phi_0^{\fr{3}{2}\iota}\nabla^2 l_0|_2\\
&+|\phi_0^{2b\iota }|_\infty(|l_0^{\nu-1}|_\infty|\phi_0^{\fr{\iota}{2}}\nabla l_0|_2|\phi_0^{2\iota}\nabla u_0|_\infty^2+|l_0^\nu|_\infty|\phi_0^{\fr{3}{2}\iota}\nabla u_0|_3|\phi_0^{3\iota}\nabla^2u_0|_6)\nonumber\\
&+|l_0^\nu|_\infty(|\phi_0^{(2b-1)\iota}|_\infty|\nabla \phi_0^{\fr{3}{2}\iota}|_6|\phi_0^{2\iota}\nabla u_0|_6^2+|\nabla\psi_0|_3|\phi_0^{\fr{\iota}{2}}\nabla l_0|_6)\nonumber\\
&+|l_0^{\nu+1}|_\infty|h_0^\fr{1}{4}\nabla^3h_0|_2+|l_0^\nu|_\infty|\phi_0^{-2\iota}|_\infty|\psi_0|_\infty^2|\phi_0^{\fr{\iota}{2}}\nabla l_0|_2\nonumber\\
&+|l_0^{\nu+1}|_\infty(|\phi_0^{-2\iota}|_\infty|\nabla\phi_0^{\fr{3}{2}\iota}|_6^3+|\nabla\phi_0^{\fr{3}{2}\iota}|_6|\nabla\psi_0|_3|\phi_0^{-\iota}|_\infty)\nonumber\\
&+|l_0^{\nu-1}|_\infty|\phi_0^{\fr{\iota}{2}}\nabla l_0|_2|\nabla l_0|_\infty|\psi_0|_\infty+|l_0^{\nu-2}|_\infty|\phi_0^{-2\iota}|_\infty|\phi_0^{\fr{3}{2}\iota}\nabla l_0|_6^3\nonumber\\
&+|l_0^\nu|_\infty(|\psi_0|_\infty|\phi_0^{\fr{\iota}{2}}\nabla^2 l_0|_2+|\nabla\psi_0|_3|\phi_0^{\fr{\iota}{2}}\nabla l_0|_6)\nonumber\\
&+|l_0^{\nu-1}|_\infty(|\psi_0|_\infty|\phi_0^{\fr{\iota}{2}}\nabla l_0|_2|\nabla l_0|_\infty+|\nabla l_0|_3|\phi_0^{\fr{5}{2}\iota}\nabla^2l_0|_6)\big)
\leq M(c_0).\nonumber
  \end{align}
Letting $\tau\rightarrow 0$, one gets from \ef{nabla3l3} and Gronwall's inequality that for $0\leq t\leq T_4=\min\{T_3,(1+Cc_4)^{-10\nu-40}\}$,
\begin{equation}\label{nabla3l5}
  \begin{split}
&|h^{\fr{1}{4}}\nabla l_{t}(t) |_2^2+|\nabla l_{t}(t) |_2^2+\frac{1}{a_4 K_0}\int^t_0|w^{-\fr{\nu}{2}}h^{-\frac{1}{4}}l_{ss}|_2^2\text{d}s\\
\leq & M(c_0)(1+c_4^{5\nu+20}(t+t^{\fr{1}{2}}))\exp(M(c_0)c_4^{\nu+4}t)\leq M(c_0),
  \end{split}
\end{equation}
which, along with \ef{ltt10}, yields
\begin{equation}\label{nabla3l6}
|(h^2+\eps^2)^{\fr{1}{4}}(l-\bar{l})|_{D^3}+|\sqrt{h}\nabla^3l|_2+|\sqrt{h}\nabla^2 l|_{D^1_*}+|\nabla^3 l |_2
\leq M(c_0)c_1^{4\nu+3}.
\end{equation}
Note that \ef{ltt} gives 
\begin{equation}\label{rfvnm}\begin{split}
-a_4\triangle\big((h^2+\epsilon^2)^{\fr{1}{4}}l_t\big) &= -a_4(h^2+\epsilon^2)^{\fr{1}{4}}\triangle l_t-a_4F(\nabla (h^2+\epsilon^2)^{\fr{1}{4}},l_t)\\
&=-w^{-\nu}(h^2+\epsilon^2)^{-\fr{1}{4}}l_{tt}+w^{-\nu}\mathcal{B}-a_4F(\nabla (h^2+\epsilon^2)^{\fr{1}{4}},l_t),\\
\end{split}
\end{equation}
with
\begin{equation}\label{mathcalB}\begin{split}
\mathcal{B}=&-((h^2+\epsilon^2)^{-\fr{1}{4}})_tl_t-((h^2+\epsilon^2)^{-\fr{1}{4}}v\cdot\nabla l)_t
\\
&+a_4(w^{\nu} (h^2+\epsilon^2)^{\frac{1}{4}})_t \triangle l
+a_5( w^{\nu} \tilde{n}g^{\frac{3}{2}}H(v))_t\\
&+a_6(w^{\nu+1} (h^2+\epsilon^2)^{-\fr{1}{4}} \text{div} \psi)_t+\Pi(l,h,w,g)_t.
  \end{split}  
\end{equation}
Next, to derive  the $L^2$-estimates of $\nabla^2 l_t$, one first deals with  the  $L^2$-estimates of 
\begin{equation*}
w^{-\nu}(h^2+\epsilon^2)^{-\fr{1}{4}}l_{tt}\ \   \text{and}\ \  \big(\mathcal{B},\hat{F}=F(\nabla (h^2+\epsilon^2)^{\fr{1}{4}},l_t)\big)
\end{equation*}
by using  \eqref{2.16}, \ef{nabla3l5}-\ef{nabla3l6} and Lemmas \ref{psi}-\ref{l2} as follows:
\begin{align*}
    |w^{-\nu}(h^2&+\epsilon^2)^{-\fr{1}{4}}l_{tt}|_2\leq|w^{-\fr{\nu}{2}}(h^2+\epsilon^2)^{-\frac{1}{4}}h^{\frac{1}{4}}|_\infty|w^{-\fr{\nu}{2}}h^{-\frac{1}{4}}l_{tt}|_2\\
    \leq &M(c_0)c_1^\frac{\nu}{2}|w^{-\fr{\nu}{2}}h^{-\frac{1}{4}}l_{tt}|_2,\\
   |\mathcal B|_2\leq& C\big(|\varphi|_\infty^{\fr{5}{4}}|h_t|_\infty|h^{-\frac{1}{4}}l_{t}|_2 +|\varphi|_\infty^{\fr{3}{4}}|v|_\infty|h^{\fr{1}{4}}\nabla l_t|_2\\
   &+\|\nabla l\|_1(|\varphi|_\infty^{\fr{3}{2}}|h_t|_\infty|v|_\infty+|\varphi|_\infty^{\fr{1}{2}}|v_t|_3)\\
   &+|(w^\nu)_t|_6|(h^2+\epsilon^2)^{\fr{1}{4}}\triangle l|_3+|w^\nu|_\infty|((h^2+\epsilon^2)^{\fr{1}{4}})_t\triangle l|_2\\
&+|w^\nu|_\infty|\tilde{n}_t|_\infty|\sqrt{g}\nabla v|_2|g\nabla v|_\infty+|w^{\nu-1}|_\infty|\tilde{n}|_\infty|w_t|_6|g\nabla v|_\infty|\sqrt{g}\nabla v|_3\\
&+|w^{\nu}|_\infty|\tilde{n}|_\infty(|\sqrt{g}\nabla v|_2|g_t|_\infty|\nabla v|_\infty+|g\nabla v|_\infty|\sqrt{g}\nabla v_t|_2)\\
   &+|w^{\nu}|_\infty  |\varphi|^{\frac{1}{2}}_\infty |w_t|_6|\nabla \psi|_3+|w^{\nu+1}|_\infty(|\varphi|^{\frac{3}{2}}_\infty |h_t|_6|\nabla \psi|_3+|\varphi|^{\frac{1}{2}}_\infty |\nabla \psi_t|_2)\\
&+|w^{\nu}|_\infty|\varphi|_\infty|w_t|_6|\nabla h^{\fr{3}{4}}|_6^2+|w^{1+\nu}|_\infty(|\varphi|^{\frac{5}{4}}_\infty|h_t|_\infty|\nabla h^{\fr{3}{8}}|_4^2\\
   &+|\varphi|^{\frac{3}{2}}_\infty|\psi|_\infty|\psi_t|_2)+|w^{-1+\nu}|_\infty|\nabla l|_3|w_t|_6|\varphi|^{\frac{1}{2}}_\infty|\psi|_\infty\\
&+|w^{\nu}|_\infty(|\varphi|_\infty^{\fr{3}{2}}|\psi|_\infty|\nabla l|_2|h_t|_\infty+|\varphi|_\infty^{\fr{1}{2}}|\psi|_\infty|\nabla l_t|_2
+|\varphi|_\infty^{\fr{1}{2}}|\psi_t|_2|\nabla l|_\infty)\\
&+ |\sqrt{g}\nabla w|_\infty(|w^{-2+\nu}|_\infty|\nabla w|_3|w_t|_6+|w^{-1+\nu}|_\infty|g^{-1}|_\infty|g_t|_\infty|\nabla w|_2)\\
&+|w^{-1+\nu}|_\infty|g^{-1}|_\infty^{\fr{1}{4}}|g^{\fr{1}{4}}\nabla w_t|_2|\sqrt{g}\nabla w|_\infty \big),\\
|\hat{F}|_2\leq &C\big(
|\varphi|^{\frac{5}{4}}_\infty|\psi|^2_\infty| h^{-\fr{1}{4}}l_t|_2+|\varphi|^{\frac{1}{2}}_\infty(|l_t|_6|\nabla \psi|_3+|\psi|_\infty|\nabla l_t|_2)\big),
\end{align*}
which, along with  \ef{ltt}, \eqref{nabla3l5}-\ef{rfvnm},  Lemma \ref{zhenok} and Lemmas \ref{psi}-\ref{l2}, implies  that
\begin{equation}\label{mathcalBFnnnm}
    \begin{split}
|(h^2+\epsilon^2)^{\fr{1}{4}}l_t|_{D^2}\leq & M(c_0)(c_1^\frac{\nu}{2}|w^{-\frac{\nu}{2}}h^{-\frac{1}{4}}l_{tt}|_2+c_4^{5\nu+10}),\\
|(h^2+\epsilon^2)^{\fr{1}{4}}\nabla^2 l_t|_{2}\leq & M(c_0)(|(h^2+\epsilon^2)^{\fr{1}{4}}l_t|_{D^2}
+|\varphi|^{\frac{5}{4}}_\infty|\psi|^2_\infty| h^{-\fr{1}{4}}l_t|_2\\
&+|\varphi|^{\frac{1}{2}}_\infty|l_t|_6|\nabla \psi|_3+|\varphi|^{\frac{1}{2}}_\infty|\psi|_\infty|\nabla l_t|_2)\\
\leq & M(c_0)(c_1^\frac{\nu}{2}|w^{-\frac{\nu}{2}}h^{-\frac{1}{4}}l_{tt}|_2+c_4^{5\nu+10}).
    \end{split}
\end{equation}
Then it follows from \eqref{nabla3l5} and \eqref{mathcalBFnnnm} that for $0\leq t\leq T_4$, $\eqref{ltt2b}_3$ holds.

\end{proof}

Finally, we derive the time weighted estimates for  $l$, which will be used  to show  that the regular solution is actually a classical one. For simplicity, set
\begin{equation*}
\begin{aligned}
H^t(v)=&4\alpha \sum_{i=1}^{3} \partial_{i}v_{i}\partial_{itt}v_{i}+2\beta\text{div}v\text{div}v_{tt}+2\alpha \sum_{i\neq j}^{3} \partial_{i}v_{j}\partial_{itt}v_{j}\\
&+2\alpha\sum_{i>j} (\partial_{itt}v_{j}\partial_{j}v_{i}+\partial_{i}v_{j}\partial_{jtt}v_{i}).
\end{aligned}
\end{equation*}
\begin{lemma}\label{shang5}For $T_5=\min\{T_4
,(1+M(c_0)c_5)^{-40-10\nu}\}$ and $t\in [0,T_5]$,  it holds that \begin{equation}\label{2.70SSS}
	\begin{split}
	t^{\frac{1}{2}}|l_t(t)|_{D^2}+t^{\frac{1}{2}}|\sqrt{h} \nabla^2l_t(t)|_2+t^{\frac{1}{2}}|h^{-\frac{1}{4}}l_{tt}(t)|_2 \leq  M(c_0)c_1^{\fr{\nu}{2}},&\\ \int^t_0s(|l_{ss}|_{D^1_*}^2+|h^{\frac{1}{4}} l_{ss}|_{D^1_*}^2)\text{d}s  \leq  M(c_0),&\\
 \fr{1}{2}c_0^{-1}\leq l(t,x)\leq \fr{3}{2}c_0 \quad \text{for} \quad (t,x)\in [0,T_5]\times \mathbb{R}^3.&
	\end{split}
	\end{equation}
\end{lemma}
\begin{proof}
    Applying $\partial_t$ to $\ef{ltt}$ yields
\begin{align}\label{lttbnm}
&(h^2+\epsilon^2)^{-\fr{1}{4}}l_{ttt}-a_4w^{\nu}  (h^2+\epsilon^2)^{\frac{1}{4}} \triangle l_{tt}\nonumber\\
=&-2((h^2+\epsilon^2)^{-\fr{1}{4}})_tl_{tt}-((h^2+\epsilon^2)^{-\fr{1}{4}})_{tt}l_t-((h^2+\epsilon^2)^{-\fr{1}{4}}v\cdot\nabla l)_{tt}\nonumber\\
&+2a_4(w^{\nu}  (h^2+\epsilon^2)^{\frac{1}{4}} )_t\triangle l_{t}+2a_4(w^{\nu} )_t ((h^2+\epsilon^2)^{\frac{1}{4}} )_t\triangle l\\
&+a_4(w^{\nu})_{tt}(h^2+\epsilon^2)^{\frac{1}{4}} \triangle l+a_4w^{\nu} ((h^2+\epsilon^2)^{\frac{1}{4}})_{tt} \triangle l\nonumber\\
&+a_5( w^{\nu} \tilde{n}g^{\frac{3}{2}}H(v))_{tt}+a_6(w^{\nu+1} (h^2+\epsilon^2)^{-\fr{1}{4}} \text{div} \psi)_{tt}+\Pi(l,h,w,g)_{tt}.\nonumber
\end{align}
Multiplying \ef{lttbnm} by $w^{-\nu}l_{tt}$, integrating over $\mathbb{R}^3$ and integration by part lead to 
\begin{align}\label{ltt1bnm}
&\fr{1}{2}\fr{d}{dt}|w^{-\fr{\nu}{2}}(h^2+\epsilon^2)^{-\fr{1}{8}}l_{tt}|^2_2+a_4|(h^2+\epsilon^2)^{\frac{1}{8}} \nabla l_{tt}|^2_2\nonumber\\
=&-\int\big(2((h^2+\epsilon^2)^{-\fr{1}{4}})_tl_{tt}+((h^2+\epsilon^2)^{-\fr{1}{4}})_{tt}l_t+((h^2+\epsilon^2)^{-\fr{1}{4}}v\cdot\nabla l)_{tt}\big)w^{-\nu}l_{tt}\nonumber\\
&+\int(2a_4(w^{\nu}  (h^2+\epsilon^2)^{\frac{1}{4}} )_t\triangle l_{t}+2a_4(w^{\nu})_t((h^2+\epsilon^2)^{\frac{1}{4}} )_t\triangle l)w^{-\nu}l_{tt}\nonumber\\
&+\int a_4(w^{\nu})_{tt}(h^2+\epsilon^2)^{\frac{1}{4}} \triangle lw^{-\nu}l_{tt}\\
&+\int(a_4w^{\nu} ((h^2+\epsilon^2)^{\frac{1}{4}})_{tt}\triangle l+a_5( w^{\nu} \tilde{n}g^{\frac{3}{2}}H(v))_{tt})w^{-\nu}l_{tt}\nonumber\\
&+\int\big(a_6(w^{\nu+1} (h^2+\epsilon^2)^{-\fr{1}{4}} \text{div} \psi)_{tt}+\Pi(l,h,w,g)_{tt}\big)w^{-\nu}l_{tt}\nonumber\\
&-a_4\int\nabla(h^2+\epsilon^2)^{\frac{1}{4}} \cdot \nabla l_{tt} l_{tt}+\fr{1}{2}\int(w^{-\nu}(h^2+\epsilon^2)^{-\fr{1}{4}})_t|l_{tt}|^2=\sum^{20}_{i=13}J_i,\nonumber
\end{align}
where     $J_i$, $i=13,14,\cdots,20$, are   given and estimated as follows:
\begin{align*}
J_{13}=& -\int\big(2((h^2+\epsilon^2)^{-\fr{1}{4}})_tl_{tt}+((h^2+\epsilon^2)^{-\fr{1}{4}})_{tt}l_t+((h^2+\epsilon^2)^{-\fr{1}{4}}v\cdot\nabla l)_{tt}\big)w^{-\nu}l_{tt} \\
\leq & C|\varphi|_\infty |h_t|_\infty|w^{-\frac{\nu}{2}}(h^2+\epsilon^2)^{-\fr{1}{8}}l_{tt}|_2^2+C\big(|\varphi|^{2}_\infty|h_t|^2_\infty|w^{-\frac{\nu}{2}}|_\infty|h^{-\fr{1}{4}}l_{t}|_2\\
&+|\varphi|_\infty|w^{-\frac{\nu}{2}}|_\infty|h_{tt}|_6|h^{-\frac{1}{4}}l_t|_3
+|\varphi|^{\frac{5}{2}}_\infty|h_t|^2_\infty|w^{-\frac{\nu}{2}}|_\infty|h^{\frac{1}{4}}\nabla l|_2|v|_\infty\\
&+|w^{-\frac{\nu}{2}}|_\infty|\nabla l|_3|\varphi|^{\frac{5}{4}}_\infty|v|_\infty|h_{tt}|_6
+|\varphi|^{\frac{3}{2}}_\infty|h_t|_\infty|w^{-\frac{\nu}{2}}|_\infty|v|_\infty|h^{\frac{1}{4}}\nabla l_t|_2\\
&+|w^{-\frac{\nu}{2}}|_\infty|\nabla l|_\infty(|\varphi|^{\frac{5}{4}}_\infty|h_t|_\infty|v_{t}|_2+|\varphi^{\frac{1}{4}}|_\infty|v_{tt}|_2)\big)|w^{-\fr{\nu}{2}}(h^2+\epsilon^2)^{-\frac{1}{8}}l_{tt}|_2\\
&+C|\varphi|_\infty|w^{-\nu}|_\infty|v_t|_3|h^{\frac{1}{4}}\nabla l_{t}|_2|h^{\frac{1}{4}}l_{tt}|_6\\
&+C|\varphi|^{\frac{1}{2}}_\infty|w^{-\frac{\nu}{2}}|_\infty|v|_\infty|h^{\frac{1}{4}}\nabla l_{tt}|_2|w^{-\fr{\nu}{2}}(h^2+\epsilon^2)^{-\frac{1}{8}}l_{tt}|_2,\\
J_{14}=& \int\big(2a_4(w^{\nu}  (h^2+\epsilon^2)^{\frac{1}{4}} )_t\triangle l_{t}+2a_4(w^{\nu})_t((h^2+\epsilon^2)^{\frac{1}{4}} )_t\triangle l\\
&+a_4(w^{\nu})_{tt}(h^2+\epsilon^2)^{\frac{1}{4}} \triangle l+a_4w^{\nu} ((h^2+\epsilon^2)^{\frac{1}{4}})_{tt} \triangle l\big) w^{-\nu}l_{tt}\\
\leq & CK_0^\frac{1}{2}\big(|\varphi|^{\frac{3}{4}}_\infty|h_t|_\infty|w^{\frac{\nu}{2}}|_\infty|\sqrt{h}\nabla^2 l_t|_2+|w^{\fr{\nu}{2}-1}|_\infty|\varphi|^{\frac{1}{4}}_\infty|w_t|_6|h_t|_\infty|\nabla^2 l|_3\\
&+|w^{\frac{\nu}{2}}|_\infty(|\varphi|^{\frac{5}{4}}_\infty|h_t|^2_\infty|\nabla^2 l|_2+|\varphi|^{\frac{1}{4}}_\infty|h_{tt}|_6|\nabla^2 l|_3)\\
&+|w^{\fr{\nu}{2}-2}|_\infty|hg^{-1}|^{\frac{1}{4}}_\infty|g^{\frac{1}{4}}w_{t}|_6|w_{t}|_6|(h^2+\epsilon^2)^{\frac{1}{4}}\nabla^2 l|_6\big)|w^{-\fr{\nu}{2}}(h^2+\epsilon^2)^{-\frac{1}{8}}l_{tt}|_2\\
&+C|w^{-1}|_\infty(|gh^{-1}|^{\frac{1}{4}}_\infty|g^{-\frac{1}{4}}w_{tt}|_2|(h^2+\epsilon^2)^{\frac{1}{4}}\nabla^2 l|_3\\
&+|\varphi|_\infty^{\fr{1}{4}}|w_t|_3|(h^2+\epsilon^2)^{\frac{1}{4}}\nabla^2l_t|_2)|h^{\frac{1}{4}}l_{tt}|_6,\\
J_{15}=& \int a_5( w^{\nu} \tilde{n}g^{\frac{3}{2}}H(v))_{tt} w^{-\nu}l_{tt}\\
\leq & CK_0^{\frac{1}{2}}\big(|\tilde{n}|_{\infty}|hg^{-1}|^{\frac{1}{4}}_{\infty}|g\nabla v|^2_{\infty}(|w^{\frac{\nu}{2}-2}|_{\infty}|w_t|_6|g^{-\frac{1}{4}}w_{t}|_3+|w^{\frac{\nu}{2}-1}|_{\infty}|g^{-\frac{1}{4}}w_{tt}|_2)\\
&+|w^{\frac{\nu}{2}}|_{\infty}|\varphi|^{\frac{1}{4}}_{\infty}|hg^{-1}|^{\frac{1}{2}}_{\infty}(|\tilde{n}_{tt}|_2|g\nabla v|_{\infty}^2+|\tilde{n}|_{\infty}|g_t|^2_{\infty}|\nabla v|^2_{4})\\
&+|w^{\frac{\nu}{2}}|_{\infty}|hg^{-1}|^{\frac{3}{2}}_{\infty}|g\nabla v|^2_{6}|\varphi|^{\frac{5}{4}}_{\infty}|\tilde{n}|_{\infty}|g_{tt}|_6\\
&+|\varphi|^{\frac{1}{4}}_{\infty}|w^{\fr{\nu}{2}-1}|_{\infty}|w_t|_6|\tilde{n}_t|_{3}|hg^{-1}|^{\frac{1}{2}}_{\infty}|g\nabla v|^2_\infty\\
&+|w^{\frac{\nu}{2}-1}|_{\infty}|g^{-\frac{1}{4}}w_{t}|_2|\tilde{n}|_\infty |hg^{-1}|^{\frac{1}{4}}_{\infty}|g_t|_\infty|g\nabla v|_{\infty}|\nabla v|_{\infty}\\
&+|w^{\frac{\nu}{2}}|_{\infty}|\tilde{n}_t|_{2}|\varphi|^{\frac{1}{4}}_{\infty}|hg^{-1}|^{\frac{1}{2}}_{\infty}|g_t|_\infty|g\nabla v|_{\infty}|\nabla v|_{\infty}\big)|w^{-\frac{\nu}{2}}(h^2+\epsilon^2)^{-\frac{1}{8}}l_{tt}|_2\\
&+\int a_5  \tilde{n}g^{\frac{3}{2}}H^t(v)l_{tt}+C\big(|\tilde{n}|_{\infty}|gh^{-1}|^{\frac{1}{4}}_{\infty}(|g^{\frac{3}{4}}\nabla v_t|_3|\sqrt{g}\nabla v_{t}|_2\\
&+|w^{-1}|_{\infty}|g^{-\frac{1}{4}}w_t|_3|g\nabla v|_{\infty}|\sqrt{g}\nabla v_{t}|_2)\\
&+|\varphi|^{\frac{1}{4}}_{\infty}|\sqrt{g}\nabla v_{t}|_2(|\tilde{n}_t|_{3}|g\nabla v|_{\infty}+|\tilde{n}|_\infty |g_t|_{\infty}|\nabla v|_{3})\big)|h^{\frac{1}{4}}l_{tt}|_6,\\
J_{16}=& \int a_6(w^{\nu+1} (h^2+\epsilon^2)^{-\fr{1}{4}} \text{div} \psi)_{tt}w^{-\nu}l_{tt}\\
\leq &\int a_6 w(h^2+\epsilon^2)^{-\fr{1}{4}}\text{div}\psi_{tt} l_{tt}+ C\big(|w^{\frac{\nu}{2}-1}|_{\infty}|\varphi|^{\frac{1}{4}}_{\infty}|w_t|^2_6|\nabla \psi|_6\\
&+|w^{\frac{\nu}{2}+1}|_{\infty}|\nabla \psi|_3(|\varphi|^{\frac{5}{4}}_{\infty}|h_{tt}|_6+|\varphi|^{\frac{9}{4}}_{\infty}|h_t|_{\infty}|h_t|_6)\\
&+|\varphi|^{\frac{3}{2}}_{\infty}|w^{\frac{\nu}{2}}|_{\infty}|hg^{-1}|^{\frac{1}{4}}_{\infty}|g^{\frac{1}{4}}w_t|_6|h_t|_{\infty}
|\nabla \psi|_3\\
&+|\varphi|^{\frac{5}{4}}_{\infty}|w^{\frac{\nu}{2}+1}|_{\infty}|h_t|_{\infty}|\nabla \psi_t|_2\big)|w^{-\frac{\nu}{2}}(h^2+\epsilon^2)^{-\frac{1}{8}}l_{tt}|_2\\
&+C|gh^{-1}|^{\frac{1}{4}}_{\infty}(|\varphi|^{\frac{1}{2}}_{\infty}|g^{-\frac{1}{4}}w_{tt}|_2|\nabla \psi|_3+|\varphi|^{\frac{1}{2}}_{\infty}|g^{-\frac{1}{4}}w_{t}|_3|\nabla \psi_t|_2)|h^{\frac{1}{4}}l_{tt}|_6,\\
J_{17}=& a_7\int (w^{\nu+1} (h^2+\epsilon^2)^{-\fr{3}{4}}\psi\cdot \psi)_{tt}w^{-\nu}l_{tt}\\
\leq & C\big(|gh^{-1}|^{\frac{1}{4}}_{\infty}|\psi|^2_{\infty}(|w^{\frac{\nu}{2}}|_{\infty}|g^{-\frac{1}{4}}w_{tt}|_2|\varphi|_{\infty}+|w^{\frac{\nu}{2}-1}|_{\infty}|\varphi|_{\infty}|w_t|_6|g^{-\frac{1}{4}}w_t|_3)\\
&+|w^{\frac{\nu}{2}+1}|_{\infty}(|\varphi|_\infty^{\frac{7}{4}}|\nabla h^{\frac{3}{4}}|^2_6|h_{tt}|_6+|\varphi|_{\infty}^2|h_t|^2_{\infty}|\nabla h^{\fr{3}{8}}|_4^2)\\
&+|\varphi|_{\infty}^{\frac{5}{4}}|w^{\frac{\nu}{2}+1}|_{\infty}(|\psi_t|_3|\psi_t|_6+|\psi|_{\infty}|\psi_{tt}|_2)\\
&+|w^{\frac{\nu}{2}}|_{\infty}(|w_t|_6|\varphi|_{\infty}^{\frac{5}{4}}|\psi|_{\infty}|\psi_t|_3+|g^{-\frac{1}{4}}w_t|_2|\varphi|^2_{\infty}|h_t|_{\infty}|\psi|_{\infty}^2|gh^{-1}|^{\fr{1}{4}}_\infty)\\
&+|w^{\frac{\nu}{2}+1}|_{\infty}|\varphi|_{\infty}^{\frac{9}{4}}|h_t|_{\infty}|\psi|_{\infty}|\psi_t|_2\big)|w^{-\frac{\nu}{2}}(h^2+\epsilon^2)^{-\frac{1}{8}}l_{tt}|_2,\\
J_{18}=& a_8\int (w^\nu (h^2+\epsilon^2)^{-\fr{1}{4}} \nabla l\cdot   \psi)_{tt}w^{-\nu}l_{tt}\\
\leq &C\big(|\nabla l|_{\infty}|\psi|_{\infty}|gh^{-1}|_{\infty}^{\frac{1}{4}}(|w^{\frac{\nu}{2}-1}|_{\infty}|g^{-\fr{1}{4}}w_{tt}|_2+|w^{\frac{\nu}{2}-2}|_{\infty}|g^{-\frac{1}{4}}w_{t}|_3|w_t|_6)\\
&+|w^{\frac{\nu}{2}}|_{\infty}|\psi|_{\infty}(|\varphi|_{\infty}^{\frac{5}{4}}|h_{tt}|_6|\nabla l|_3+|\varphi|_{\infty}^{\frac{5}{2}}|h_t|_{\infty}^2|h^{\frac{1}{4}}\nabla l|_2)\\
&+|w^{\frac{\nu}{2}}|_{\infty}(|\varphi|_{\infty}^{\frac{1}{2}}|\psi|_{\infty}|h^{\frac{1}{4}}\nabla l_{tt}|_2+|\varphi|_{\infty}^{\frac{1}{4}}|\nabla l|_{\infty}|\psi_{tt}|_2)\\
&+|w^{\frac{\nu}{2}-1}|_{\infty}|g^{-\frac{1}{4}}w_{t}|_3|gh^{-1}|^{\frac{1}{4}}_{\infty}|\nabla l|_{\infty}(|\varphi|_{\infty}|h_t|_{6}|\psi|_{\infty}+|\psi_t|_6)\\
&+|w^{\frac{\nu}{2}}|_{\infty}|\varphi|_{\infty}^{\frac{3}{2}}|h_t|_{\infty}(|\psi|_{\infty}|h^{\frac{1}{4}}\nabla l_t|_2+|h^{\frac{1}{4}}\nabla l|_3|\psi_t|_6)\big)|w^{-\frac{\nu}{2}}(h^2+\epsilon^2)^{-\frac{1}{8}}l_{tt}|_2\\
&+C(|w^{-1}|_{\infty}|g^{-\frac{1}{4}}w_{t}|_3|gh^{-1}|^{\frac{1}{4}}_{\infty}|\varphi|_{\infty}^{\frac{3}{4}}|h^{\frac{1}{4}}\nabla l_t|_2|\psi|_{\infty}\\
&+|\varphi|_{\infty}|h^{\frac{1}{4}}\nabla l_t|_2|\psi_t|_3)|h^{\frac{1}{4}}l_{tt}|_6,\\
J_{19}=& a_9\int (w^{\nu-1} \sqrt{g}\nabla w\cdot \nabla w)_{tt}w^{-\nu}l_{tt}\\
\leq & CK_0^{\frac{1}{2}}(|\sqrt{g}\nabla w|^2_{\infty}|hg^{-1}|^{\frac{1}{4}}_{\infty}(|w^{\frac{\nu}{2}-2}|_{\infty}|g^{-\frac{1}{4}}w_{tt}|_2+|w^{\frac{\nu}{2}-3}|_{\infty}|g^{-\frac{1}{4}}w_t|_3|w_t|_6)\\
&+|w^{\frac{\nu}{2}-1}|_{\infty}|hg^{-1}|^{\frac{1}{4}}_{\infty}(|g^{-\fr{1}{4}}|_{\infty}|g_{tt}|_6|\nabla w|_6^2+|g^{-1}|_{\infty}^{\frac{5}{4}}|g_t|_{\infty}^2|\nabla w|_3|\nabla w|_6)\\
&+|w^{\frac{\nu}{2}-1}|_{\infty}|hg^{-1}|^{\frac{1}{4}}_{\infty}|\sqrt{g}\nabla w|_{\infty}|g^{\frac{1}{4}}\nabla w_{tt}|_2\\
&+|w^{\frac{\nu}{2}-2}|_{\infty}|hg^{-1}|^{\frac{1}{4}}_{\infty}(|g^{-\frac{1}{4}} w_{t}|_2|g_t|_{\infty}|\nabla w|^2_{\infty}+|g^{\frac{1}{4}} w_{t}|_6|\nabla w_t|_3|\sqrt{g}\nabla w|_{\infty})\\
&+|w^{\frac{\nu}{2}-1}|_{\infty}|hg^{-1}|^{\frac{1}{4}}_{\infty}|g^{-1} |_{\infty}|g_t|_{\infty}|\sqrt{g}\nabla w|_{\infty}|g^\fr{1}{4}\nabla w_t|_2)|w^{-\frac{\nu}{2}}(h^2+\epsilon^2)^{-\frac{1}{8}}l_{tt}|_2\\
&+C |w^{-1}|_{\infty}|gh^{-1}|^{\frac{1}{4}}_{\infty}|g^{\frac{1}{4}}\nabla w_t|_2|\nabla w_t|_3 |h^{\frac{1}{4}}l_{tt}|_6,\\
J_{20}=& -a_4\int\nabla(h^2+\epsilon^2)^{\frac{1}{4}} \cdot \nabla l_{tt} l_{tt} + \fr{1}{2}\int(w^{-\nu}(h^2+\epsilon^2)^{-\frac{1}{4}})_t|l_{tt}|^2\\
\leq & C(|\varphi|^{\frac{1}{2}}_\infty|w^{\frac{\nu}{2}}|_\infty|\psi|_\infty|h^{\frac{1}{4}}\nabla l_{tt}|_2+ |\varphi|_\infty |h_t|_\infty|w^{-\fr{\nu}{2}}(h^2+\epsilon^2)^{-\frac{1}{8}}l_{tt}|_2\\
&+|gh^{-1}|^\fr{1}{4}_\infty|w^{-\fr{\nu}{2}-1}|_\infty|g^{-\fr{1}{4}}w_t|_3|\varphi|_\infty^{\fr{1}{4}}|h^{\fr{1}{4}}l_{tt}|_6)|w^{-\fr{\nu}{2}}(h^2+\epsilon^2)^{-\frac{1}{8}}l_{tt}|_2.
\end{align*}
To finish the estimates on $J_{15}$ and $J_{16}$, one can integrate by parts to get  
\begin{equation*}
\begin{aligned}
\int   \tilde{n}g^{\frac{3}{2}}H^t(v)l_{tt}\leq &C|w^{\frac{\nu}{2}}|_{\infty}|g\nabla v|_\infty|v_{tt}|_2( |\tilde{n}|_{\infty}|\nabla g|_{\infty}|hg^{-1}|^{\frac{1}{4}}_{\infty}
|g^{-1}|^{\frac{1}{4}}_{\infty}\\
&+|\psi|_\infty|\varphi|^{\frac{1}{4}-b}_\infty|gh^{-1}|^{\frac{1}{2}}_{\infty})|w^{-\frac{\nu}{2}}(h^2+\epsilon^2)^{-\frac{1}{8}}l_{tt}|_2\\
&+C|n|_\infty|g^{\frac{1}{4}}v_{tt}|_3|gh^{-1}|^{\frac{1}{4}}_{\infty}(|g\nabla^2 v|_2|h^{\frac{1}{4}}l_{tt}|_6+|g\nabla v|_6|h^{\frac{1}{4}}\nabla l_{tt}|_2),\\
\int  w (h^2+\epsilon^2)^{-\frac{1}{4}}  \text{div}\psi_{tt} l_{tt}=&-\int (\nabla w (h^2+\epsilon^2)^{-\frac{1}{4}} + w \nabla (h^2+\epsilon^2)^{-\frac{1}{4}} )\cdot \psi_{tt}l_{tt}\\
&-\int w (h^2+\epsilon^2)^{-\frac{1}{4}} \psi_{tt}\cdot \nabla l_{tt}\\
\leq &C( |w^{\frac{\nu}{2}}|_{\infty}|\varphi|^{\frac{1}{4}}_{\infty}|\nabla w|_{\infty}+|w^{\frac{\nu}{2}+1}|_{\infty}|\varphi|^{\frac{5}{4}}_{\infty}|\psi|_{\infty})\\
&\cdot |\psi_{tt}|_2|w^{-\frac{\nu}{2}}(h^2+\epsilon^2)^{-\frac{1}{8}}l_{tt}|_2
+C|w|_{\infty}|\varphi|^{\frac{3}{4}}_{\infty}|\psi_{tt}|_2|h^{\frac{1}{4}}\nabla l_{tt}|_2.
\end{aligned}
\end{equation*}

Multiplying $\ef{ltt1bnm}$ by $t$ and integrating over $(\tau,t)$, one can obtain from above estimates on $J_i$ ($i=13,...,20$), \eqref{2.16} and Lemmas \ref{psi}-\ref{l3} that 
\begin{equation}\label{2.73lll}
\begin{split}
&K_0^{-1}t|w^{-\fr{\nu}{2}}h^{-\frac{1}{4}}l_{tt}|^2_2+\fr{a_4}{4}\int^t_\tau s|(h^2+\epsilon^2)^{\frac{1}{8}} \nabla l_{ss}|^2_2\text{d}s\\
\leq & \tau|w^{-\fr{\nu}{2}}h^{-\frac{1}{4}}l_{tt}(\tau)|^2_2+M(c_0)(c^{14+8\nu}_5t+1)+M(c_0)c^{14+3\nu}_4\int^t_\tau s|w^{-\fr{\nu}{2}}h^{-\fr{1}{4}}l_{ss}|^2_2\text{d}s,
\end{split}
\end{equation}
where one has used the inequality
\begin{equation*}
\begin{split}
|g^{\fr{1}{4}} v_{tt}|_3\leq C| v_{tt}|_2^{\fr{1}{2}}|\sqrt{g} v_{tt}|_6^{\fr{1}{2}}.
\end{split}
\end{equation*}

Note that due to $\ef{nabla3l3}$,  there exists a sequence $s_k$ such that
\begin{equation*}
s_k\longrightarrow 0, \quad \text{and} \quad s_k|w^{-\fr{\nu}{2}}h^{-\fr{1}{4}}l_{tt}(s_k,x)|^2_2\longrightarrow 0,\quad \text{as}\quad k\longrightarrow \infty.
\end{equation*}
Taking $\tau=s_k$ and letting $k\rightarrow \infty$ in $\ef{2.73lll}$, one  gets   by  Gronwall's inequality that 
\begin{equation}\label{2.74lll}
\begin{split}
t|w^{-\fr{\nu}{2}}h^{-\frac{1}{4}}l_{tt}|^2_2+\fr{a_4}{4}\int^t_0 s|(h^2+\epsilon^2)^{\frac{1}{8}} \nabla l_{ss}|^2_2\text{d}s+\int^t_0s|\nabla l_{ss}|^2_2\text{d}s\leq M(c_0),
\end{split}
\end{equation}
for $0\leq t\leq \min\{T_4,(1+Cc_5)^{-40-10\nu}\}$.

which, along with  $\ef{mathcalBFnnnm}$, yields  that
\begin{equation}\label{2.75lll}
t^{\fr{1}{2}}|h^{-\fr{1}{4}}l_{tt}(t)|_2+t^{\fr{1}{2}}|\nabla^2l_t(t)|_2+t^{\fr{1}{2}}|\sqrt{h}\nabla^2l_t(t)|_2\leq  M(c_0)c_1^{\fr{\nu}{2}}.
\end{equation}

Finally, for  $0\leq t\leq T_5$, due to  \ef{nabla3l5} and \ef{2.75lll}, one has 
\begin{equation}\label{luinfty}\begin{split}    
   |l|_\infty=& |l_0+\int^t_0 l_s\text{d}s|_\infty
   \leq  |l_0|_\infty+C\int_0^t |\nabla l_s|_2^{\fr{1}{2}}|\nabla^2 l_s|_2^{\fr{1}{2}} \text{d}s
    \leq \fr{3}{2} c_0,\\
   l=&l_0+\int^t_0 l_s\text{d}s\geq l_0-C\int_0^t |\nabla l_s|_2^{\fr{1}{2}}|\nabla^2 l_s|_2^{\fr{1}{2}} \text{d}s 
    \geq \fr{1}{2} c_0^{-1}.
   \end{split}
\end{equation}

\end{proof}
\subsubsection{The existence and the related a  priori estimates of  $u$}
According to the classical well-posedness theory for the linear  parabolic equations (see \cite{oar}),   there exists a unique strong solution $u$ to the following  sub-problem  of  $\ef{ln}$  in $[0,T_5]\times \mathbb{R}^3$:
\begin{equation}\label{sub-ln2}\left\{\begin{aligned}
&u_t+v\cdot \nabla v+a_1\phi\nabla l+l\nabla\phi+a_2l^\nu\sqrt {h^2  +\epsilon^2} Lu\\
=&a_2g\nabla l^\nu\cdot Q(v)+a_3l^\nu \psi \cdot Q(v),\\
&u|_{t=0}=u_0\quad  \text{for} \quad x\in\mathbb{R}^3,\\
&u\rightarrow 0 \quad \text{as} \hspace{2mm}|x|\rightarrow \infty \quad \rm for\quad t\geq 0.
\end{aligned}\right.\end{equation}
Next, we show that such a  $T_5$ only depends on the fixed parameters $(A, R, c_v, \alpha, \beta, \gamma, \delta,\daleth)$ and  $(\phi_0,u_0,l_0)$. by deriving some uniform a priori estimates for $u$.   Set
\begin{equation*}
\begin{split}
\mathcal{K}=&v\cdot \nabla v+a_1\phi\nabla l+l\nabla\phi+a_2\sqrt {h^2+\epsilon^2}l^\nu Lu -a_2\nabla l^\nu\cdot gQ(v)-a_3 l^\nu\psi\cdot Q(v),\\
\mathcal{H}=&-u_t-v\cdot\nabla v-l\nabla\phi-a_1\phi\nabla l+a_2g\nabla l^\nu \cdot Q(v)+a_3 l^\nu \psi\cdot Q(v).
\end{split}
\end{equation*}
\begin{lemma}\label{lu}
\rm {For} $t\in [0, T_5]$, it holds that 
	\begin{equation}\label{2.61}\begin{aligned}
	|\sqrt{h}\nabla u(t)|^2_2+\|u(t)\|_1^2+\int^t_0(\|\nabla u\|_1^2+|u_s|^2_2)\text{d}s\leq & M(c_0),\\
(|u|^2_{D^2}+|h\nabla^2u|^2_2+|u_t|^2_2)(t)+\int^t_0(|u|_{D^3}^2+|u_s|^2_{D^1_*})\text{d}s\leq & M(c_0).
	\end{aligned}\end{equation}
\end{lemma}
\begin{proof}
   First, one estimates $\|u\|_1$. Multiplying $\ef{ln}_2$ by $l^{-\nu} u$ and $l^{-\nu} u_t$ respectively,  and integrating over $\mathbb{R}^3$, one can obtain by integration by parts,  Lemmas \ref{phiphi}-\ref{shang5} and \ref{lem2as}, \text{H\"older's} and \text{Young's} inequalities  that  
\begin{equation}\label{2.61b}
\begin{split}
&\fr{1}{2}\fr{d}{dt}|l^{-\fr{\nu}{2}}u|_2^2+a_2\alpha|(h^2+\epsilon^2)^{\fr{1}{4}}\nabla u|^2_2+a_2(\alpha+\beta)|(h^2+\epsilon^2)^{\fr{1}{4}}\text{div} u|^2_2 \\
\leq& M(c_0)(1+|l_t|^{\frac{1}{2}}_{D^2})|l^{-\fr{\nu}{2}}u|^2_2+M(c_0)c_4^4+\fr{1}{2}a_2\alpha|\sqrt{h}\nabla u|^2_2,\\
& \fr{1}{2}\fr{d}{dt}(a_2\alpha|(h^2+\epsilon^2)^{\fr{1}{4}}\nabla u|^2_2+a_2(\alpha+\beta)|(h^2+\epsilon^2)^{\fr{1}{4}}\text{div} u|^2_2)+|l^{-\fr{\nu}{2}}u_t|^2_2 \\
\leq& M(c_0)c_4^2|\sqrt{h}\nabla u|_2^2+M(c_0)c_4^4+\fr{1}{2}|l^{-\fr{\nu}{2}}u_t|^2_2,
\end{split}
\end{equation}
which\footnote{When  $\daleth=0$ in \ef{ln}, in Lemma 3.5 of  \cite{dxz}, one can obtain the $L^\infty$ boundedness of $l_t$ by the transport structure of \eqref{transport}, while it fails here  due to the appearance of the singular dissipation. }, 
along with $\text{Gronwall's} $ inequalty and Lemma \ref{l3}, yields  that for $0\leq t\leq T_5$,
\begin{equation}\label{ul2}
\begin{split}
&\|u\|_1^2+|\sqrt{h}\nabla u|^2_2+\int^t_0(|\sqrt{h}\nabla u|^2_2+|u_s|^2_2\big)\text{d}s
\leq 
 M(c_0).
\end{split}
\end{equation}

Notice that $u$ solves the following elliptic equations
\begin{equation}\label{2.61h}
\begin{split}
&a_2L(\sqrt{h^2+\epsilon^2}u) =l^{-\nu}\mathcal{H}-a_2G(\nabla\sqrt{h^2+\epsilon^2},u).
\end{split}
\end{equation}

Thus to derive   the $L^2$ estimate of $\nabla^2 u$, it is sufficient to get  the $L^2$ estimates of $$(\mathcal{H},\tilde{G}=G(\nabla\sqrt{h^2+\epsilon^2},u)).$$
It follows  from \eqref{Gdingyi}, \eqref{2.16}, \ef{jichuxinxivbnm}, \ef{ltt9vbnmm},  \eqref{ul2} and  Lemmas \ref{phiphi}-\ref{shang5} that 
\begin{equation}\label{2.61illl}
\begin{split}
|\mathcal{H}|_2\leq & C(|u_t|_2+|v|_6|\nabla v|_3+|l|_\infty|\nabla \phi|_2+|\phi|_\infty|\nabla l|_2
+|\nabla l|_3|l^{\nu-1}|_\infty | g \nabla v|_6\\
&+|l^\nu |_\infty|\psi|_\infty|\nabla v|_2)
\leq  M(c_0)(|u_t|_2+1),\\
|\tilde{G}|_2
\leq & C(|\nabla\sqrt{h^2+\epsilon^2}|_\infty |\nabla u|_2+|\nabla^2 \sqrt{h^2+\epsilon^2}|_3 |u|_6)\leq M(c_0),
\end{split}
\end{equation}
where one  also has used the facts that 
\begin{equation}\label{hl2}
\begin{split}
\|l\|_{D^2}
\leq  \|l_0\|_{ D^2}+t^{\fr{1}{2}}\Big(\int^t_0\|l_s\|^2_{ D^2}\text{d}s\Big)^{\fr{1}{2}}
\leq M(c_0)(1+c_1^{2\nu}t^{\frac{1}{2}})
\leq & M(c_0),\\
|\nabla^2\sqrt{h^2+\epsilon^2}|_3 \leq      C(|\varphi|^{\frac{1}{2}}_\infty|\nabla h^{\frac{3}{4}}|^2_6+|\nabla \psi|_3)\leq & M(c_0).
\end{split}
\end{equation}

Then it follows from \eqref{ul2}-\eqref{2.61illl},  Lemma \ref{zhenok} and  Lemmas \ref{psi}-\ref{gh}  that 
\begin{equation}\label{2.61i}
\begin{split}
|\sqrt{h^2+\epsilon^2}u|_{D^2}\leq& C(|l^{-\nu}\mathcal{H}|_2+|G(\nabla\sqrt{h^2+\epsilon^2},u)|_2)
\leq  M(c_0)(|u_t|_2+1),
\end{split}
\end{equation}
\begin{equation}\label{2.61ijjj}
\begin{split}
|\sqrt{h^2+\epsilon^2}\nabla^2u|_2\leq& C( |\sqrt{h^2+\epsilon^2}u|_{D^2} +|\nabla\psi|_3|u|_6+|\psi|_\infty|\nabla u|_2  \\
&+ |\psi|^2_\infty|u|_2|\varphi|_\infty) 
\leq  C |\sqrt{h^2+\epsilon^2}u|_{D^2} +M(c_0),
\end{split}
\end{equation}
which, along with \eqref{ul2}, yields $\eqref{2.61}_1$.

Next one estimates $|u|_{D^2}$. Applying $\partial_t$ to $\ef{ln}_2$ yields
\begin{equation}\label{2.61j}\begin{split}
 &u_{tt}+a_2l^{\nu}\sqrt{h^2+\epsilon^2}Lu_t+(v\cdot \nabla v)_t+(l\nabla\phi)_t+a_1(\phi\nabla l)_t\\
=&-a_2(l^{\nu}\sqrt{h^2+\epsilon^2})_tLu+(a_2g\nabla l^\nu\cdot Q(v)+a_3l^\nu\psi \cdot Q(v))_t.
\end{split}
\end{equation}
Multiplying $\ef{2.61j}$ by $l^{-\nu}u_t$, integrating over $\mathbb{R}^3$ and integration by parts lead to 
\begin{align}\label{2.61l}
&  \fr{1}{2}\fr{d}{dt}|l^{-\fr{\nu}{2}}u_{t}|^2_2+a_2\alpha|(h^2+\epsilon^2)^{\fr{1}{4}}\nabla u_t|_2^2+a_2(\alpha+\beta)|(h^2+\epsilon^2)^{\fr{1}{4}}\text{div} u_t|_2^2\nonumber\\
=&\int l^{-\nu}\Big(-(v\cdot \nabla v)_t-(l \nabla \phi)_t-a_1( \phi\nabla l)_t-a_2(l^\nu\sqrt{h^2+\epsilon^2})_tLu\nonumber\\
&+(a_2g\nabla l^\nu\cdot Q(v)+a_3l^\nu\psi\cdot Q(v))_t\Big)\cdot u_t\nonumber\\
&-\int a_2\Big(\alpha(\nabla\sqrt{h^2+\epsilon^2}\cdot\nabla)  u_t+(\alpha+\beta)\nabla\sqrt{h^2+\epsilon^2}\text{div}u_t\Big)\cdot u_t+\frac{1}{2}\int (l^{-\nu})_t|u_t|^2\nonumber\\
\leq& C|l^{-\frac{\nu}{2}}|_\infty(|v|_\infty|\nabla v_t|_2+|v_t|_2|\nabla v|_\infty+|l_t|_6|\nabla\phi|_3+|l|_\infty|\nabla\phi_t|_2\\
&+|\nabla l_t|_2|\phi|_\infty+|\phi_t|_\infty|\nabla l|_2)|l^{-{\frac{\nu}{2}}}u_t|_2+C|l^{-1}|_\infty|l_t|_6|\sqrt{h^2+\epsilon^2}\nabla^2u|_2|u_t|_3\nonumber\\
&+C|l^{\frac{\nu}{2}}|_\infty|h_t|_\infty|\nabla^2u|_2|l^{-\frac{\nu}{2}}u_t|_2
+C\Big(|l^{\frac{\nu}{2}-2}|_\infty|g\nabla v|_\infty|l_t|_6|\nabla l|_3\nonumber\\
&+|l^{\frac{\nu}{2}-1}|_\infty(|g_t|_\infty|\nabla v|_\infty|\nabla l|_2+|g\nabla v|_\infty|\nabla l_t|_2+|\psi|_\infty|l_t|_6|\nabla v|_3)\nonumber\\
&+|l^{\frac{\nu}{2}}|_\infty(|\psi_t|_2|\nabla v|_\infty+|\psi|_\infty
|\nabla v_t|_2)\Big)|l^{-\frac{\nu}{2}}u_t|_2\nonumber\\
&+C|l^{-1}|_\infty|gh^{-1}|^{\frac{1}{2}}_\infty|\sqrt{g}\nabla v_t|_2|\nabla l|_3|\sqrt{h}u_t|_6\nonumber\\
&+C|l^{\frac{\nu}{2}}|_\infty|\varphi|^{\fr{1}{2}}_\infty|\psi|_\infty|\sqrt{h}\nabla u_t|_2|l^{-\frac{\nu}{2}}u_t|_2+C|l^{-\frac{\nu}{2}-1}|_\infty|l_t|_6|u_t|_3|l^{-\frac{\nu}{2}}u_t|_2,\nonumber
\end{align}
Integrating $\ef{2.61l}$ over $(\tau,t)$ $(\tau\in(0,t))$,  one can get by using \ef{2.16}, Lemmas \ref{phiphi}-\ref{shang5}  and \text{Young's} inequality that
\begin{equation}\label{2.61n}
\begin{split}
&\fr{1}{2}|l^{-\fr{\nu}{2}}u_t(t)|^2_2+\fr{a_2\alpha}{2}\int^t_\tau|\sqrt{h}\nabla u_s|^2_2\text{d}s   \\
 \leq & \fr{1}{2}|l^{-\fr{\nu}{2}}u_t(\tau)|^2_2+M(c_0)c_4^{2}\int^t_0|l^{-\fr{\nu}{2}}u_s|^2\text{d}s+M(c_0)c_4^{4+2\nu}t+M(c_0).
\end{split}
\end{equation}
 It follows from  $\ef{ln}_2$,  $\ef{2.14}$, $\ef{incc}$, \ef{incc*}, and the  regularities of $(v,g,w,h,\phi,l)$ shown in \eqref{4.1*}-\eqref{2.13} that
\begin{equation*}\label{2.611p}\begin{split}
\lim\sup_{\tau\rightarrow 0}|u_t(\tau)|_2
\leq & C(|u_0|_\infty|\nabla u_0|_2+|\phi_0|_\infty|\nabla l_0|_2+|\nabla\phi_0|_2|l
_0|_\infty+|\psi_0|_\infty|l_0^\nu|_\infty|\nabla u_0|_2 \\
&+|l_0^\nu|_\infty(|g_2|_2+|Lu_0|_2)+|l_0^{\nu-1}|_\infty|\phi_0^{2\iota}\nabla u_0|_\infty|\nabla l_0|_2)\leq M(c_0).
\end{split}
\end{equation*}
Letting $\tau\rightarrow 0$ in \ef{2.61n},  according to  \text{Gronwall's} inequality,  Lemma \ref{shang5} and \ef{2.61i}-\ef{2.61ijjj}, one can obtain  that for $0\leq t\leq T_5$,
\begin{equation}\label{2.61q}
\begin{split}
|u_t(t)|_2^2+|\sqrt{h^2+\epsilon^2}u(t)|_{D^2}+ |h\nabla^2 u(t)|_2 
+\int^t_0|\sqrt{h}\nabla u_s|^2_2\text{d}s
\leq  M(c_0).
\end{split}
\end{equation}

Similarly, to estimate $|\nabla^3 u|_2$, one needs to derive  the $L^2$ estimates of 
$$(\nabla \mathcal{H},\nabla \tilde{G}=\nabla G(\nabla\sqrt{h^2+\epsilon^2},u)).$$
According to  \eqref{Gdingyi}, \eqref{2.16}, \eqref{ul2},  \eqref{hl2},  \eqref{2.61q} and    Lemmas \ref{phiphi}-\ref{shang5},
\begin{equation}\label{2.61t}\begin{split}
 |\mathcal{H}|_{D^1_*}
\leq & C(|u_t|_{D^1_*}+|v|_\infty|\nabla^2 v|_2+|\nabla v|_6|\nabla v|_3+|l|_\infty|\nabla^2 \phi|_2+|\nabla \phi|_3|\nabla l|_6\\
&+| \phi|_\infty|\nabla^2 l|_2+|\nabla g|_\infty |\nabla l^\nu|_\infty |\nabla v|_2+ |\nabla^2 l^\nu|_3|g\nabla v|_6\\
&+ |\nabla l^\nu|_\infty |g\nabla^2 v|_2+|\nabla l^\nu |_\infty|\psi|_\infty|\nabla v|_2+|l^\nu |_\infty|\nabla \psi|_3|\nabla v|_6\\
&+ |l^\nu |_\infty|\psi|_\infty|\nabla^2 v|_2)\leq M(c_0)(|u_t|_{D^1_*}+c_3^{2\nu+3}),\\
 |\tilde{G}|_{D^1_*}
\leq &C( |\nabla\sqrt{h^2+\epsilon^2}|_\infty |\nabla^2 u|_2+|\nabla^2 \sqrt{h^2+\epsilon^2}|_3 |\nabla u|_6\\
&+|\nabla^3 \sqrt{h^2+\epsilon^2}|_2 | u|_\infty)\leq  M(c_0),
\end{split}
\end{equation}
where one has used the fact that 
\begin{equation}\label{2.61rrnnn}
\begin{split}
|\nabla^3\sqrt{h^2+\epsilon^2}|_2 \leq &  M(c_0)(|\nabla h^{\fr{3}{4}}|^3_6|\varphi|_\infty^{\fr{5}{4}}+|h^{-\fr{1}{4}}\nabla^2 h |_2|\psi|_\infty|\varphi|_\infty^{\fr{3}{4}}+|\nabla^3 h|_2).
\end{split}
\end{equation}

Hence, by \eqref{ul2}-\eqref{2.61illl},  \eqref{2.61q}-\eqref{2.61t},   Lemmas \ref{psi}-\ref{gh} and Lemma \ref{zhenok}, one has
\begin{equation}\label{3jieu}\begin{split}
|\sqrt{h^2+\epsilon^2}u(t)|_{D^3}\leq& C|l^{-\nu}\mathcal{H}|_{D^1_*}+ C|G(\nabla\sqrt{h^2+\epsilon^2},u)|_{D^1_*}\\
\leq &M(c_0)(|u_t|_{D^1_*}+c_3^{2\nu+3}),\\
|\sqrt{h^2+\epsilon^2}\nabla^3u(t)|_2\leq& C( |\sqrt{h^2+\epsilon^2}u(t)|_{D^3} +\|\psi\|_{L^\infty \cap D^{1,3}\cap D^2}\|u\|_2)\\
&+C(1+\|\psi\|_{L^\infty \cap D^{1,3}\cap D^2}^3\| u\|_1)(1+|\varphi|^2_\infty)\\
\leq& M(c_0)(|\sqrt{h^2+\epsilon^2}u(t)|_{D^3}+c_3^{2\nu+3}),
\end{split}
\end{equation}
which, along with $\ef{2.61q}$ and Lemma \ref{varphi}, yields $\eqref{2.61}_2$.
Then the proof is complete.

\end{proof}
We now turn to estimate  the higher order derivatives of  $u$.
\begin{lemma}\label{hu}
\rm {For} $t\in [0, T_5]$, it holds that 
	\begin{equation}\label{2.62}\begin{aligned}
	|u(t)|_{D^3}+|h\nabla^2 u(t)|_{D^1_*}
	\leq M(c_0)c_3^{2\nu+3},&\\
	|u_t(t)|_{D^1_*}+|\sqrt{h}\nabla u_t(t)|_2+\int^t_0(|u_{ss}|^2_2+|u_s|^2_{D^2})\text{d}s\leq M(c_0),&\\
	\int^t_0(|h\nabla^2u_s|^2_2+|u|^2_{D^4}+|h\nabla^2u|_{D^2}^2+|(h\nabla^2u)_s|^2_2)\text{d}s\leq M(c_0).&
	\end{aligned}\end{equation}

\end{lemma}
\begin{proof}
Multiplying \ef{2.61j} by $l^{-\nu}u_{tt}$ and integrating over $\mathbb{R}^3$ lead to 
\begin{equation}\label{2.62a}
\begin{split}
& \fr{1}{2}\fr{d}{dt}(a_2\alpha|(h^2+\epsilon^2)^{\fr{1}{4}}\nabla u_t|^2_2+a_2(\alpha+\beta)|(h^2+\epsilon^2)^{\fr{1}{4}}\text{div}u_t|^2_2)+|l^{-\fr{\nu}{2}}u_{tt}|_2^2
=\sum^4_{i=1}I_i,
\end{split}
\end{equation}
where  $I_i$, $i=1,2,3,4$,  are given and  estimated as follows:
\begin{equation}\label{i11}
\begin{split}
I_1=&\int l^{-\nu}\Big(-(v\cdot \nabla v)_t-(l\nabla \phi)_t-a_1(\phi\nabla l)_t\\
&-a_2(l^\nu)_t\sqrt{h^2+\epsilon^2}Lu-a_2l^\nu\fr{h}{\sqrt{h^2+\epsilon^2}} h_tLu\Big)\cdot u_{tt}\\
\leq& C|l^{-\fr{\nu}{2}}|_\infty\big(|v|_\infty|\nabla v_t|_2+|v_t|_2|\nabla v|_\infty+|l_t|_6|\nabla\phi|_3+|\nabla l_t|_2|\phi|_\infty\\
&+|\phi_t|_\infty|\nabla l|_2+|l|_\infty|\nabla\phi_t|_2\\
&+|l^{\nu-1}|_\infty|l_t|_6|\sqrt{h^2+\epsilon^2}\nabla^2u|_3+|l^{\nu}|_\infty|h_t|_\infty|\nabla^2u|_2\big)|l^{-\fr{\nu}{2}}u_{tt}|_2,\\
I_2=&\int l^{-\nu} \big(a_2g\nabla l^\nu\cdot Q(v)+a_3l^\nu\psi\cdot Q(v)\big)_t\cdot u_{tt}\\
\leq &C|l^{-\fr{\nu}{2}}|_\infty\big(|(\nabla l^\nu)_t|_2|g\nabla v|_\infty+|g_t|_\infty|\nabla l^\nu|_3|\nabla v|_6\\
&+|\sqrt{h}\nabla l^\nu|_\infty|gh^{-1}|^{\fr{1}{2}}_\infty|\sqrt{g}\nabla v_t|_2+|l^\nu|_\infty|\psi|_\infty|\nabla v_t|_2\\
&+|l^\nu|_\infty|\psi_t|_2|\nabla v|_\infty+|(l^\nu)_t|_6|\psi|_\infty|\nabla v|_3\big)|l^{-\fr{\nu}{2}}u_{tt}|_2,\\
I_3+I_4=&-\int a_2\Big(\alpha(\nabla\sqrt{h^2+\epsilon^2}\cdot\nabla)  u_t+(\alpha+\beta)\nabla\sqrt{h^2+\epsilon^2}\text{div}u_t\Big)\cdot u_{tt}\\
&
+\frac{1}{2}\int a_2\fr{h}{\sqrt{h^2+\epsilon^2}}h_t(\alpha|\nabla u_t|^2+(\alpha+\beta)|\text{div} u_t|^2)\\
\leq &C(|l^{\fr{\nu}{2}}|_\infty|\varphi|^{\fr{1}{2}}_\infty|\psi|_\infty|\sqrt{h}\nabla u_t|_2|l^{-\fr{\nu}{2}}u_{tt}|_2
+|h_t|_\infty|\sqrt{h}\nabla u_t|^2_2|\varphi|_\infty).
\end{split}
\end{equation}

Integrating $\ef{2.62a}$ over $(\tau,t)$ and using \ef{i11} yield that for $0\leq t\leq T_5$,
\begin{equation}\label{2.62b}
\begin{split}
&|\sqrt{h}\nabla u_t(t)|^2_2+\int^t_\tau|l^{-\fr{\nu}{2}}u_{ss}|^2_2\text{d}s\\
\leq& C|(h^2+\epsilon^2)^{\fr{1}{4}}\nabla u_t(\tau)|^2_2+M(c_0)c_4^2\int^t_0|\sqrt{h}\nabla u_s|^2_2\text{d}s +M(c_0)(c_4^{7\nu+6}t+1),
\end{split}
\end{equation}
where  \ef{2.16},  Lemmas \ref{phiphi}-\ref{gh} and \ref{l}-\ref{shang5} have been used.

It follows from the following fact
\begin{equation*}
  \begin{split}
&\sqrt{h_0}l_0^\nu \Big(\sqrt{h_0^2+\epsilon^2}\nabla Lu_0+\fr{h_0}{\sqrt{h_0^2+\epsilon^2}}Lu_0\otimes\nabla h_0 \Big) 
    =\fr{l_0^\nu\sqrt{h_0}}{\sqrt{h_0^2+\epsilon^2}}\Big(\sqrt{h_0}g_3+\epsilon^2\nabla Lu_0\Big),
 \end{split}
\end{equation*}
$\ef{ln}_2$, \ef{4.1*}, $\ef{2.14}$, the  regularities of $(v,g,w,h,\phi,l)$ shown in \eqref{4.1*}-\eqref{2.13}, and Remark 3.1  that 
\begin{equation}\label{2.62d}
\begin{split}
&\lim\sup_{\tau\rightarrow 0}|\sqrt{h}\nabla u_t(\tau)|_2
\leq \lim\sup_{\tau\rightarrow 0}|\sqrt{h}\nabla \mathcal{K}(\tau)|_2\\
 \leq& C(|\phi_0^\iota u_0|_6|\nabla^2u_0|_3+|\nabla u_0|_\infty|\phi_0^\iota\nabla u_0|_2+|l_0|_\infty|\phi_0^\iota\nabla^2\phi_0|_2\\
 &+|\nabla^2l_0\phi_0^{\iota+1}|_2
 +|\nabla l_0|_3|\phi_0^\iota\nabla \phi_0|_6+|l_0^\nu|_\infty(|\nabla\psi_0|_3|\phi_0^\iota\nabla u_0|_6\\
 &+|\psi_0|_\infty|\phi^\iota_0\nabla^2u_0|_2)+|l^{\nu-1}_0|_\infty|\psi_0|_\infty|\phi^\iota_0\nabla u_0|_6|\nabla l_0|_3\\
 &+|l_0^\nu|_\infty|g_3|_2+|l_0^\nu|_\infty|\varphi_0|^{\fr{1}{2}}_\infty|\nabla^3u_0|_2
+|l_0^{\nu-1}|_\infty|h_0^{\fr{3}{2}}Lu_0|_6|\nabla l_0|_3\\
&+|g_2|_2|\phi_0^{-\iota}|_\infty|l_0^{\nu-1}|_\infty|\nabla l_0|_\infty+|\sqrt{h_0}\nabla^2l_0^\nu|_2| h_0
\nabla u_0|_\infty\\
&+|\sqrt{h_0}\nabla l_0^\nu|_6(|h_0\nabla^2 u_0|_3+|\psi_0|_\infty|\nabla u_0|_3))\leq M(c_0).
\end{split}
\end{equation}

Letting $\tau\rightarrow 0$ in  \eqref{2.62b}, one gets from Gronwall's inequality that for $0\leq t\leq T_5$,
\begin{equation}\label{2.62m}\begin{split}
&|\sqrt h\nabla u_t(t)|_2^2+|\nabla u_t(t)|^2_2+\int^t_0|u_{ss}|^2_2\text{d}s
\leq  M(c_0),
\end{split}
\end{equation}
which, along with   $\ef{3jieu}$, yields  
\begin{equation}\label{2.63}
|\sqrt{h^2+\epsilon^2}u|_{D^3}+|\sqrt{h^2+\epsilon^2}\nabla^3u|_2+|h\nabla^2u|_{D^1}+|\nabla^3u|_2\leq M(c_0)c_3^{2\nu+3}.
\end{equation}

Next, denote $\mathcal{G}=\mathcal{H}_{t}-a_2(l^\nu\sqrt{h^2+\epsilon^2})_tLu$. One notes that  $\ef{2.61j}$ gives
\begin{equation}\label{2.64}\begin{split}
a_2L(\sqrt{h^2+\epsilon^2}u_t)&=a_2\sqrt{h^2+\epsilon^2}Lu_t-a_2G(\nabla\sqrt{h^2+\epsilon^2},u_t)\\
&=l^{-\nu}\mathcal{G}-a_2G(\nabla\sqrt{h^2+\epsilon^2},u_t).
\end{split}
\end{equation}
For deriving   the $L^2$ estimates of $(\nabla^2 u_t,\nabla^4 u)$, one needs to estimate the  $L^2$ norm  of $(\mathcal{G},\widehat{G}=G(\nabla\sqrt{h^2+\epsilon^2},u_t),\nabla^2 \mathcal{H})$,
which  follows from \eqref{Gdingyi}, \eqref{2.16}, \eqref{hl2}, \eqref{2.62m}-\eqref{2.63} and  Lemmas \ref{phiphi}-\ref{lu} that  
\begin{equation}\label{2.61ccvv}
\begin{split}
|\mathcal{G}|_2\leq & C(|u_{tt}|_2+\|v\|_2|\nabla v_t|_2+\|l\|_{L^\infty\cap D^1\cap D^2}\| \phi_t\|_1+(|\phi|_\infty+|\nabla \phi|_3)|l_t|_{ D^1}\\
&+|(l^\nu)_t |_6|\sqrt{h^2+\epsilon^2} Lu|_3 +|l^\nu |_\infty |h_t|_\infty|\nabla^2u|_2+|g_t|_\infty|\nabla l^\nu|_2
|\nabla v|_\infty\\
&+|g\nabla v|_\infty|\nabla(l^\nu)_t|_2+|l^{\nu-1}|_\infty|\sqrt{h}\nabla l|_\infty|gh^{-1}|^{\fr{1}{2}}_\infty|\sqrt{g}\nabla v_t|_2\\
&+|(l^\nu)_t|_6|\psi|_\infty|\nabla v|_3+|l^\nu|_\infty|\psi_t|_2|\nabla v|_\infty+|l^\nu|_\infty|\psi|_\infty|\nabla v_t|_2)\\
\leq & M(c_0)(|u_{tt}|_2+c_4^{3\nu+3}),\\
|\mathcal{H}|_{D^2}\leq & C(|u_t|_{D^2}+\|v\|_2\|\nabla v\|_2+\|l\|_{L^\infty\cap D^1\cap D^3}(|\phi|_\infty+\|\nabla \phi\|_2)\\
&+\|\nabla l^\nu\|_2 (\|g\nabla v\|_{L^\infty\cap D^1\cap D^2}+\|\nabla g\|_{L^\infty\cap D^2 }\|\nabla v\|_2)\\
& +\|l^\nu \|_{L^\infty\cap D^1\cap D^2}\|\psi\|_{L^q\cap D^{1,3}\cap D^2}\|\nabla v\|_2)\\
\leq & M(c_0)(|u_t|_{D^2}+c_4^{6\nu+5}),\\
|\widehat{G}|_2
\leq & C(|\nabla\sqrt{h^2+\epsilon^2}|_\infty |\nabla u_t|_2+|\nabla^2 \sqrt{h^2+\epsilon^2}|_3 |u_t|_6)\leq M(c_0).
\end{split}
\end{equation}

It follows from \eqref{2.61h}-\eqref{2.61illl}, \eqref{2.61t}, \eqref{2.62m}-\eqref{2.61ccvv},   Lemmas \ref{phiphi}-\ref{lu} and Lemma \ref{zhenok} that
\begin{equation}\label{2.65}\begin{split}
\displaystyle
|\sqrt{h^2+\epsilon^2}u_t|_{D^2}\leq &  C|l^{-\nu}\mathcal{G}|_2+C|G(\nabla\sqrt{h^2+\epsilon^2},u_t)|_2\\
\displaystyle
\leq& M(c_0)(|u_{tt}|_2+c_4^{3\nu+3}),\\[2pt]
\displaystyle
|\sqrt{h^2+\epsilon^2}\nabla^2u_t|_2\leq & C(|\sqrt{h^2+\epsilon^2}u_t|_{D^2}+|\nabla u_t|_2(|\psi|_\infty+|\nabla\psi|_3)\\
&+|\psi|^2_\infty|u_t|_2|\varphi|_\infty)
\leq  M(c_0)(|u_{tt}|_2+c_4^{3\nu+3}),\\[2pt]
\displaystyle
|(h\nabla^2u)_t|_2\leq & C(|h\nabla^2u_t|_2+|h_t|_\infty|\nabla^2 u|_2)
\leq  M(c_0)(|u_{tt}|_2+c_4^{3\nu+3}),\\[2pt]
\displaystyle
|u|_{D^4}\leq& C|(h^2+\epsilon^2)^{-\fr{1}{2}}l^{-\nu}\mathcal{H}|_{D^2}
\leq   M(c_0)(|u_{t}|_{D^2}
+c_4^{6\nu+5})\\
\leq & M(c_0)(|u_{tt}|_2+c_4^{6\nu+5}).
\end{split}
\end{equation}

Due to  $\ef{ln}_2$, it holds that  for multi-index $\varsigma\in\mathbb{Z}^3_+$ with $|\varsigma|=2$,  
\begin{equation}\label{2.67}
\begin{split}
&a_2L(\sqrt{h^2+\epsilon^2}\nabla^\varsigma u)=a_2\sqrt{h^2+\epsilon^2}\nabla^\varsigma Lu-a_2G(\nabla\sqrt{h^2+\epsilon^2},\nabla^\varsigma u)\\
=& \sqrt{h^2+\epsilon^2}\nabla^\varsigma\big[\big(\sqrt{h^2+\epsilon^2})^{-1}l^{-\nu}\mathcal{H}\big]
-a_2G(\nabla\sqrt{h^2+\epsilon^2},\nabla^\varsigma u),\\
\end{split}
\end{equation}
which, along with \eqref{2.61illl}-\eqref{hl2}, \eqref{2.61t}, \eqref{2.62m}-\eqref{2.63}, \eqref{2.61ccvv}-\eqref{2.65},   Lemmas \ref{phiphi}-\ref{lu} and Lemma \ref{zhenok},  implies that 
\begin{equation}\label{2.68}
\begin{split}
|\sqrt{h^2+\epsilon^2}\nabla^2u(t)|_{D^2}
\leq& C|\sqrt{h^2+\epsilon^2}\nabla^\varsigma\big[\big(\sqrt{h^2+\epsilon^2})^{-1}l^{-\nu}\mathcal{H}\big]|_{2}\\
&+C(|\psi|_\infty|u|_{D^3}+|\nabla\psi|_3|\nabla^2u|_6+|\nabla^2u|_2|\psi|^2_\infty|\varphi|_\infty)\\
\leq & M(c_0)(|u_{tt}|_2+c_4^{6\nu+5}).
\end{split}
\end{equation}

At last,   \eqref{2.62m}-\eqref{2.63}, \eqref{2.65}, \eqref{2.68} and Lemma \ref{varphi} yield \eqref{2.62}.

The proof  of this lemma is complete.
\end{proof}
\smallskip
\smallskip
\smallskip
Finally,  the following  time weighted estimates for the velocity $u$ hold.
\begin{lemma}\label{hu2}For $t\in [0,T_5]$, it holds that 
	\begin{equation}\label{2.70}
\begin{aligned}
	t|u_t(t)|^2_{D^2}+t|h\nabla^2u_t(t)|^2_2+t|u_{tt}(t)|^2_2+t|u(t)|^2_{D^4} & \leq  M(c_0)c^{6\nu+4}_4,\\
	\int^t_0s(|u_{ss}|_{D^1_*}^2+|u_s|^2_{D^3}+|\sqrt{h} u_{ss}|_{D^1_*}^2)\text{\rm d}s & \leq  M(c_0)c^{6\nu+4}_4.
\end{aligned}
	\end{equation}
\end{lemma}
\begin{proof} Differentiating  $\ef{2.61j}$ with respect to $t$ yields 
\begin{equation}\label{2.71}
\begin{split}
&u_{ttt}+a_2\sqrt{h^2+\epsilon^2}l^\nu Lu_{tt}\\
=& -(v\cdot\nabla v)_{tt}-a_1(\phi\nabla l)_{tt}-(l\nabla\phi)_{tt}
+a_3(l^\nu\psi\cdot Q(v))_{tt}
+a_2(g\nabla l^\nu\cdot Q(v))_{tt}\\
&-a_2( \sqrt{h^2+\epsilon^2}l^\nu)_{tt}Lu-2a_2(l^\nu)_{t}\sqrt{h^2+\epsilon^2}L u_t-2a_2\frac{h}{\sqrt{h^2+\epsilon^2}}h_{t} l^\nu Lu_t.
\end{split}
\end{equation}
Multiplying $\ef{2.71}$ by $l^{-\nu}u_{tt}$ and integrating over $\mathbb{R}^3$ give
\begin{equation}\label{2.72}
\begin{split}
&\fr{1}{2}\fr{d}{dt}|l^{-\fr{\nu}{2}}u_{tt}|^2_2+a_2\alpha|(h^2+\epsilon^2)^{\fr{1}{4}}\nabla u_{tt}|^2_2+a_2(\alpha+\beta)|(h^2+\epsilon^2)^{\fr{1}{4}}\text{div}u_{tt}|^2_2=\sum^{8}_{i=5}I_i,
\end{split}
\end{equation}
where  $I_i$, $i=5,6,7,8$, are given and estimated as follows.
\begin{align}\label{h1}
I_5=&\int l^{-\nu}\big(-(v\cdot\nabla v)_{tt}-a_1(\phi\nabla l)_{tt}-(l\nabla\phi)_{tt}\big)\cdot u_{tt}\nonumber\\
\leq& C|l^{-\fr{\nu}{2}}|_\infty\big(|\nabla v_t|_6|v_t|_3+|\nabla v|_\infty|v_{tt}|_2+|v|_\infty|\nabla v_{tt}|_2+|\phi|_\infty|\nabla l_{tt}|_2\nonumber\\
&+
|\phi_{tt}|_2|\nabla l|_\infty+|\phi_t|_\infty|\nabla l_t|_2+|l_t|_6|\nabla\phi_t|_3\nonumber\\
&+|l_{tt}|_6|\nabla\phi|_3+|l|_\infty|\nabla\phi_{tt}|_2)|l^{-\fr{\nu}{2}}u_{tt}|_2,\nonumber\\
I_6=&a_3\int l^{-\nu} (l^\nu\psi\cdot Q(v))_{tt}\cdot u_{tt}\nonumber\\
\leq& C|l^{-\fr{\nu}{2}}|_\infty(|l^\nu|_\infty|\psi_{tt}|_2|\nabla v|_\infty+|l^{\nu-1}|_\infty|l_{tt}|_6|\psi|_\infty|\nabla v|_3
\nonumber\\
&+|l^{\nu-2}|_\infty|l_{t}|^2_6|\psi|_\infty|\nabla v|_6+|l^\nu|_\infty|\psi|_\infty|\nabla v_{tt}|_2+|\psi_t|_3|(l^\nu)_t|_6|\nabla v|_\infty
\nonumber\\
&+|\psi|_\infty|(l^\nu)_t|_6|\nabla v_t|_3
+|l^\nu|_\infty|\psi_t|_3|\nabla v_t|_6
)|l^{-\fr{\nu}{2}}u_{tt}|_2,\nonumber\\
I_7=&a_2\int l^{-\nu}(g\nabla l^\nu\cdot Q(v))_{tt}\cdot u_{tt}\nonumber\\
\leq& C|l^{-\fr{\nu}{2}}|_\infty\Big(|l^{\nu-3}|_\infty|g\nabla v|_\infty|l_{t}|^2_6|\nabla l|_6+|l^{\nu-2}|_\infty|g\nabla v|_\infty|l_{tt}|_6|\nabla l|_3\nonumber\\
&+ l^{\nu-2}|_\infty|g\nabla v|_\infty|l_{t}|_6|\nabla l_t|_3+|l^{\nu-1}|_\infty|\nabla l_{t}|_2|g_t|_\infty|\nabla v |_\infty\nonumber\\
&+|l^{\nu-2}|_\infty| l_{t}|_6(|\nabla l|_3|g_t|_\infty|\nabla v |_\infty+|\nabla l |_6|g\nabla v_t|_6)\Big)|l^{-\fr{\nu}{2}}u_{tt}|_2\nonumber\\
&+C|l^{-1}|_\infty| \nabla l_{t}|_2|g\nabla v_t|_6|u_{tt}|_3+C|l^{\frac{\nu}{2}-1}|_\infty|\nabla l|_3|g_{tt}|_6|\nabla v|_\infty|l^{-\fr{\nu}{2}}u_{tt}|_2\\
&+C|l^{\frac{\nu}{2}-1}|_\infty|gh^{-1}|^{\frac{1}{2}}_\infty|\sqrt{h}\nabla l|_\infty|\sqrt{g}\nabla v_{tt}|_2|l^{-\fr{\nu}{2}}u_{tt}|_2\nonumber\\
&+C|l^{\frac{\nu}{2}-1}|_\infty(|\nabla l|_\infty|g_t|_\infty|\nabla v_t|_2+|g\nabla v|_\infty|\nabla l_{tt}|_2)|l^{-\fr{\nu}{2}}u_{tt}|_2,\nonumber\\
I_8=&-a_2\int l^{-\nu}\big(( \sqrt{h^2+\epsilon^2}l^\nu)_{tt}Lu+2(l^\nu)_{t}\sqrt{h^2+\epsilon^2}L u_t\nonumber\\
&+\frac{2h}{\sqrt{h^2+\epsilon^2}}h_{t} l^\nu Lu_t+\alpha l^{\nu}(\nabla\sqrt{h^2+\epsilon^2}\cdot\nabla)  u_{tt}\nonumber\\
&+(\alpha+\beta)l^{\nu}\nabla\sqrt{h^2+\epsilon^2}\text{div}u_{tt}  \big) \cdot u_{tt}+\fr{1}{2}\int (l^{-\nu})_t|u_{tt}|^2\nonumber\\
\leq& C|l^{-\fr{\nu}{2}}|_\infty
\big(|(l^\nu)_t|_6|h_t|_\infty|\nabla^2 u|_3
+|l^{\nu-2}|_\infty|l_t|^2_6|\sqrt{h^2+\epsilon^2}\nabla^2u|_6\nonumber\\
&+|l^\nu|_\infty|h_{tt}|_6|\nabla^2u|_3+|l^\nu|_\infty|h_t|^2_\infty|\varphi|_\infty|\nabla^2 u|_2\nonumber\\
&+|l^\nu|_\infty|h_t|_\infty|\varphi|_\infty|h\nabla^2 u_{t}|_2+|l^\nu|_\infty|\psi|_\infty|\sqrt{h}\nabla u_{tt}|_2|\varphi|^{\fr{1}{2}}_\infty\big)|l^{-\fr{\nu}{2}}u_{tt}|_2\nonumber\\
&+C|l^{-1}|_\infty|l_t|_6|\sqrt{h^2+\epsilon^2}\nabla^2u_t|_2|u_{tt}|_3
\nonumber\\
&+C|l^{-1}|_\infty|l_{tt}|_6|\sqrt{h^2+\epsilon^2}\nabla^2u|_2|u_{tt}|_3 \nonumber\\
&+C|l^{-\fr{\nu}{2}-1}|_\infty|l_t|_6|l^{-\fr{\nu}{2}}u_{tt}|_2|u_{tt}|_3.\nonumber
\end{align}
Multiplying $\ef{2.72}$ by $t$ and integrating over $(\tau,t)$, one can obtain from the  estimates on $I_i$ ($i=5,...,8$), \eqref{2.16} and Lemmas \ref{phiphi}-\ref{hu} that 

\begin{equation}\label{2.73}
\begin{split}
&t|l^{-\fr{\nu}{2}}u_{tt}(t)|^2_2+\fr{a_2\alpha}{4}\int^t_\tau s|\sqrt{h}\nabla u_{ss}|^2_2\text{d}s\\
\leq & \tau|l^{-\fr{\nu}{2}}u_{tt}(\tau)|^2_2+M(c_0)c^{6\nu+4}_4(1+t)+M(c_0)c^{2\nu+8}_5\int^t_\tau s|l^{-\fr{\nu}{2}}u_{ss}|^2_2\text{d}s.
\end{split}
\end{equation}
Due to $\ef{2.62m}$,    there exists a sequence $s_k$ such that
\begin{equation*}
s_k\longrightarrow 0, \quad \text{and} \quad s_k|u_{tt}(s_k,x)|^2_2\longrightarrow 0,\quad \text{as}\quad k\longrightarrow \infty.
\end{equation*}
Taking $\tau=s_k$ and letting $k\rightarrow \infty$ in $\ef{2.73}$, one has  by  Gronwall's inequality that
\begin{equation}\label{2.74}
\begin{split}
t|u_{tt}(t)|^2_2+&\int^t_0s|\sqrt{h}\nabla u_{ss}|^2_2 \text{d}s+\int^t_0s|\nabla u_{ss}|^2_2\text{d}s\leq M(c_0)c^{6\nu+4}_4,
\end{split}
\end{equation}
for $0\leq t\leq T_5$, which, along with  $\ef{2.65}$  and $\ef{2.74}$, yields  that
\begin{equation}\label{2.75}
t^{\fr{1}{2}}|\nabla^2u_t(t)|_2+t^{\fr{1}{2}}|h\nabla^2u_t(t)|_2+ t^{\fr{1}{2}}|\nabla^4u(t)|_2\leq  M(c_0)c^{3\nu+2}_4.
\end{equation}

Next, to derive the $L^2$ estimate of $\nabla^3 u_t$, one deals with  the $L^2$ estimates of 
$$(\nabla \mathcal{G},\nabla \widehat{G}=\nabla G(\nabla\sqrt{h^2+\epsilon^2},u_t)).$$
It follows from  \eqref{Gdingyi},  \eqref{hl2}, \eqref{2.61rrnnn} and Lemmas \ref{phiphi}-\ref{hu} that 
\begin{equation}\label{2.61ccvvmmmm}
\begin{split}
|\mathcal{G}|_{D^1_*}\leq & C(|u_{tt}|_{D^1_*}+\|\nabla v\|_2|\nabla v_t|_2+|v|_\infty|\nabla^2 v_t|_2\\
&+ \|l\|_{L^\infty\cap D^1\cap D^3}\| \phi_t\|_2+\|l_t\|_{ D^1_*\cap D^2}(|\phi|_\infty+\|\nabla \phi\|_2\\
&+\|l^{\nu-1} \|_{1,\infty}\|l_t\|_{L^\infty\cap D^2}(\|\sqrt{h^2+\epsilon^2} Lu\|_1+|\psi|_\infty|\nabla^2 u|_2)\\
&+(1+|\psi|_\infty)(1+|\varphi|_\infty)\|h_t\|_{L^\infty\cap D^2}\|l^{\nu} \|_{1,\infty} \|\nabla^2 u\|_1\\
&+\|g_t\|_{L^\infty\cap D^1}\|\nabla l^\nu\|_2\|\nabla v\|_2+\|\nabla l^\nu\|_2(|\nabla g|_\infty|\nabla v_t|_2+|g\nabla^2 v_t|_2)\\
&+(|g\nabla v|_\infty+|\nabla g|_\infty\|\nabla v\|_2+\|g\nabla^2 v\|_1)\|l_t\|_{{ D^1_*\cap  D^2}}\|l^{\nu-1} \|_{ L^\infty\cap D^1\cap  D^3}\\
&+\|l^{\nu-1} \|_{1,\infty}\|l_t\|_{  D^1_*}\|\psi\|_{L^\infty \cap  D^{1,3}}\|\nabla v\|_2\\
&+\|l^{\nu} \|_{1,\infty}\|\psi_t\|_1\|\nabla v\|_2+\|l^{\nu} \|_{1,\infty}\|\psi\|_{L^\infty \cap  D^{1,3}}\|\nabla v_t\|_1)\\
\leq &M(c_0)(|\nabla u_{tt}|_2+c^{4\nu+3}_4|g\nabla ^2 v_t|_{2}+c_4^{5\nu+5}|l_t|_{D^2}+c_4^{5\nu+7}),\\
|\widehat{G}|_{D^1_*}\leq &C( |\nabla\sqrt{h^2+\epsilon^2}|_\infty |\nabla^2 u_t|_2+|\nabla^2 \sqrt{h^2+\epsilon^2}|_3 |\nabla u_t|_6\\
&+|\nabla^3 \sqrt{h^2+\epsilon^2}|_2 | u_t|_\infty)\leq  M(c_0)(|u_t|_{D^2}+c^{2\nu+3}_4).
\end{split}
\end{equation}
Hence $\ef{2.64}$, \eqref{2.61ccvvmmmm}, Lemmas \ref{zhenok} and   \ref{phiphi}-\ref{hu} yield that  for $0\leq t\leq T_5$,
\begin{equation*}
\begin{split}
|\sqrt{h^2+\epsilon^2}u_t|_{D^3}\leq&  C|l^{-\nu}\mathcal{G}|_{D^1_*}+C|G(\nabla\sqrt{h^2+\epsilon^2},u_t)|_{D^1_*}\\
\leq &M(c_0)(\| u_{tt}\|_1+|u_t|_{D^2}+c_4^{6\nu+7}(|g\nabla ^2 v_t|_{2}+|l_t|_{D^2}+1)),\\
|\sqrt{h^2+\epsilon^2}\nabla^3u_t(t)|_2\leq& C(|\sqrt{h^2+\epsilon^2}u_t|_{D^3}+|u_t|_\infty|\nabla^2\psi|_2+|\nabla u_t|_6|\nabla \psi|_3\\
&+|\nabla^2u_t|_2|\psi|_\infty
+|\nabla u_t|_2\|\psi\|^2_{L^\infty \cap D^{1,3}\cap D^2}|\varphi|_\infty+|u_t|_2|\psi|_\infty^3|\varphi|^2_\infty)\\
\leq & C|\sqrt{h^2+\epsilon^2}u_t|_{D^3}+M(c_0)(|u_t|_{D^2}+c_4^{2\nu+3}),
\end{split}
\end{equation*}
which, along with $\ef{2.62}$, $\ef{2.74}$-$\ef{2.75}$ and Lemma \ref{varphi}, yields $\eqref{2.70}_2$.

 The proof of Lemma \ref{hu2} is complete.
\end{proof}

It follows from Lemmas \ref{phiphi}-\ref{hu2}  that for 
 $0\leq t\leq T_5=\min\{T^*,(1+M(c_0)c_5)^{-40-10\nu}\}$,
\begin{align*}
\|(\phi-\eta)(t)\|^2_{D^1_*\cap D^3}+\|\phi_t(t)\|^2_2+
|\phi_{tt}(t)|^2_2+\int^t_0\|\phi_{ss}\|^2_1\text{d}s\leq &Cc_4^6,\\
\|\psi(t)\|^2_{L^q\cap D^{1,3}\cap D^2}\leq M(c_0),\hspace{2mm}|\psi_t(t)|_2\leq Cc_3^2,\quad 
|h_t(t)|_\infty^2\leq & Cc_3^3c_4,\\ 
h(t,x)>\fr{1}{2c_0}, \ \fr{2}{3}\eta^{-2\iota}<\varphi, \ 
|\psi_t(t)|^2_{D^1_*}+\int^t_0(|\psi_{ss}|^2_2+|h_{ss}|^2_6)\text{d}s\leq & Cc_4^4,\\ 
\tilde{C}^{-1}\leq gh^{-1}(t,x)\leq  \tilde{C},\ \  \ |\xi(t)|_{D^1_*}+|\zeta(t)|_4+|h^{-\frac{1}{4}}\nabla^2 h(t)|_2\leq& M(c_0),\\
 	\|\tilde{n}(t)\|_{ L^\infty\cap D^{1,q}\cap D^{1,4}\cap D^{1,6} \cap D^{2} \cap D^3}\leq  M(c_0),\quad |\tilde{n}_t(t)|_2\leq & M(c_0)c_1,\\
|\tilde{n}_t(t)|_\infty+	|\nabla \tilde{n}_t(t)|_2+ |\nabla \tilde{n}_t(t)|_6\leq M(c_0)c_4^2,\ \  |\tilde{n}_{tt}(t)|_2\leq& M(c_0)c_4^3,\\
|u|^2_\infty+|\sqrt{h}\nabla u(t)|^2_2+\|u(t)\|^2_1+\int^t_0\big(\|\nabla u\|^2_1+|u_s|^2_2\big)\text{d}s\leq & M(c_0),\\
|\nabla l(t)|^2_2+|h^{\fr{1}{4}}\nabla l(t)|^2_2+
\int^t_0(|h^{-\fr{1}{4}}l_s|_2^2+|\sqrt{h} \nabla^2l|^2_2+|\nabla^2l|^2_2)\text{d}s\leq & M(c_0)c^{3\nu}_1,\\
|h^{-\frac{1}{4}}l_{t}(t)|^2_2+|\sqrt{h}\nabla^2l(t)|^2_2+\int^t_0(|h^{\frac{1}{4}} \nabla l_s|^2_2+|\sqrt{h}\nabla^3l |_2^2)\text{d}s
   \leq &M(c_0)c^{4\nu+2}_1,\\
	|h^{\fr{1}{4}}\nabla l_t(t)|^2_2+
  |\sqrt{h}\nabla^3l(t) |_2^2+ \int^t_0(|h^{-\fr{1}{4}} l_{ss}|_2^2+|\sqrt{h}\nabla^2 l_s|_2^2)\text{d}s\leq &M(c_0)c_1^{8\nu+6},\\
  t|l_t(t)|^2_{D^2}+t|\sqrt{h} \nabla^2l_t(t)|^2_2+t|h^{-\frac{1}{4}}l_{tt}(t)|^2_2  \leq & M(c_0)c^{\nu}_1,\\
  \int^t_0s(|l_{ss}|_{D^1_*}^2+|h^{\frac{1}{4}} l_{ss}|_{D^1_*}^2)\text{d}s  \leq  M(c_0),\quad  \fr{1}{2}c_0^{-1}\leq l(x,t)\leq & \fr{3}{2}c_0,\\
  (|u|_{D^2}^2+|h\nabla^2u|^2_2+|u_t|^2_2)(t)+\int^t_0(|u|^2_{D^3}+|h\nabla^2u|_{D^1_*}^2+|u_s|^2_{D^1_*})\text{d}s\leq & M(c_0),  \\
  (|u_t|^2_{D^1_*}+|\sqrt{h}\nabla u_t|^2_2+|u|^2_{D^3}+|h\nabla^2u|^2_{D^1_*})(t)+\int^t_0|u_s|^2_{D^2}\text{d}s\leq &  M(c_0)c^{4\nu+6}_3, \\
  \int^t_0(|u_{ss}|^2_2+|u|^2_{D^4}+|h\nabla^2u|^2_{D^2}+|(h\nabla^2u)_s|^2_2)\text{d}s\leq & M(c_0),  \\
t|u_t(t)|^2_{D^2}+t|h\nabla^2u_t(t)|^2_2+t|u_{tt}(t)|^2_2+t|u(t)|^2_{D^4}  \leq & M(c_0)c^{6\nu+4}_4,\\
\int^t_0s(|u_{ss}|_{D^1_*}^2+|h\nabla^3 u_s|^2_{2}+|\sqrt{h} u_{ss}|_{D^1_*}^2)\text{d}s  \leq & M(c_0)c^{6\nu+4}_4.
\end{align*}
Then  setting  
 \begin{equation*}
\begin{split}
c_1^2=& M(c_0)^2,\quad 
c_2^2=c_3^2=M(c_0)^{8\nu+7},\\
c_4^2=&M(c_0)^{24\nu^2+45\nu+23},\quad  c_5^2=M(c_0)^{72\nu^3+207\nu^2+204\nu+70},\\
T^*=&T_5=
\min\{T,(1+M(c_0)c_5)^{-40-10\nu}\},\\
\end{split}
\end{equation*}
one can arrive at  the desired  estimates\eqref{key1kk}.

In summary, our  analysis shows that $T_*=T_5$ only depends on the fixed parameters $(A, R, c_v, \alpha, \beta, \gamma, \delta,\daleth)$ and  $(\phi_0,u_0,l_0)$. The proof for Lemma 3.1 is complete.

\subsection{Vanishing of the artificial dissipations}

By the uniform estimates   \ef{key1kk}, one can now obtain  the local well-posedness of    \ef{ln} with  $\epsilon=0$ and $\phi^\eta_0\geq \eta$ for any   constant $\eta>0$. For simplicity, denote $B_R$ a ball centered at origin with radius $R$.
\begin{lemma}\label{epsilon0}
	Let $\eta\in (0,1]$, $\ef{can1}$ hold,  the initial data $(\phi_0,u_0,l_0,h_0)$ satisfy 
	\eqref{a}-\eqref{2.8*},  \eqref{2.14} hold, and the constants $(c_1,c_2,c_3,c_4,c_5,T^*)$ be the ones obtained in {\rm Lemma 3.1}.  Then if $(v,g,w)$ satisfy \eqref{4.1*} and \eqref{2.16}, there exist  a unique strong  solution $
	(\phi^\eta,u^\eta,l^\eta, h^\eta)$
	  in $[0,T^*]\times\mathbb{R}^3$ to \ef{ln} with  $\epsilon=0$ satisfying \ef{2.13}. Moreover,  \ef{key1kk} hold for $
	(\phi^\eta,u^\eta,l^\eta, h^\eta)$ uniformly (independent of $\eta$).
\end{lemma}
\begin{proof} First, it follows from Lemmas \ref{ls}--\ref{hu2} that  for every $\epsilon>0$ and $\eta>0$,   there exist a time $T^*>0$, independent of $(\epsilon,\eta)$, and   a unique strong solution $(\phi^{\epsilon,\eta}, u^{\epsilon,\eta}, l^{\epsilon,\eta}, h^{\epsilon,\eta})(t,x)$ in $[0,T^*]\times \mathbb{R}^3$  to (\ref{ln}) satisfying  the estimates in  \ef{key1kk} which are independent of  $(\epsilon,\eta)$.
Then  by  using of the characteristic method and the  standard energy estimates for $(\ref{ln})_4$,  one can show that for $0\leq t\leq T^*$,
\begin{equation}\label{related}
|h^{\epsilon,\eta}(t)|_{\infty}+|\nabla h^{\epsilon,\eta}(t)|_2+ |h^{\epsilon,\eta}_t(t)|_{2}\leq C(A, R, c_v,\digamma, \eta, \alpha, \beta, \gamma, \delta, T^*, c_0).
\end{equation}

Thus, it follows from  \ef{key1kk}-\ef{related}  and   Lemma \ref{aubin}  that for any $R> 0$,  there exists a subsequence of solutions (still denoted by) $(\phi^{\epsilon,\eta}, u^{\epsilon,\eta},l^{\epsilon,\eta},  h^{\epsilon,\eta} )$, which  converges to  a limit $(\phi^\eta,  u^\eta, l^\eta, h^\eta) $ as $\epsilon \rightarrow 0$ in the following  strong sense:
\begin{equation}\label{ert1}\begin{split}
&(\phi^{\epsilon,\eta}, u^{\epsilon,\eta}, l^{\epsilon,\eta}, h^{\epsilon,\eta}) \rightarrow (\phi^\eta,  u^\eta, l^\eta, h^\eta ) \ \ \text{ in } \ C([0,T^*];H^2(B_R)).
\end{split}
\end{equation}
Second,  via the classical  weak convergence   arguments, one can easy show that  $(\phi^\eta,  u^\eta, l^\eta, h^\eta) $   is the unique strong  solution   to \ef{ln}   with  $\epsilon=0$ satisfying \ef{2.13} with $T$ replaced by $T^*$, and the uniform estimates \ef{key1kk}.

Thus the proof of Lemma  \ref{epsilon0} is complete.

\end{proof}

\subsection{Nonlinear approximation solutions away from vacuum}

In this subsection, we will prove the local existence  of the unique strong  solution to the following Cauchy problem under the assumption that $\phi^\eta_0\geq \eta$:
\begin{equation}\label{nl}\left\{\begin{aligned}
&\phi^{\eta}_t+u^{\eta}\cdot\nabla\phi^{\eta}+(\gamma-1)\phi^{\eta} \text{div}u^{\eta}=0,\\[2pt]
&u^{\eta}_t+u^{\eta}\cdot \nabla u^{\eta}+a_1\phi^{\eta}\nabla l^{\eta}+l^{\eta}\nabla\phi^{\eta}+a_2(l^\eta)^\nu h^{\eta} Lu^{\eta}\\[2pt]
=& a_2 h^{\eta} \nabla (l^\eta)^\nu  \cdot Q(u^{\eta})+a_3(l^\eta)^\nu  \psi^{\eta} \cdot Q(u^{\eta}),\\[2pt]
& (\phi^{\eta})^{-\iota}(l^{\eta}_t+u^{\eta}\cdot\nabla l^{\eta})-a_4(\phi^{\eta})^{\iota}(l^{\eta})^\nu \triangle l^{\eta}\\
=&a_5(l^\eta)^\nu  n^{\eta}(\phi^\eta)^{3\iota} H(u^{\eta})+a_6(l^{\eta})^{\nu+1} (\phi^{\eta})^{-\iota}\text{div} \psi^{\eta}+\Theta(\phi^{\eta},l^{\eta},\psi^{\eta}),\\
&h^{\eta}_t+u^\eta\cdot\nabla h^\eta+(\delta-1) (\phi^\eta)^{2\iota}\text{div} u^{\eta}=0,\\[2pt]
&(\phi^{\eta},u^{\eta},l^{\eta},h^{\eta})|_{t=0}=(\phi^\eta_0,u^\eta_0,l^\eta_0,h_0^\eta)
=(\phi_0+\eta,u_0,l_0,(\phi_0+\eta)^{2\iota}) \ \text{in} \ \mathbb{R}^3,\\[2pt]
&(\phi^{\eta},u^{\eta},l^{\eta},h^{\eta})\rightarrow (\eta,0,\bar{l},\eta^{2\iota}) \quad \text{as} \hspace{2mm}|x|\rightarrow \infty \quad \rm for\quad t\geq 0,
\end{aligned}\right.\end{equation}
where $\psi^\eta=\fr{a\delta}{\delta-1}\nabla h^\eta$ and $n^{\eta}=(ah^{\eta})^b$. For simplicity, in the rest of this subsection,  $C$  will denote a positive  generic constant   independent of $\eta$ and $k$.
\begin{theorem}\label{nltm}
Let $\ef{can1}$ hold and $\eta>0$. Assume that the initial data $(\phi_0,u_0,l_0,h_0)$ satisfy
	\eqref{a}-\eqref{2.8*}, and \ef{2.14} holds with a  constant $c_0>0$ independent of $\eta$. Then there exists a time $T_*>0$, independent of $\eta$, and   a unique strong solution 
	$
	(\phi^\eta, u^\eta,l^\eta, h^\eta =\phi^{2\iota})
	$
	 in $[0,T_*]\times\mathbb{R}^3$ to \ef{nl} satisfying \ef{2.13}, and  the uniform estimates (independent of $\eta$) \ef{key1kk}   hold for $(\phi^\eta,u^\eta,l^\eta,h^\eta)$ with $T^*$ replaced by $T_*$. 
\end{theorem}
The proof is given  by an iteration scheme described below. \\

 Let $(\phi^0,u^0,l^0,h^0)$ be the solution to the following Cauchy problem
\begin{equation}\label{xyz}\left\{\begin{aligned}
\displaystyle
&U_t+u_0\cdot \nabla U=0,\ \ \text{in} \ \ (0,\infty)\times\mathbb{R}^3, \\[2pt]
&Y_t-W\triangle Y=0,\ \ \text{in} \ \ (0,\infty)\times\mathbb{R}^3, \\[2pt]
\displaystyle
&W^{-\frac{1}{2}}Z_t-W^{\frac{1}{2}}\triangle Z=0,\ \ \text{in} \ \ (0,\infty)\times\mathbb{R}^3, \\[2pt]
\displaystyle
&W_t+u_0\cdot \nabla W=0,\ \ \text{in} \ \ (0,\infty)\times\mathbb{R}^3, \\[2pt]
\displaystyle
&(U,Y,Z,W)|_{t=0}=(\phi^\eta_0,u^\eta_0,l^\eta_0,h^\eta_0)
= (\phi_0+\eta,u_0,l_0,(\phi_0+\eta)^{2\iota})  \ \ \text{in} \ \ \mathbb{R}^3, \\[2pt]
\displaystyle
&(U,Y,Z,W)\rightarrow (\eta,0,\bar{l},\eta^{2\iota}) \quad \text{as} \hspace{2mm}|x|\rightarrow \infty \quad \rm for\quad t\geq 0.
\end{aligned}\right.\end{equation}
First, it follows from Remarks \eqref{etawuguan}--\ref{r1} that 
\begin{align*}
&(h^\eta_0)^{-\frac{1}{4}}Z_t(0)=(h^\eta_0)^{\frac{3}{4}}\triangle l_0=g^\eta_5,\quad  (h^\eta_0)^{\frac{1}{4}}\nabla Z_t(0)=g^\eta_6+\frac{1}{2}(h^\eta_0)^{-\frac{1}{2}}\nabla h^\eta_0g^\eta_5,\\
&|(h^\eta_0)^{-\frac{1}{4}}Z_t(0)|_2+| (h^\eta_0)^{\frac{1}{4}}\nabla Z_t(0)|_2\leq |g^\eta_5|_2+|g^\eta_6|_2+C|(h^\eta_0)^{-\frac{1}{2}}|_\infty|\nabla h^\eta_0|_\infty|g^\eta_5|_2\leq M(c_0),\\
&|Y_t(0)|_2=|(\phi_0+\eta)^{2\iota}\triangle u_0|_2\leq M(c_0),\\
&|(h^\eta_0)^{\frac{1}{2}}\nabla Y_t(0)|_2\leq C (|(h^\eta_0)^{\frac{3}{2}}\nabla^3 u_0|_2+|(h^\eta_0)^{-\frac{1}{2}}|_\infty|\nabla h^\eta_0|_\infty|h^\eta_0\nabla^2 u_0|_2)\leq M(c_0).
\end{align*}
Second, according to the classical energy estimates  on  transport equations, the weighted energy estimates for parabolic equations, and Remark \ref{r1}, one can  
choose  $\bar{T}\in (0,T^*]$ small enough such that \eqref{4.1*} and the uniform estimates (independent of $\eta$) \ef{2.16} hold for $(\phi^0,u^0,l^0,h^0,\psi^0=\fr{a\delta}{\delta-1}\nabla h^0)$ with $T^*$ replaced by $\bar{T}$.

\begin{proof} \textbf{Step 1:} Existence. One starts with the initial  iteration $(v,w,g)=(u^0,l^0, h^0)$, and has a strong solution $(\phi^1,u^1,l^1,h^1)$ to   \ef{ln} with $\epsilon=0$. Inductively, 
given $(u^k,l^k, h^k)$ for $k\geq 1$, define  $(\phi^{k+1},u^{k+1},l^{k+1},h^{k+1})$ by solving the  problem:
\begin{equation}\label{k+1}\left\{\begin{aligned}
\displaystyle
&\phi^{k+1}_t+u^k\cdot\nabla\phi^{k+1}+(\gamma-1)\phi^{k+1}\text{div}u^k=0,\\[2pt]
\displaystyle
&(l^{k+1})^{-\nu}(u_t^{k+1}+u^k\cdot\nabla u^k+a_1\phi^{k+1}\nabla l^{k+1}+l^{k+1}\nabla\phi^{k+1})\\[2pt]
\displaystyle
&+a_2h^{k+1}Lu^{k+1}
=a_2(l^{k+1})^{-\nu}h^k\nabla(l^{k+1})^{\nu}\cdot Q(u^k)+a_3\psi^{k+1}\cdot Q(u^k),\\[2pt]
\displaystyle
& (h^{k+1})^{-\frac{1}{2}}(l^{k+1}_t+u^k\cdot\nabla l^{k+1})-a_4(h^{k+1})^{\frac{1}{2}}(l^k)^\nu \triangle l^{k+1}\\
=&a_5(l^k)^\nu n^{k+1}(h^k)^{\frac{3}{2}}H(u^k)+a_6(l^k)^{\nu+1} (h^{k+1})^{-\frac{1}{2}}\text{div} \psi^{k+1}+\Pi^{k+1},\\[2pt]
\displaystyle
&h^{k+1}_t+u^k\cdot \nabla h^{k+1}+(\delta-1)h^k\text{div}u^k=0,\\[2pt]
\displaystyle
&(\phi^{k+1},u^{k+1},l^{k+1},h^{k+1})|_{t=0}=(\phi^\eta_0,u^\eta_0,l^\eta_0,h^\eta_0)\\[2pt]
=& (\phi_0+\eta,u_0,l_0,(\phi_0+\eta)^{2\iota}) \ \ \text{in} \ \ \mathbb{R}^3, \\[2pt]
\displaystyle
&(\phi^{k+1},u^{k+1},l^{k+1},h^{k+1})\longrightarrow (\eta,0,\bar{l},\eta^{2\iota}) \quad \text{as} \hspace{2mm}|x|\rightarrow \infty \quad \rm for\quad t\geq 0,
\end{aligned}\right.\end{equation}
where $\psi^{k+1}=\fr{a\delta}{\delta-1}\nabla h^{k+1}$, $ n^{k+1}=(a h^{k+1})^b$ and 
\begin{equation}\label{2.1mmm}
\begin{split}
\Pi^{k+1}=&a_7(l^k)^{\nu+1} (h^{k+1})^{-\frac{3}{2}}\psi^{k+1}\cdot \psi^{k+1}
+a_8(l^k)^\nu(h^{k+1})^{-\frac{1}{2}} \nabla l^{k+1}\cdot   \psi^{k+1} \\
&+a_9(l^k)^{\nu-1} (h^k)^{\frac{1}{2}}\nabla l^k\cdot \nabla l^k.
\end{split}
\end{equation}

It follows from Lemma \ref{epsilon0} with $(v,w, g)$ replaced by  $(u^k,l^k, h^k)$ and  mathematical induction that  one  can solve $\ef{k+1}$ locally in time to get  $(\phi^{k+1},u^{k+1},l^{k+1},\\ h^{k+1})$  satisfying   the uniform estimates \ef{key1kk}. Moreover, $\psi^{k+1}$ solves
\begin{equation}\label{k+2}
\psi^{k+1}_t+\nabla(u^k\cdot \psi^{k+1})+(\delta-1)\psi^k\text{div}u^k+a\delta h^k\nabla\text{div}u^k=0.
\end{equation}

To show the strong convergence of $(\phi^k,u^k,l^k,\psi^k)$, we set
\begin{equation*}
\begin{split}
&\bar{\phi}^{k+1}=\phi^{k+1}-\phi^k,\ \ \bar{u}^{k+1}=u^{k+1}-u^k,\ \ \bar{l}^{k+1}=l^{k+1}-l^k,\ \ \\
& \bar{\psi}^{k+1}=\psi^{k+1}-\psi^k,\ \  \bar{h}^{k+1}=h^{k+1}-h^k,\ \  \bar{n}^{k+1}=n^{k+1}-n^k.
\end{split}
\end{equation*}
All iterates with positive indices are generated by (3.120) with
the same initial data. Evaluating the velocity and entropy equations
at $t=0$, we see that $u_t^j(0)$ and $l_t^j(0)$ are independent of
$j\geq 1$. Consequently,
\[
\bar u_t^{k+1}(0)=0,\qquad
\bar l_t^{k+1}(0)=0,\qquad k\geq 1.
\]
These identities are not required for the difference between the
first iterate and the auxiliary zeroth iterate.
The difference system below holds for $k\geq 1$. In the final
contraction argument, we take $k\geq 4$, so that every shifted
difference equation used in the argument involves only iterates
with positive indices.

Then \eqref{k+1} and \ef{k+2} yield 
\begin{equation}\label{k+3}\left\{\begin{aligned}
\displaystyle
&\bar{\phi}^{k+1}_t+u^k\cdot\nabla\bar{\phi}^{k+1}+\bar{u}^k\cdot\nabla\phi^k+(\gamma-1)(\bar{\phi}^{k+1}\text{div}u^k+\phi^k\text{div}\bar{u}^k)=0,\\[4pt]
\displaystyle
&(l^{k+1})^{-\nu}\bar{u}_t^{k+1}
+a_2h^{k+1}L\bar{u}^{k+1}+a_2\bar{h}^{k+1}Lu^{k}
= \sum_{i=1}^4 \bar{\mathcal{U}}^{k+1}_i,\\
\displaystyle
\displaystyle
&(h^{k+1})^{-\frac{1}{2}}\bar{l}^{k+1}_t-a_4\sqrt{h^{k+1}}(l^k)^\nu \triangle \bar{l}^{k+1}
=\sum_{i=1}^4 \bar{\mathcal{L}}^{k+1}_i+\bar{\Pi}^{k+1},\\
\displaystyle
\displaystyle
&\bar{\psi}^{k+1}_t+\nabla(u^k\cdot \bar{\psi}^{k+1}+\bar{u}^k\cdot\psi^k)+(\delta-1)(\bar{\psi}^k\text{div}u^k+\psi^{k-1}\text{div}\bar{u}^k)\\[2pt]
\displaystyle
&+a\delta(h^k\nabla\text{div}\bar{u}^k+\bar{h}^k\nabla\text{div}u^{k-1})=0,\\[4pt]
\displaystyle
&(\bar{\phi}^{k+1},\bar{u}^{k+1},\bar{l}^{k+1},\bar{\psi}^{k+1})|_{t=0}=(0,0,0,0)\ \  \text{in}\ \ \mathbb{R}^3,
\\[4pt]
\displaystyle
&(\bar{\phi}^{k+1},\bar{u}^{k+1},\bar{l}^{k+1},\bar{\psi}^{k+1})\longrightarrow (0,0,0,0) \quad \text{as} \ \ |x|\rightarrow \infty \quad \rm for\quad t\geq 0,
\end{aligned}\right.\end{equation}
where 
\begin{align*}\label{hbar}
\bar{\mathcal{U}}^{k+1}_1=&-(l^{k+1})^{-\nu}(u^k\cdot\nabla \bar{u}^k+\bar{u}^k\cdot\nabla u^{k-1})\\
&-\big((l^{k+1})^{-\nu}-(l^k)^{-\nu}\big)(u_t^{k}+u^{k-1}\cdot\nabla u^{k-1}), \\
\bar{\mathcal{U}}^{k+1}_2=&-(l^{k+1})^{-\nu}(a_1\bar{\phi}^{k+1}\nabla l^{k+1}+a_1\phi^k\nabla\bar{l}^{k+1}+\bar{l}^{k+1}\nabla\phi^{k+1}+l^k\nabla\bar{\phi}^{k+1})
\\[2pt]
\displaystyle
&-\big((l^{k+1})^{-\nu}-(l^k)^{-\nu}\big)(a_1\phi^{k}\nabla l^{k}+l^{k}\nabla\phi^{k}), \\
\bar{\mathcal{U}}^{k+1}_3=& a_2(l^{k+1})^{-\nu}\Big(h^k\big(\nabla(l^{k+1})^{\nu}-\nabla(l^{k})^{\nu}\big)\cdot Q(u^k)
+h^k\nabla(l^k)^\nu\cdot Q(\bar{u}^k)\\[2pt]
\displaystyle
&+\bar{h}^k\nabla (l^k)^\nu\cdot Q(u^{k-1})\Big)
+a_3\bar{\psi}^{k+1}\cdot Q(u^k)+a_3\psi^{k}\cdot Q(\bar{u}^k), \\
\bar{\mathcal{U}}^{k+1}_4=&a_2\big((l^{k+1})^{-\nu}-(l^k)^{-\nu}\big)h^{k-1}\nabla(l^k)^\nu\cdot Q(u^{k-1}),  \\
\bar{\mathcal{L}}^{k+1}_1=&-(h^{k+1})^{-\frac{1}{2}}(u^k\cdot\nabla \bar{l}^{k+1}+ \bar{u}^{k}\cdot\nabla l^k)\\
&-((h^{k+1})^{-\frac{1}{2}}-(h^{k})^{-\frac{1}{2}})(l^{k}_t+u^{k-1}\cdot\nabla l^{k}),\\ 
\bar{\mathcal{L}}^{k+1}_2=&a_4\big(\sqrt{h^{k+1}}((l^k)^\nu-(l^{k-1})^\nu) +(\sqrt{h^{k+1}}-\sqrt{h^{k}}) (l^{k-1})^\nu\big)\triangle l^{k},\\ 
\bar{\mathcal{L}}^{k+1}_3=&a_5(l^k)^\nu n^{k+1}\big((h^k)^{\frac{3}{2}}(H(u^k)-H(u^{k-1}))+((h^k)^{\frac{3}{2}}-(h^{k-1})^{\frac{3}{2}})H(u^{k-1})\big)\\
\displaystyle
&+a_5(h^{k-1})^{\frac{3}{2}}H(u^{k-1})\big((l^k)^\nu\bar{n}^{k+1}+((l^k)^\nu-(l^{k-1})^\nu)n^{k}\big),\\
\bar{\mathcal{L}}^{k+1}_4=&a_6(l^k)^{\nu+1}\big( (h^{k+1})^{-\frac{1}{2}}\text{div} \bar{\psi}^{k+1}+ ((h^{k+1})^{-\frac{1}{2}}-(h^{k})^{-\frac{1}{2}})\text{div} \psi^{k}\big)\\
&+a_6((l^k)^{\nu+1}-(l^{k-1})^{\nu+1}) (h^{k})^{-\frac{1}{2}}\text{div} \psi^{k},\\
\displaystyle
\bar{\Pi}^{k+1}=&a_7(l^k)^{\nu+1}( (h^{k+1})^{-\frac{3}{2}}\bar{\psi}^{k+1}\cdot (\psi^{k+1}+\psi^{k})+((h^{k+1})^{-\frac{3}{2}}-(h^{k})^{-\frac{3}{2}})\psi^{k}\cdot \psi^{k})\\
&+a_7((l^k)^{\nu+1}-(l^{k-1})^{\nu+1}) (h^{k})^{-\frac{3}{2}}\psi^{k}\cdot \psi^{k}\\
&+a_8(l^k)^\nu(h^{k+1})^{-\frac{1}{2}} (\nabla l^{k+1}\cdot   \bar{\psi}^{k+1}+\nabla \bar{l}^{k+1}\cdot   \psi^{k})\\
&+a_8\big((l^k)^\nu((h^{k+1})^{-\frac{1}{2}}-(h^{k})^{-\frac{1}{2}})+((l^k)^\nu-(l^{k-1})^\nu)(h^{k})^{-\frac{1}{2}}\big) \nabla l^{k}\cdot   \psi^{k}\\
&+a_9(l^k)^{\nu-1} \sqrt{h^{k}}\nabla \bar{l}^k\cdot (\nabla l^k+\nabla l^{k-1})\\
&+a_9\big((l^k)^{\nu-1}(\sqrt{h^{k}}-\sqrt{h^{k-1}})+\sqrt{h^{k-1}}((l^k)^{\nu-1}-(l^{k-1})^{\nu-1}))|\nabla l^{k-1}|^2.
\end{align*}

Next, starting from  \eqref{k+3},  one will show  that  $\{(\phi^k,u^k,l^k,\psi^k)\}_{k=1}^\infty$ is actually a  Cauchy sequence in proper functional spaces, which requires  some  
 estimates for  
$\bar{\phi}^{k+1}\in H^2$, $\bar{\psi}^{k+1}\in H^1$, and $(\bar{u}^{k+1},\bar{l}^{k+1})$ in some suitable weighted $H^2$ spaces. For this purpose,   one first needs the following lemma. 
\begin{lemma}\label{cancel}
\begin{equation*}
\begin{split}
(\bar{h}^{k+1},\ \bar{\phi}^{k+1}) \in L^\infty([0,\bar{T}];H^3)\quad \text{and } \quad \bar{\psi}^{k+1}\in L^\infty([0,\bar{T}];H^2) \quad \text{for} \quad k=1,2,....
\end{split}
\end{equation*}
\end{lemma}
 The proof follows from the same argument for    Lemma 4.4 of \cite{GG}. This  lemma   helps to deal with some singular terms of  type  $\infty-\infty$ such as $a_2\bar{h}^{k+1}Lu^{k} $ in $\eqref{k+3}_2$.

\textbf{Step 1.1:} Estimates on $(\bar{\phi}^{k+1},\bar{\psi}^{k+1})$.
Actually, by  the standard energy estimates for  hyperbolic equations,   $\ef{k+3}_1$ and$\ef{k+3}_4 $ yield that 
\begin{equation}\label{phibar2bbb}
\begin{split}
\fr{d}{dt}\|\bar{\phi}^{k+1}\|^2_2\leq &C\sigma^{-1}\|\bar{\phi}^{k+1}\|^2_2+\sigma\|\nabla\bar{u}^k\|^2_1
+C| \nabla^3 \bar{u}^k |_2 |\bar{\phi}^{k+1} |_{D^2},\\
\fr{d}{dt}\|\bar{\psi}^{k+1}\|^2_1
\leq & C\sigma^{-1}\|\bar{\psi}^{k+1}\|^2_1+\sigma(|\sqrt{h^k}\nabla\bar{u}^k|^2_2+\|\bar{\psi}^{k}\|_1^2\\
&+|h^k\nabla^2\bar{u}^k|^2_2)+C|h^k \nabla^3 \bar{u}^k |_2|\nabla \bar{\psi}^{k+1} |_2,
\end{split}
\end{equation}
where $\sigma\in \big(0,\min\{1,\frac{a_4}{32},\frac{a_2 \alpha}{32}\}\big)$ is a constant   to be determined later.

\textbf{Step 1.2:} Estimates on $\bar{l}^{k+1}$.
 Multiplying $\ef{k+3}_3$ by $\bar{l}^{k+1}_t$ and integrating over $\mathbb{R}^3$ yield that
\begin{equation}\label{lcha1}
\begin{split}
\frac{a_4}{2} \frac{d}{dt}|(h^{k+1})^{\frac14} (l^k)^{\frac{\nu}{2}} \nabla \bar{l}^{k+1}|_2^2+|(h^{k+1})^{-\frac14}   \bar{l}_t^{k+1}|_2^2=\sum_{i=1}^5
N_i,
\end{split}
\end{equation}
where $N_i, i=1,2, \cdots, 5$ are defined and estimated as follows.
\begin{align}
\label{lcha2}
N_1=&\int\Big( \bar{\mathcal{L}}^{k+1}_1-a_4 \big( (l^k )^{\nu} \nabla \sqrt{h^{k+1}}+\sqrt{h^{k+1}} \nabla (l^k )^{\nu} \big) \cdot \nabla \bar{l}^{k+1} \Big) \bar{l}_t^{k+1} \nonumber\\
&+\int \frac{a_4}{2} \Big(  \big(\sqrt{h^{k+1}} \big)_t(l^k )^{\nu}+ \sqrt{h^{k+1}} \big( (l^k )^{\nu}\big)_t \Big) \nabla \bar{l}^{k+1} \cdot \nabla \bar{l}^{k+1} \nonumber\\
\le & C\Big(|(h^{k+1})^{\frac{1}{4}}(l^k)^{\frac{\nu}{2}} \nabla \bar{l}^{k+1} |_2
+|\bar{\psi}^{k+1}|_2
+|(l^k)^{-\frac{\nu}{2}} \bar{u}^k |_2 \Big)
\big|  (h^{k+1})^{-\frac{1}{4}} \bar{l}^{k+1}_t\big|_2\nonumber\\
&+ C(1+|l^k_t|^{\frac{1}{2}}_{D^2})  | (h^{k+1} )^{\frac{1}{4}} (l^k )^{\frac{\nu}{2}} \nabla \bar{l}^{k+1} |_2^2,\nonumber\\
N_2=&\int \bar{\mathcal{L}}^{k+1}_2 \bar{l}^{k+1}_t
 \le   C \big(   |(l^{k-1})^{\frac{\nu}{2}} (h^k)^{\frac{1}{4}} \nabla\bar{l}^k |_2|\bar{l}^{k+1}_t|_3+ |\bar{\psi}^{k+1}|_2  |(h^{k+1})^{-\frac{1}{4}}\bar{l}^{k+1}_t |_2\big),\nonumber\\
 N_3=& \int \bar{\mathcal{L}}^{k+1}_3 \bar{l}_t^{k+1} 
\leq C  \big( |h^k \nabla \bar{u}^k  |_6+ | \bar{\psi}^k  |_2\\
&+ | \bar{\psi}^{k+1} |_2+ |  (l^{k-1} )^{\frac{\nu}{2}}  (h^k )^{\frac{1}{4}} \nabla \bar{l}^k |_2 \big)| (h^{k+1} )^{-\frac{1}{4}} \bar{l}_t^{k+1} |_ 2,\nonumber\\
N_4=&\int\bar{\mathcal{L}}^{k+1}_4\bar{l}^{k+1}_t
\leq C \Big( \| \bar{\psi}^{k+1}  \|_1+
 |  (l^{k-1} )^{\frac{\nu}{2}}  (h^k )^{\frac{1}{4}} \nabla \bar{l}^{k} |_2 \Big) | (h^{k+1} )^{-\frac{1}{4}} \bar{l}_t^{k+1} |_2,\nonumber\\
 N_5=& \int \bar{\Pi}^{k+1} \cdot \bar{l}_t^{k+1} 
\le  C | (h^{k+1} )^{-\frac{1}{4}} \bar{l}_t^{k+1} |_2 ( | (l^{k-1} )^{\frac{\nu}{2}} (h^k )^{\frac{1}{4}} \nabla \bar{l}^k |_2\nonumber\\
&+ | (l^k )^{\frac{\nu}{2}} (h^{k+1} )^{\frac{1}{4}} \nabla \bar{l}^{k+1} |_2+ |\bar{\psi}^{k+1} |_2+ |\bar{\psi}^k |_2 ),\nonumber
\end{align}
where one has used the  facts that
\begin{align}\label{errorjiqiao}
 &|\sqrt{h^{k+1}}-\sqrt{h^{k}}|_6 \le C|\bar{\psi}^{k+1}|_2,\quad  |h^{k+1} (h^k)^{-1} |_{\infty}+ |h^k (h^{k+1})^{-1}|_{\infty} \le  C,\nonumber\\
 &|\bar{l}_t^{k+1}|_3\leq C
|(h^{k+1})^{-\frac{1}{4}}\bar{l}_t^{k+1}|^{\frac{1}{2}}_2|(h^{k+1})^{\frac{1}{4}}\bar{l}_t^{k+1}|^{\frac{1}{2}}_6,\nonumber\\
& | (h^{k+1} )^{\frac{1}{4}} \big( (h^k )^{\frac{3}{2}}- (h^{k-1} )^{\frac{3}{2}} \big) \nabla u^{k-1}\cdot \nabla u^{k-1}  |_2\nonumber\\
\le &   |  (h^{k+1} )^{\frac{1}{4}} \big( (h^k )^{\frac{3}{4}}- (h^{k-1} )^{\frac{3}{4}} \big) \big( (h^k )^{\frac{3}{4}}+ (h^{k-1} )^{\frac{3}{4}} \big) \nabla u^{k-1} \cdot \nabla u^{k-1}  |_2 \nonumber\\
\le & C |\bar{h}^k |_6  |h^{k-1}\nabla u^{k-1}  |_ 6  |\nabla u^{k-1}  |_6,\\
 &a^{-b}\nabla\bar{n}^{k+1}
=b\frac{\delta-1}{a\delta}\Big(((h^{k+1})^{b-1}-(h^{k})^{b-1})\psi^{k+1}+(h^{k})^{b-1}\bar{\psi}^{k+1}\Big),\nonumber\\
&a^{-b}|\bar{n}^{k+1}|_6=|(h^{k+1})^b-(h^{k})^b|_6\leq C|\bar{\psi}^{k+1}|_2,\nonumber\\
& (h^k )^{-\frac{3}{2}}- (h^{k-1} )^{-\frac{3}{2}}=((h^k )^{-\frac{3}{4}}- (h^{k-1} )^{-\frac{3}{4}})((h^k )^{-\frac{3}{4}}+ (h^{k-1} )^{-\frac{3}{4}}).\nonumber
\end{align}
It is noted  that the second estimate in $\ef{errorjiqiao}_2$ follows from Lemma \ref{gh}.
Then according to  \eqref{lcha1}-\eqref{lcha2}, one has  
\begin{equation}\label{nbarl}
\begin{split}
&\frac{a_4}{2}\frac{d}{dt}|(h^{k+1})^{\frac14} (l^k)^{\frac{\nu}{2}} \nabla \bar{l}^{k+1}|_2^2+|(h^{k+1})^{-\frac14}   \bar{l}_t^{k+1}|_2^2\\
\leq &  C\big((1+|\sqrt{h^k}\nabla^2 l^k_t|_{2})  | (h^{k+1} )^{\frac{1}{4}} (l^k )^{\frac{\nu}{2}} \nabla \bar{l}^{k+1} |_2^2+\sigma^{-3}| (h^{k+1} )^{-\frac{1}{4}} \bar{l}_t^{k+1} |^2_2\\
&+ \|\bar{\psi}^{k+1} \|^2_1\big) +\sigma(|\bar{\psi}^{k}|^2_2
+|(l^k)^{-\frac{\nu}{2}} \bar{u}^k |^2_2+|\sqrt{h^k}\nabla\bar{u}^k|^2_2+|h^k\nabla^2\bar{u}^{k}|^2_2\\
&+|  (l^{k-1} )^{\frac{\nu}{2}}  (h^k )^{\frac{1}{4}} \nabla \bar{l}^k |^2_2+|(h^{k+1})^{\frac{1}{4}}(l^k)^{\frac{\nu}{2}}\nabla\bar{l}_t^{k+1}|_2^2).
\end{split}
\end{equation}

Next, applying $\partial_t$ to $\ef{k+3}_3$ yields 
\begin{equation*}
\begin{aligned}
&(h^{k+1})^{-\frac{1}{2}}\bar{l}^{k+1}_{tt}-a_4(h^{k+1})^{\frac{1}{2}}(l^k)^\nu \triangle \bar{l}^{k+1}_t\\
=&-(h^{k+1})^{-\frac{1}{2}}_t\bar{l}^{k+1}_t+a_4 (\sqrt{h^{k+1}}(l^k)^{\nu})_t\triangle\bar{l}^{k+1}+\sum_{i=1}^4 (\bar{\mathcal{L}}^{k+1}_i)_t+\bar{\Pi}_t^{k+1}.
\end{aligned}
\end{equation*}
Then multiplying the above equation  by $\bar{l}^{k+1}_{t}$ and  integrating over $\mathbb{R}^3$, one has 
\begin{equation}\label{barl2}
\frac{1}{2}\frac{d}{dt}|(h^{k+1})^{-\frac{1}{4}}\bar{l}_t^{k+1}|^2_2+a_4|(h^{k+1})^{\frac{1}{4}}(l^k)^{\frac{\nu}{2}}\nabla\bar{l}_t^{k+1}|_2^2=\sum\limits_{i=6}^{11} N_i.
\end{equation}
Here  $N_i$, $i=6,7,\cdots,11$ are given and estimated as follows:
\begin{align}\label{barl21}
N_6=&\int\Big( -\frac{1}{2} ((h^{k+1})^{-\frac{1}{2}})_t(\bar{l}_t^{k+1})^2
+(\bar{\mathcal{L}}^{k+1}_1)_t \bar{l}_t^{k+1}\Big) \nonumber\\
\leq & C|(h^{k+1})^{-\frac{1}{4}}\bar{l}^{k+1}_t|_2^2+C\Big( |(l^k)^{\frac{\nu}{2}}(h^{k+1})^{\frac{1}{4}}\nabla\bar{l}^{k+1}|_2+|(l^k)^{-\frac{\nu}{2}}\bar{u}^k|_2\nonumber\\
&+|(h^{k+1})^{\frac{1}{4}}(l^k)^{\frac{\nu}{2}}\nabla\bar{l}^{k+1}_t|_2+|(l^k)^{-\frac{\nu}{2}}\bar{u}_t^k|_2+|\bar{u}^k|_\infty+|\bar{\psi}^{k+1}_t|_2\nonumber\\
&+|\bar{\psi}^{k+1}|_2\Big)|(h^{k+1})^{-\frac{1}{4}}\bar{l}^{k+1}_t|_2+C\Big(|(h^{k+1})^{\frac{1}{4}}(l^k)^{\frac{\nu}{2}}\nabla \bar{l}^{k+1}|_2\nonumber\\
&+|\bar{\psi}^{k+1}|_2(1+|(h^{k})^{-\frac{1}{4}}l^k_{tt}|_2)\Big)| (h^{k+1})^{-\frac{1}{4}}\bar{l}^{k+1}_t|_3,\nonumber\\
N_7=&\int \Big((\bar{\mathcal{L}}^{k+1}_2)_t+a_4\big(\sqrt{h^{k+1}}(l^k)^{\nu}\big)_t\triangle\bar{l}^{k+1}\Big)\bar{l}^{k+1}_t\nonumber\\
\leq &  C\Big( (1+|\sqrt{h^k}\nabla^2 l^k_t|_2)\big(|(h^k)^{\frac{1}{4}}(l^{k-1})^{\frac{\nu}{2}}\nabla\bar{l}^{k}|_2
\nonumber\\
&+|\sqrt{h^{k}}\nabla^2\bar{l}^{k}|_2+\|\bar{\psi}^{k+1}\|_1\big)+|(h^k)^{-\frac{1}{4}}\bar{l}^k_t|_2+|\bar{\psi}^{k+1}_t|_2\nonumber\\
&+|\sqrt{h^{k+1}} \nabla^2 \bar{l}^{k+1}|_2+| (h^k)^{\frac{1}{4}}\nabla\bar{l}^k_t|_2\Big)|(h^{k+1})^{-\frac{1}{4}}\bar{l}^{k+1}_t|_2\nonumber\\
&+C(|\sqrt{h^k}\nabla^2 l^k_t|_2|\nabla\bar{l}^{k}|_2+|\sqrt{h^{k+1}}\nabla^2 \bar{l}^{k+1}|_2)|\bar{l}^{k+1}_t|_3,\nonumber\\
N_{8}=&\int (\bar{\mathcal{L}}^{k+1}_3)_t\bar{l}^{k+1}_t
\leq  C\Big(|h^k\nabla \bar{u}^k|_6+|\nabla \bar{u}^k|_2+|\bar{\psi}^{k}_t|_2+|\bar{\psi}^{k+1}_t|_2\\
&+|(h^k)^{\frac{1}{4}}\nabla\bar{l}^{k}_t|_2
+\big(1+|(h^{k-1}\nabla^2 u^{k-1}_t,h^{k}\nabla^2 u^{k}_t)|_2\big)\big(|(h^{k})^{\frac{3}{4}}\nabla \bar{u}^k|_3\nonumber\\
&+|(h^k)^{\frac{1}{4}}(l^{k-1})^{\frac{\nu}{2}}\nabla\bar{l}^{k}|_2+|\bar{\psi}^{k+1}|_2+|\bar{\psi}^{k}|_2\big)\Big) |(h^{k+1})^{-\frac{1}{4}}\bar{l}^{k+1}_t|_2\nonumber\\
&+C|\sqrt{h^k}\nabla \bar{u}^k_t|_2|\bar{l}_t^{k+1}|_3,\nonumber\\
N_{9}=& \int (\bar{\mathcal{L}}^{k+1}_4)_t\bar{l}^{k+1}_t
\leq  C\Big(\| \bar{\psi}^{k+1}\|_1+|\bar{\psi}^{k+1}_t|_2+|(h^k)^{\frac{1}{4}}(l^{k-1})^{\frac{\nu}{2}}\nabla\bar{l}^{k}|_2\nonumber\\
&+|\sqrt{h^{k}}\nabla^2\bar{l}^{k}|_2+| (h^k)^{\frac{1}{4}}\nabla\bar{l}^k_t|_2\Big) |(h^{k+1})^{-\frac{1}{4}}\bar{l}^{k+1}_t|_2\nonumber\\
&+C\| \bar{\psi}^{k+1}\|_1|(h^{k+1})^{\frac{1}{4}}(l^k)^{\frac{\nu}{2}}\nabla\bar{l}_t^{k+1}|_2+N_*,\nonumber\\
N_{10}=&\int \bar{\Pi}_t^{k+1}\bar{l}^{k+1}_t
\leq  C\Big(\| \bar{\psi}^{k+1}\|_1+|\bar{\psi}^{k+1}_t|_2+|\bar{\psi}^{k}_t|_2+\|\bar{\psi}^{k}\|_1\nonumber\\
&+(1+|\sqrt{h^{k}}\nabla^2 l^k_t|_2)(|(h^k)^{\frac{1}{4}}(l^{k-1})^{\frac{\nu}{2}}\nabla\bar{l}^{k}|_2+|\sqrt{h^{k}}\nabla^2\bar{l}^{k}|_2)
\nonumber\\
&+| (h^k)^{\frac{1}{4}}\nabla\bar{l}^k_t|_2+(1+|\sqrt{h^k}\nabla^2 l^k_t|_2)|(h^{k+1})^{\frac14} (l^k)^{\frac{\nu}{2}} \nabla \bar{l}^{k+1}|_2\nonumber\\
&+|(h^{k})^{-\frac{1}{4}}\bar{l}^{k}_t|_2\Big) |(h^{k+1})^{-\frac{1}{4}}\bar{l}^{k+1}_t|_2
+C(\|\bar{\psi}^{k+1}\|_1+|(h^{k+1})^{-\frac{1}{4}}\bar{l}^{k+1}_t|_2\nonumber\\
&+|(h^{k+1})^{\frac14} (l^k)^{\frac{\nu}{2}} \nabla \bar{l}^{k+1}|_2)|(h^{k+1})^{\frac{1}{4}}(l^k)^{\frac{\nu}{2}}\nabla\bar{l}_t^{k+1}|_2,\nonumber\\
N_{11}=&-a_4\int\big(\nabla\sqrt{h^{k+1}}(l^k)^{\nu}\cdot \nabla \bar{l}_t^{k+1}+\sqrt{h^{k+1}}\nabla(l^k)^{\nu} \nabla \bar{l}_t^{k+1}\big)\bar{l}_t^{k+1}\nonumber\\
&\leq C |(h^{k+1})^{\frac{1}{4}}(l^k)^{\frac{\nu}{2}} \nabla \bar{l}_t^{k+1}|_2|(h^{k+1})^{-\frac{1}{4}}\bar{l}_t^{k+1}|_2,\nonumber
\end{align}
where one has used \ef{errorjiqiao}. By $\eqref{k+3}_4$,  the remaining term  $N_*$ in $N_{9}$ can be treated  as follows by the integration by parts,  
\begin{align}\label{yuxiang11}
N_*=&a_6\int (l^k)^{\nu+1}(h^{k+1})^{-\frac{1}{2}} \text{div} \bar{\psi}^{k+1}_t\bar{l}^{k+1}_t \nonumber\\
=&-a_6\int (l^k)^{\nu+1}(h^{k+1})^{-\frac{1}{2}}\bar{l}^{k+1}_t  \text{div} \Big(\nabla(u^k\cdot \bar{\psi}^{k+1})+\nabla(\bar{u}^k\cdot\psi^k)\nonumber\\
&+(\delta-1)(\bar{\psi}^k\text{div}u^k+\psi^{k-1}\text{div}\bar{u}^k)+a\delta(h^k\nabla\text{div}\bar{u}^k+\bar{h}^k\nabla\text{div}u^{k-1})\Big)\\
\leq & C\Big(\| \bar{\psi}^{k+1}\|_1+\| \bar{\psi}^{k}\|_1+|\sqrt{h^k}\nabla\bar{u}^k|_2+|h^k\nabla^2\bar{u}^{k}|_2\nonumber\\
&+|h^k\nabla^3\bar{u}^{k}|_2\Big)|(h^{k+1})^{-\frac{1}{4}}\bar{l}^{k+1}_t|_2+C\| \bar{\psi}^{k+1}\|_1|(h^{k+1})^{\frac{1}{4}}(l^k)^{\frac{\nu}{2}}\nabla\bar{l}_t^{k+1}|_2.\nonumber
\end{align}

Moreover, $\eqref{k+3}_4$ implies that  
\begin{equation}\label{psishijian}
\begin{split}
|\bar{\psi}^{k+1}_t|_2\leq &C\big(\|\bar{\psi}^{k+1}\|_1+\|\bar{u}^k\|_1+|\bar{\psi}^k|_2+|h^k\nabla^2 \bar{u}^k|_2\big).
\end{split}
\end{equation}
Then collecting estimates  \eqref{barl2}-\eqref{psishijian} yields that 
\begin{align}\label{barl26}
&\frac{1}{2}\frac{d}{dt}|(h^{k+1})^{-\frac{1}{4}}\bar{l}_t^{k+1}|^2_2+\frac{a_4}{2}|(h^{k+1})^{\frac{1}{4}}(l^k)^{\frac{\nu}{2}}\nabla\bar{l}_t^{k+1}|_2^2\nonumber\\
\leq &  C\sigma^{-2}(1+|\sqrt{h^k}\nabla^2 l^k_t|^2_{2}+|(h^{k})^{-\frac{1}{4}}l^k_{tt}|^2_2+|h^{k-1}\nabla^2 u^{k-1}_t|^2_2+|h^{k}\nabla^2 u^{k}_t|^2_2) \nonumber\\
&( \|\bar{\psi}^{k+1} \|^2_1+|(h^{k+1})^{-\frac{1}{4}}\bar{l}_t^{k+1}|^2_2+ | (h^{k+1} )^{\frac{1}{4}} (l^k )^{\frac{\nu}{2}} \nabla \bar{l}^{k+1} |_2^2) +\sigma(\|\bar{\psi}^{k}\|^2_1\nonumber\\
&+|(l^k)^{-\frac{\nu}{2}} \bar{u}^k |^2_2+|\sqrt{h^k}\nabla\bar{u}^k|^2_2+|(l^k)^{-\frac{\nu}{2}}\bar{u}_t^k|^2_2+|h^k\nabla^2\bar{u}^{k}|^2_2
+|(h^k)^{-\frac{1}{4}}\bar{l}^k_t|^2_2\\
&+(1+|\sqrt{h^k}\nabla^2 l^k_t|_2)|  (l^{k-1} )^{\frac{\nu}{2}}  (h^k )^{\frac{1}{4}} \nabla \bar{l}^k |^2_2+|\sqrt{h^{k}}\nabla^2\bar{l}^{k}|^2_2
+|\sqrt{h^{k+1}} \nabla^2 \bar{l}^{k+1}|^2_2\nonumber\\
&+\|\bar{u}^{k-1}\|^2_1+|\bar{\psi}^{k-1}|^2_2+|h^{k-1}\nabla^2 \bar{u}^{k-1}|^2_2)+C\tilde{\epsilon}^{-2}|(h^{k+1})^{-\frac{1}{4}}\bar{l}_t^{k+1}|^2_2\nonumber\\
&+\tilde{\epsilon} (| (h^k)^{\frac{1}{4}}\nabla\bar{l}^k_t|^2_2+|\sqrt{h^k}\nabla \bar{u}^k_t|^2_2)+C|h^k\nabla^3\bar{u}^{k}|_2|(h^{k+1})^{-\frac{1}{4}}\bar{l}^{k+1}_t|_2,\nonumber
\end{align}
where  $\tilde{\epsilon}\in (0,1)$ is a  constant to  be determined later.

\textbf{Step 1.3:} Estimates on $\bar{u}^{k+1}$.
 First, multiplying $\ef{k+3}_2$ by $2\bar{u}^{k+1}$ and integrating over $\mathbb{R}^3$ yield   that
\begin{equation}\label{ubar}
\begin{split}
&\fr{d}{dt}|(l^{k+1})^{-\fr{\nu}{2}}\bar{u}^{k+1}|^2_2+a_2\alpha|\sqrt{h^{k+1}}\nabla\bar{u}^{k+1}|^2_2\\
\leq& C\sigma^{-1}(1+|\nabla^2 l^{k+1}_t|_2)|(l^{k+1})^{-\fr{\nu}{2}}\bar{u}^{k+1}|^2_2+\sigma(|\sqrt{h^{k}}\nabla\bar{u}^{k}|^2_2
+|h^k\nabla^2\bar{u}^k|^2_2\\
&+|\bar{\psi}^{k}|^2_2)+C(\|\bar{\phi}^{k+1}\|^2_1+|\bar{\psi}^{k+1}|^2_2+|\nabla \bar{l}^{k+1}|^2_2).
\end{split}
\end{equation}

Second, multiplying $\ef{k+3}_2$ by $2\bar{u}_t^{k+1}$ and integrating over $\mathbb{R}^3$ give  that
\begin{equation}\label{utbar1}
\begin{split}
&|(l^{k+1})^{-\fr{\nu}{2}}\bar{u}^{k+1}_t|_2^2+a_2\fr{d}{dt}(\alpha|\sqrt{h^{k+1}}\nabla\bar{u}^{k+1}|^2_2+(\alpha+\beta)|\sqrt{h^{k+1}}\text{div}\bar{u}^{k+1}|^2_2)\\
\leq& C(|\sqrt{h^{k+1}}\nabla\bar{u}^{k+1}|^2_2+\|\bar{\phi}^{k+1}\|^2_1+|(h^{k+1})^{\frac14} (l^k)^{\frac{\nu}{2}} \nabla \bar{l}^{k+1}|^2_2+|\bar{\psi}^{k+1}|^2_2\\
&+\sigma^{-1} |(l^{k+1})^{-\fr{\nu}{2}}\bar{u}^{k+1}_t|^2_2  )
+\sigma(|\sqrt{h^k}\nabla\bar{u}^k|^2_2+|\bar{\psi}^k|^2_2).
\end{split}
\end{equation}

Next, applying $\partial_t$ to $\ef{k+3}_2$ gives 
\begin{equation*}
\begin{aligned}
&(l^{k+1})^{-\nu} \bar{u}_{tt}^{k+1}+a_2h^{k+1}L\bar{u}_t^{k+1}\\
=&-((l^{k+1})^{-\nu})_t\bar{u}_{t}^{k+1}-a_2h^{k+1}_t L\bar{u}^{k+1}-a_2 (\bar{h}^{k+1}Lu^{k} )_t
+ \sum_{i=1}^4 (\bar{\mathcal{U}}^{k+1}_i)_t.
\end{aligned}
\end{equation*}
Then multiplying  above system  by $2\bar{u}_{t}^{k+1}$ and integrating over $\mathbb{R}^3$ lead to 
\begin{equation}\label{erjieuu}
\begin{aligned}
&  \frac{d}{dt}|(l^{k+1})^{-\frac{\nu}{2}}\bar{u}_t^{k+1}|^2_2+ 2a_2\alpha |\sqrt{h^{k+1}}\nabla\bar{u}_t^{k+1}|_2^2 +2a_2(\alpha+\beta)|\sqrt{h^{k+1}}\text{div}\bar{u}_t^{k+1}|_2^2= \sum\limits_{i=1}^{5} O_i,
 \end{aligned}
\end{equation}
  where $O_i, i=1,2,\cdots,5$ are given and estimated as follows:
\begin{align}\label{j7ooouuu}
 O_1=&\int (-(l^{k+1})^{-\nu})_t  (\bar{u}_t^{k+1} )^2+(2\bar{\mathcal{U}}^{k+1}_1)_t\cdot\bar{u}_{t}^{k+1})\nonumber\\
 \le &  C |l^{k+1}_t|^{\frac{1}{2}}_{D^2}|(l^{k+1})^{-\frac{\nu}{2}}\bar{u}_t^{k+1}|^2_2 +  C \big( |\sqrt{h^{k}}\nabla\bar{u}_t^{k}|_2+ |(h^{k+1})^{-\frac{1}{4}}\bar{l}_t^{k+1}|_2\nonumber\\
   & +\|\nabla \bar{u}^k\|_1+  (1+|u^k_{tt}|_2)| (l^k)^{\frac{\nu}{2}} (h^{k+1})^{\frac14}\nabla\bar{l}^{k+1} |_2 \big) |(l^{k+1})^{-\frac{\nu}{2}}\bar{u}_t^{k+1}|_2\nonumber\\
 &+C(1+|u^k_{tt}|_2)| (l^k)^{\frac{\nu}{2}} (h^{k+1})^{\frac14}\nabla\bar{l}^{k+1} |_2 |\sqrt{h^{k+1}}\nabla\bar{u}_t^{k+1}|_2\nonumber\\
 &+C|(h^{k+1})^{-\frac{1}{4}}\bar{l}_t^{k+1}|_2|\sqrt{h^{k+1}} \nabla\bar{u}_t^{k+1}|_2,\nonumber\\
  \displaystyle
 O_{2}=&\int - 2a_2 \Big( \alpha (\nabla h^{k+1} \cdot \nabla)  \bar{u}^{k+1}_t+(\alpha+\beta)\nabla h^{k+1} \text{div}\bar{u}^{k+1}_t\nonumber\\
 &+ h^{k+1}_t L\bar{u}^{k+1} + (\bar{h}^{k+1}Lu^{k} )_t\Big)\cdot \bar{u}_{t}^{k+1}\nonumber\\
 \displaystyle
 \le &  C \big(  |\sqrt{h^{k+1}} \nabla\bar{u}_t^{k+1}|_2+ |\nabla^2 \bar{u}^{k+1}|_2 + |\bar{\psi}^{k+1}_t|_2\nonumber\\
 &+|\nabla^2 u^k_t|_{2}\|\bar{\psi}^{k+1}\|_1\big) |(l^{k+1})^{-\frac{\nu}{2}}\bar{u}_t^{k+1}|_2,\nonumber\\
  O_{3}=&\int 2 (\bar{\mathcal{U}}^{k+1}_2)_t\cdot \bar{u}_{t}^{k+1} 
 \le   C\big(  \| \bar{\phi}^{k+1}\|_2  + \|  \bar{\phi}^{k+1}_t \|_1+ |(h^{k+1})^{-\frac{1}{4}}\bar{l}_t^{k+1}|_2   \\
 \displaystyle
 &  +(1+|l^k_t|^{\frac{1}{2}}_{D^2}+|l^{k+1}_t|^{\frac{1}{2}}_{D^2}) | (l^k)^{\frac{\nu}{2}} (h^{k+1})^{\frac14}\nabla\bar{l}^{k+1} |_2   \big) |(l^{k+1})^{-\frac{\nu}{2}}\bar{u}_t^{k+1}|_2\nonumber\\
 \displaystyle
 & +C|(h^{k+1})^{-\frac{1}{4}}\bar{l}_t^{k+1}|_2|\sqrt{h^{k+1}} \nabla\bar{u}_t^{k+1}|_2,\nonumber\\
 O_{4}=&\int 2(\bar{\mathcal{U}}^{k+1}_3)_t\cdot \bar{u}_{t}^{k+1}
 \le  C \big( (1+|l^{k+1}_t|_{D^2}\nonumber\\
 &+|h^k \nabla^2 u^k_t|_2)|(h^{k+1})^{\frac{1}{4}}(l^{k})^{\frac{\nu}{2}}\nabla\bar{l}^{k+1}|_2+ |(h^{k+1})^{-\frac{1}{4}}\bar{l}_t^{k+1}|_2 
 \nonumber\\
 &+ (1+|l^{k}_t|^{\frac{1}{2}}_{D^2})(\|\nabla \bar{u}^k\|_1+|h^k \nabla^2 \bar{u}^k|_2+\|\bar{\psi}^{k}\|_1)\nonumber\\
 \displaystyle
&+ | \sqrt{h^k}\nabla \bar{u}^k_t|_2+|\bar{\psi}^k_t|_2+|\bar{\psi}^{k+1}_t|_2+ \| \bar{\psi}^{k+1}\|_1 \big)|(l^{k+1})^{-\frac{\nu}{2}}\bar{u}_t^{k+1}|_2\nonumber\\
\displaystyle
&+C((1+|h^k \nabla^2 u^k_t|_2)|(h^{k+1})^{\frac{1}{4}}(l^{k})^{\frac{\nu}{2}}\nabla\bar{l}^{k+1}|_2\nonumber\\
\displaystyle
&+|(h^{k+1})^{-\frac{1}{4}}\bar{l}_t^{k+1}|_2+ \| \bar{\psi}^{k+1}\|_1)|\sqrt{h^{k+1}}\nabla\bar{u}_t^{k+1}|_2,\nonumber\\
\displaystyle
 O_{5}=&\int 2(\bar{\mathcal{U}}^{k+1}_4)_t \cdot \bar{u}_{t}^{k+1}
\le C|(h^{k+1})^{\frac{1}{4}}(l^{k})^{\frac{\nu}{2}}\nabla\bar{l}^{k+1}|_2|\sqrt{h^{k+1}} \nabla\bar{u}_t^{k+1}|\nonumber\\
&+C \big(|(h^{k+1})^{-\frac{1}{4}}\bar{l}_t^{k+1}|_2 + |(h^{k+1})^{\frac{1}{4}}(l^{k})^{\frac{\nu}{2}}\nabla\bar{l}^{k+1}|_2 \big) |(l^{k+1})^{-\frac{\nu}{2}}\bar{u}_t^{k+1}|_2,\nonumber
\end{align}
where one has used  integration by parts in $O_{3}$ and $O_{4}$ to deal with   $\nabla\bar{l}_t^{k+1}$.

It follows from \ef{erjieuu}-\ef{j7ooouuu}, \ef{psishijian},  $\ef{k+3}_1$
and Young's inequality that
\begin{align}\label{utbar1llll}
&\frac{d}{dt}|(l^{k+1})^{-\frac{\nu}{2}}\bar{u}_t^{k+1}|^2_2+ a_2\alpha |\sqrt{h^{k+1}}\nabla\bar{u}_t^{k+1}|_2^2\nonumber\\
\le &  C\sigma^{-1}  (1+|l^{k+1}_t|^2_{D^2}+|l^{k}_t|_{D^2}+|h^k \nabla^2 u^k_t|^2_2+|u^k_{tt}|^2_2)( \| \bar{\psi}^{k+1}\|^2_1
 +\| \bar{\phi}^{k+1}\|^2_2\nonumber\\
&+|(h^{k+1})^{\frac{1}{4}}(l^{k})^{\frac{\nu}{2}}\nabla\bar{l}^{k+1}|^2_2+ |(h^{k+1})^{-\frac{1}{4}}\bar{l}_t^{k+1}|^2_2 +|(l^{k+1})^{-\frac{\nu}{2}}\bar{u}_t^{k+1}|^2_2)\\
 &+C\sigma( |(l^k)^{-\frac{\nu}{2}} \bar{u}^k |^2_2+|\sqrt{h^{k}}\nabla\bar{u}^{k}|^2_2+|h^k \nabla^2 \bar{u}^k|^2_2
 +|h^{k+1} \nabla^2 \bar{u}^{k+1} |^2_2\nonumber\\
 &+\|\bar{\psi}^{k}\|^2_1+|\bar{\psi}^{k-1}|^2_2+\|\bar{u}^{k-1}\|^2_1+|h^{k-1}\nabla^2 \bar{u}^{k-1}|^2_2 \big)\nonumber\\
 &+C\tilde{\epsilon}^{-1}|(l^{k+1})^{-\frac{\nu}{2}}\bar{u}_t^{k+1}|^2_2+\tilde{\epsilon}|\sqrt{h^k}\nabla \bar{u}^k_t|^2_2.\nonumber
\end{align}

\textbf{Step 1.4:} Strong convergences of the approximation solutions. For simplicity, we denote
 $\mathcal{T}^k=(| (h^k)^{\frac{1}{4}}\nabla\bar{l}^k_t|^2_2+|\sqrt{h^k}\nabla \bar{u}^k_t|^2_2)$ and 
\begin{equation*}
\begin{split}
\mathcal{I}^k (t)=&
C\sigma^{-3}(1+|\sqrt{h^k}\nabla^2 l^k_t|^2_{2}+|\sqrt{h^{k+1}}\nabla^2 l^{k+1}_t|^2_{2}+|(h^{k})^{-\frac{1}{4}}l^k_{tt}|^2_2\\
&+|h^{k-1}\nabla^2 u^{k-1}_t|^2_2+|h^{k}\nabla^2 u^{k}_t|^2_2+|u^k_{tt}|^2_2).
\end{split}
\end{equation*}
By the same arguments used in  the derivations of \ef{nabla2hl}, \ef{2.61i} and \ef{3jieu}, it follows directly from  $\ef{k+3}_2$-$\ef{k+3}_3$ and Lemma \ref{zhenok} that 
\begin{align*}
|\sqrt{h^{k+1}}\nabla^2\bar{l}^{k+1}|_2\leq& C(|(h^{k+1})^{-\fr{1}{4}}\bar{l}^{k+1}_t|_2+|\sqrt{h^{k}}\nabla\bar{u}^{k}|_2+| (l^{k} )^{\frac{\nu}{2}}  (h^{k+1} )^{\frac{1}{4}} \nabla \bar{l}^{k+1}|_2\\
&+| (l^{k-1} )^{\frac{\nu}{2}}  (h^{k} )^{\frac{1}{4}} \nabla \bar{l}^{k}|_2+|\bar{\psi}^{k}|_2+\|\bar{\psi}^{k+1}\|_1),\\
|h^{k+1}\nabla^2\bar{u}^{k+1}|_2\leq& C(|(l^{k+1})^{-\fr{\nu}{2}}\bar{u}^{k+1}_t|_2+|\sqrt{h^{k+1}}\nabla\bar{u}^{k+1}|_2+|\sqrt{h^{k}}\nabla\bar{u}^{k}|_2\\
&+\|\bar{\phi}^{k+1}\|_1+|\bar{\psi}^{k}|_2+|\bar{\psi}^{k+1}|_2
+| (l^{k} )^{\frac{\nu}{2}}  (h^{k+1} )^{\frac{1}{4}} \nabla \bar{l}^{k+1}|_2),\\
|h^{k+1}\nabla^3\bar{u}^{k+1}|_2\leq& C(|(l^{k+1})^{-\fr{\nu}{2}}\bar{u}^{k+1}_t|_2+|\sqrt{h^{k+1}}\nabla\bar{u}^{k+1}_t|_2+|\sqrt{h^{k+1}}\nabla\bar{u}^{k+1}|_2\\
&+|\sqrt{h^{k}}\nabla\bar{u}^{k}|_2+|(h^{k+1})^{-\fr{1}{4}}\bar{l}^{k+1}_t|_2+| (l^{k-1} )^{\frac{\nu}{2}}  (h^{k} )^{\frac{1}{4}} \nabla \bar{l}^{k}|_2\\
&+\|\bar{\phi}^{k+1}\|_2+\|\bar{\psi}^{k}\|_1+\|\bar{\psi}^{k+1}\|_1
+| (l^{k} )^{\frac{\nu}{2}}  (h^{k+1} )^{\frac{1}{4}} \nabla \bar{l}^{k+1}|_2\\
&+|(l^{k})^{-\fr{\nu}{2}}\bar{u}^{k}_t|_2+|\sqrt{h^{k-1}}\nabla\bar{u}^{k-1}|_2+\|\bar{\phi}^{k}\|_1+|\bar{\psi}^{k-1}|_2),
\end{align*}
which, along with  \ef{phibar2bbb},    \ef{nbarl}, \ef{barl26}, \ef{ubar}, \ef{utbar1} and \ef{utbar1llll}, yields  that 
\begin{align}\label{allbar}
&\fr{d}{dt}(\|\bar{\phi}^{k+1}\|^2_2+\|\bar{\psi}^{k+1}\|^2_1+|(h^{k+1})^{\frac14} (l^k)^{\frac{\nu}{2}} \nabla \bar{l}^{k+1}|_2^2+|(h^{k+1})^{-\frac{1}{4}}\bar{l}_t^{k+1}|^2_2\nonumber\\
&+|(l^{k+1})^{-\fr{\nu}{2}}\bar{u}^{k+1}|^2_2+|\sqrt{h^{k+1}}\nabla\bar{u}^{k+1}|^2_2+|\sqrt{h^{k+1}}\text{div}\bar{u}^{k+1}|^2_2+|(l^{k+1})^{-\frac{\nu}{2}}\bar{u}_t^{k+1}|^2_2)\nonumber\\
&+|(h^{k+1})^{-\frac14}   \bar{l}_t^{k+1}|_2^2+|(h^{k+1})^{\frac{1}{4}}(l^k)^{\frac{\nu}{2}}\nabla\bar{l}_t^{k+1}|_2^2+|\sqrt{h^{k+1}}\nabla\bar{u}^{k+1}|^2_2\nonumber\\
&+|(l^{k+1})^{-\fr{\nu}{2}}\bar{u}_t^{k+1}|^2_2+ |\sqrt{h^{k+1}} \nabla\bar{u}_t^{k+1}|_2^2\nonumber\\
\leq & \mathcal{I}^k (t)(\|\bar{\phi}^{k+1}\|^2_2+\|\bar{\psi}^{k+1}\|^2_1+|(h^{k+1})^{\frac14} (l^k)^{\frac{\nu}{2}} \nabla \bar{l}^{k+1}|_2^2+|(h^{k+1})^{-\frac{1}{4}}\bar{l}_t^{k+1}|^2_2\nonumber\\
&+|(l^{k+1})^{-\fr{\nu}{2}}\bar{u}^{k+1}|^2_2+|\sqrt{h^{k+1}}\nabla\bar{u}^{k+1}|^2_2+|(l^{k+1})^{-\frac{\nu}{2}}\bar{u}_t^{k+1}|^2_2)\\
&+C\sigma\big(|\sqrt{h^k}\nabla\bar{u}^k|^2_2+\|\bar{\phi}^k\|^2_1+\|\bar{\psi}^k\|^2_1
+|(l^{k})^{-\fr{\nu}{2}}\bar{u}^{k}|^2_2+|\sqrt{h^{k-1}}\nabla\bar{u}^{k-1}|^2_2\nonumber\\
&+(1+|\sqrt{h^k}\nabla^2 l^k_t|_2)|  (l^{k-1} )^{\frac{\nu}{2}}  (h^k )^{\frac{1}{4}} \nabla \bar{l}^k |^2_2+|(l^{k})^{-\fr{\nu}{2}}\bar{u}_t^{k}|^2_2+|(h^k)^{-\frac{1}{4}}\bar{l}^k_t|^2_2
\nonumber\\
&+|\bar{\psi}^{k-1}|_2^2+\|\bar{u}^{k-1}\|^2_1+|  (l^{k-2} )^{\frac{\nu}{2}}  (h^{k-1} )^{\frac{1}{4}} \nabla \bar{l}^{k-1} |^2_2+|(l^{k-1})^{-\fr{\nu}{2}}\bar{u}^{k-1}_t|^2_2\nonumber\\
&+|\sqrt{h^{k-2}}\nabla\bar{u}^{k-2}|^2_2+\|\bar{\phi}^{k-1}\|^2_1+|\bar{\psi}^{k-2}|^2_2\big)+C\tilde{\epsilon}^{-2}(|(h^{k+1})^{-\frac{1}{4}}\bar{l}_t^{k+1}|^2_2\nonumber\\
&+|(l^{k+1})^{-\fr{\nu}{2}}\bar{u}_t^{k+1}|^2_2+\|\bar{\phi}^{k+1}\|^2_2+\|\bar{\psi}^{k+1}\|^2_1)+\tilde{\epsilon}\mathcal{T}^k.\nonumber
\end{align}

Now, define 
\begin{equation*}
\begin{split}
\varGamma^{k+1}(t)=&\sup_{0\leq s \leq t}\|\bar\phi^{k+1}\|_2^2+\sup_{0\leq s \leq t}\|\bar\psi^{k+1}\|_1^2+\sup_{0\leq s \leq t}|(h^{k+1})^{\frac14} (l^k)^{\frac{\nu}{2}} \nabla \bar{l}^{k+1}|_2^2\\
&+\sup_{0\leq s \leq t}|(h^{k+1})^{-\frac{1}{4}}\bar{l}_t^{k+1}|^2_2+\sup_{0\leq s \leq t}|(l^{k+1})^{-\fr{\nu}{2}}\bar{u}^{k+1}|^2_2\\
&+\sup_{0\leq s \leq t}|\sqrt{h^{k+1}}\nabla\bar{u}^{k+1}|_2^2
+\sup_{0\leq s \leq t}|(l^{k+1})^{-\frac{\nu}{2}}\bar{u}_t^{k+1}|^2_2,\\
\Xi^{k+1}(t)=&|(h^{k+1})^{-\frac14}   \bar{l}_t^{k+1}|_2^2+|(h^{k+1})^{\frac{1}{4}}(l^k)^{\frac{\nu}{2}}\nabla\bar{l}_t^{k+1}|_2^2\\
&+|\sqrt{h^{k+1}}\nabla\bar{u}^{k+1}|^2_2+|(l^{k+1})^{-\fr{\nu}{2}}\bar{u}_t^{k+1}|^2_2+ |\sqrt{h^{k+1}}\nabla\bar{u}_t^{k+1}|_2^2.
\end{split}
\end{equation*}

Then it follows from \ef{allbar}  and Gronwall's inequality  that 
\begin{equation}\label{Gamma}
\begin{split}
&\Omega^{k+1}(t)=\varGamma^{k+1}(t)+\int^t_0 \Xi^{k+1}(s)\text{d}s\\
\leq&C\Big(\int_0^t\tilde{\epsilon}(|\sqrt{h^k}\nabla \bar{u}^k_s|^2_2
+| (h^k)^{\frac{1}{4}}\nabla\bar{l}^k_s|^2_2)\text{d}s+(t+\sqrt{t})\sigma\varGamma^{k}(t)\\
&+t\sigma\varGamma^{k-1}(t)+t\sigma\varGamma^{k-2}(t)\Big)\exp{(C\sigma^{-3}t+C\sigma^{-3}+C\tilde{\epsilon}^{-2}t)}.
\end{split}
\end{equation}
One can  choose  $\sigma \in \big(0,\min\{1,\frac{a_4}{32},\frac{a_2 \alpha}{32}\}\big)$,  $T_*\in (0,\min\{1,\bar{T}\}]$ and $\tilde{\epsilon} \in (0,1)$ such that
\begin{equation*}
\begin{split}
C\tilde{\epsilon}\exp{(C\sigma^{-3}T_*+C\sigma^{-3}+C\tilde{\epsilon}^{-2}T_*)}\leq & \fr{1}{32},\\
C\sqrt{T_*} \sigma \exp{(C\sigma^{-3}T_*+C\sigma^{-3}+C\tilde{\epsilon}^{-2}T_*)}\leq & \fr{1}{32},
\end{split}
\end{equation*}
and then there exists a positive constant $\Lambda<\frac{1}{3}$ such that 
\begin{equation}\label{Gamma1}
\begin{split}
\Omega^{k+1}\leq \Lambda (\Omega^{k}+\Omega^{k-1}+\Omega^{k-2}).
\end{split}
\end{equation}
Here and below, we write $\Omega^j=\Omega^j(T_*)$.
By (3.7), the common initial data, and the comparability of
consecutive weights in (3.127), we have
\[
\Omega^2+\Omega^3+\Omega^4\leq C,
\]
where the positive constant $C$ is independent of $\eta$.

Set $r=3\Lambda\in(0,1)$ and
\[
M_k=\max\{\Omega^k,\Omega^{k-1},\Omega^{k-2}\},
\qquad k\geq 4.
\]
It follows from (3.141) that
\[
M_{k+3}\leq rM_k,\qquad k\geq 4.
\]
Hence, by induction,
\[
M_{4+3p}\leq r^pM_4,\qquad p=0,1,2,\ldots.
\]
With $r_0=r^{1/3}\in(0,1)$, this implies
\[
\Omega^k\leq Cr_0^k,\qquad k\geq 2.
\]
In particular,
\[
\sum_{k=2}^{\infty}(\Omega^k)^{1/2}<\infty.
\]
Together with the uniform estimates (3.7), this yields the
Cauchy property of the sequence of iterates with positive indices:
\begin{equation}\label{barpsi}
\begin{split}
\lim_{k\rightarrow \infty}(\|\bar\phi^{k+1}\|_{s'}+\|\bar u^{k+1}\|_{s'}+\|\bar l^{k+1}\|_{L^\infty\cap D^1\cap D^{s'}})=&0,\\
\lim_{k\rightarrow \infty}(|\bar u^{k+1}_t|_2+|\bar l^{k+1}_t|_2+\|\bar\psi^{k+1}\|_{L^\infty\cap L^q}+|\bar h^{k+1}|_\infty)=&0,
\end{split}
\end{equation}
 for any $s'\in [1,3)$.
Then   there exist a subsequence (still denoted by $(\phi^k,u^k, l^k, \psi^k)$) and   limit functions  $(\phi^\eta,u^\eta,l^\eta,\psi^\eta)$ such that
\begin{equation}\label{key0}
\begin{split}
&(\phi^k-(\phi_0+\eta),u^k)\rightarrow  (\phi^\eta-(\phi_0+\eta), u^\eta) \ \ \text{in}\ \ L^\infty([0,T_*];H^{s'}),\\
 & l^k-\bar{l}\rightarrow   l^\eta-\bar{l}\ \ \text{in}\ \ L^\infty([0,T_*];L^\infty\cap D^1\cap D^{s'}),\\
 & (u^k_t,l^k_t)\rightarrow  (u^\eta_t,l^\eta_t)\ \ \text{in}\ \ L^\infty([0,T_*];L^2),\\
 &\psi^k\rightarrow  \psi^\eta \ \ \text{in}\ \ L^\infty([0,T_*];L^\infty\cap L^q),\\ &\ h^k\rightarrow  h^\eta \ \ \text{in}\ \ L^\infty([0,T_*];L^\infty).
\end{split}
\end{equation}
 Again due to   \ef{key1kk},  there exists a subsequence (still denoted by $(\phi^k,u^k, l^k, \psi^k)$) converging  to $(\phi^\eta,u^\eta,l^\eta,\psi^\eta)$ in the weak or weak* sense.
According to the lower semi-continuity of norms,  the  estimates in \ef{key1kk} still hold  for  $(\phi^\eta,u^\eta, 
l^\eta,\psi^\eta)$,  which are independent of $\eta$. Moreover, 
by the initial assumption that $
 \psi^\eta_0=\frac{a\delta}{\delta-1}\nabla (\phi^\eta_0)^{2\iota}$, one can easily verify  that 
$$
 \psi^\eta=\frac{a\delta}{\delta-1}\nabla (\phi^\eta)^{2\iota}.
$$
Thus,  one easily gets that $(\phi^\eta,u^\eta, l^\eta,\psi^\eta)$ is a weak solution in the sense of distributions to \eqref{nl} satisfying \ef{key1kk}.

\textbf{Step 2:} Uniqueness.  Let $(\phi_1,  u_1, l_1,\psi_1)$ and $(\phi_2, u_2, l_2,\psi_2)$ be two strong solutions to the   Cauchy problem (\ref{nl}) satisfying the  estimates in \ef{key1kk}. Set
\begin{equation*}
\begin{split}
h_i=&\phi^{2\iota}_i,\quad n_i=(ah_i)^b,\quad  i=1,2; \quad \bar{h}=h_1-h_2,\\
\bar{\phi}=& \phi_1-\phi_2,\  \    \bar{u}=u_1-u_2, \ \  \bar{l}=l_1-l_2,\ \  \bar{\psi}=\psi_1-\psi_2.
\end{split}
\end{equation*}
Then  (\ref{nl}) implies that
 \begin{equation}
\label{eq:1.2wcvb}
\begin{cases}
  \displaystyle
\quad \bar{\phi}_t+u_1\cdot \nabla\bar{\phi}+\bar{u} \cdot\nabla\phi_2+(\gamma-1)(\bar{\phi}\text{div}u_1 +\phi_2\text{div}\bar{u})=0,\\[8pt]
 \displaystyle
\quad  \bar{u}_t+ u_1\cdot\nabla \bar{u}+l_1\nabla \bar{\phi}+a_1\phi_1\nabla \bar{l}+ a_2l^\nu_1 h_1L\bar{u} \\[8pt]
\displaystyle
=- \bar{u} \cdot \nabla u_2-a_1\bar{\phi} \nabla l_2-\bar{l} \nabla \phi_2- a_2(l^\nu_1 h_1-l^\nu_2 h_2)Lu_2\\[8pt]
\displaystyle
\quad +a_2( h_1 \nabla l^\nu_1 \cdot Q(u_1)-h_2\nabla l^\nu_2 \cdot Q(u_2))\\[8pt]
\displaystyle
\quad +a_3(l^\nu_1 \psi_1 \cdot Q(u_1)-l^\nu_2 \psi_2 \cdot Q(u_2)),\\[8pt]
\quad \phi^{-\iota}_1(\bar{l}_t+u_1\cdot\nabla \bar{l}+\bar{u}\cdot \nabla l_2)-a_4\phi^{\iota}_1l^\nu_1 \triangle \bar{l}\\[8pt]
\displaystyle
=-(\phi^{-\iota}_1-\phi^{-\iota}_2)((l_2)_t+u_2\cdot\nabla l_2)+ a_4(\phi^{\iota}_1l^\nu_1 \triangle l_2-\phi^{\iota}_2l^\nu_2 \triangle l_2)\\[8pt]
\displaystyle
\quad +a_5(l^\nu_1 n_1\phi^{3\iota}_1 H(u_1)-l^\nu_2 n_2\phi^{3\iota}_2 H(u_2))\\[8pt]
\displaystyle
\quad +a_6(l^{\nu+1}_1 \phi^{-\iota}_1\text{div} \psi_1-l^{\nu+1}_2 \phi^{-\iota}_2\text{div} \psi_2)+\Theta(\phi_1,l_1,\psi_1)-\Theta(\phi_2,l_2,\psi_2),\\[8pt]
\quad \bar{h}_t+u_1\cdot \nabla\bar{h}+\bar{u}\cdot\nabla h_2+(\delta-1)(\bar{h} \text{div}u_2 +h_1\text{div}\bar{u})=0,\\[8pt]
 \displaystyle
\quad \bar{\psi}_t+\sum_{k=1}^3 A_k(u_1)
\partial_k\bar{\psi}+B(u_{1})\bar{\psi}+a\delta(\bar{h} \nabla\text{div}u_2 +h_1\nabla \text{div}\bar{u})\\[8pt]
\displaystyle
=-\sum_{k=1}^3A_k(\bar{u})
\partial_k\psi_{2}-B(\bar{u}) \psi_{2},\\[2pt]
\displaystyle
\quad
(\bar{\phi},\bar{u},\bar{l},\bar{h},\bar{\psi})|_{t=0}=(0,0,0,0,0) \quad \text{in}\quad \mathbb{R}^3,\\[4pt]
\displaystyle
\quad
(\bar{\phi},\bar{u},\bar{l},\bar{h},\bar{\psi})\longrightarrow (0,0,0,0,0) \quad \text{as} \ \ |x|\rightarrow \infty \quad \rm for\quad t\geq 0.
\end{cases}
\end{equation}

Set
\begin{equation*}
\begin{split}
\Phi(t)=&\|\bar\phi\|_2^2+\|\bar\psi\|_1^2+|h_1^{\frac14} l_1^{\frac{\nu}{2}} \nabla \bar{l}|_2^2+|h_1^{-\frac{1}{4}}\bar{l}_t|^2_2+|l_1^{-\fr{\nu}{2}}\bar{u}|^2_2
\\
&+|\sqrt{h_1}\nabla\bar{u}|_2^2+|\sqrt{h_1}\text{div}\bar{u}|_2^2
+|l_1^{-\frac{\nu}{2}}\bar{u}_t|^2_2.
\end{split}
\end{equation*}
In a similar way as for  (\ref{Gamma}), one  can show that
\begin{equation*}\begin{split}
&\frac{d}{dt}\Phi(t)+C\big( |h_1^{-\frac14}   \bar{l}_t|_2^2+|h_1^{\frac{1}{4}}l_1^{\frac{\nu}{2}}\nabla\bar{l}_t|_2^2+|\nabla \bar{u}|^2_2+ |l_1^{-\fr{\nu}{2}}\bar{u}_t|^2_2
+ |\sqrt{h_1}\nabla\bar{u}_t|_2^2\big)\leq \tilde{H}(t)\Phi (t),
\end{split}
\end{equation*}
with a  function $\tilde{H}(t)\in L^1([0,T_*])$ satisfying 
$$ \int_{0}^{t}\tilde{H}(s)\ \text{\rm d}s\leq C \quad \text{for} \quad 0\leq t\leq T_*.$$ It follows from  Gronwall's inequality that
$\bar{\phi}=\bar{l}=0$ and $\bar{\psi}=\bar{u}=0$,
which shows  the  desired uniqueness.

\textbf{Step 3.} The time-continuity  follows from  the same  arguments as in Lemma \ref{ls}. 

Thus the proof of Theorem \ref{nltm} is completed.

\end{proof}

\subsection{Limit  to the flow with far field vacuum}
Based on the uniform estimates  \ef{key1kk}, we are ready to prove Theorem \ref{3.1}.

\begin{proof} \textbf{Step 1:} The locally uniform positivity of $\phi$. For any $\eta\in (0,1)$, set 
\begin{equation*}
\phi_0^\eta=\phi_0+\eta,\ \ \psi^\eta_0=\fr{a\delta}{\delta-1}\nabla (\phi_0+\eta)^{2\iota},\ \ h_0^\eta=(\phi_0+\eta)^{2\iota}.
\end{equation*}
Then the corresponding initial compatibility conditions can be written as 
\begin{equation}\label{cc2}
\begin{split}
&\nabla u_0=(\phi_0+\eta)^{-\iota}g^\eta_1,\ \  Lu_0=(\phi_0+\eta)^{-2\iota}g^\eta_2,\\
&\nabla((\phi_0+\eta)^{2\iota} Lu_0)=(\phi_0+\eta)^{-\iota}g^\eta_3,\ \ \nabla l_0=(\phi_0+\eta)^{-\frac{\iota}{2}}g^\eta_4, \\
&\triangle l_0=(\phi_0+\eta)^{-\frac{3}{2}\iota}g^\eta_5,\ \ \nabla((\phi_0+\eta)^{\iota}\triangle l_0)=(\phi_0+\eta)^{-\frac{3}{2}\iota}g^\eta_6,
\end{split}
\end{equation}
where $g_i^\eta (i=1,2,...,6)$ are given as
\begin{equation*}
\begin{cases}
\displaystyle
g_1^\eta=\fr{\phi_0^{-\iota}}{(\phi_0+\eta)^{-\iota}}g_1,\ \ g_2^\eta=\fr{\phi_0^{-2\iota}}{(\phi_0+\eta)^{-2\iota}}g_2,\\[4pt]
\displaystyle
g_3^\eta=\fr{\phi_0^{-3\iota}}{(\phi_0+\eta)^{-3\iota}}(g_3-\fr{\eta\nabla\phi_0^{2\iota}}{\phi_0+\eta}\phi_0^\iota Lu_0),\\[4pt]
\displaystyle
g_4^\eta=\fr{\phi_0^{-\frac{\iota}{2}}}{(\phi_0+\eta)^{-\frac{\iota}{2}}}g_4, \ \ g_5^\eta=\fr{\phi_0^{-\frac{3}{2}\iota}}{(\phi_0+\eta)^{-\frac{3}{2}\iota}}g_5,\\
\displaystyle
g_6^\eta=\fr{\phi_0^{-\frac{5}{2}\iota}}{(\phi_0+\eta)^{-\frac{5}{2}\iota}}(g_6-\fr{\eta\nabla\phi_0^{\iota}}{\phi_0+\eta}\phi_0^{\frac{3}{2}\iota} \triangle l_0).
\end{cases}
\end{equation*}
It follows from  \ef{a}-\ef{2.8*} that there exists a $\eta_1>0$ such that if $0<\eta<\eta_1$, then 
\begin{equation}\label{inda}
\begin{split}
2+\eta+\bar{l}+\|\phi^\eta_0-\eta\|_{D^1_*\cap D^3}+\|u_0\|_{3}+\|\nabla h^\eta_0\|_{L^q\cap D^{1,3}\cap D^2}&\\
+|(h^\eta_0)^{\frac{1}{4}}\nabla^3h^\eta_0|_2+\|\nabla (h^\eta_0)^{\frac{3}{4}}\|_{D^1_*}
+|\nabla (h^\eta_0)^{\frac{3}{8}}|_{4}+|(h^\eta_0)^{-1}|_\infty+|g^\eta_1|_2&\\
+|g^\eta_2|_2+|g^\eta_3|_2+|g^\eta_4|_2+|g^\eta_5|_2
+|g^\eta_6|_2+\|l_0-\bar{l}\|_{D^1_*\cap D^3}+|l^{-1}_0|_\infty &\leq  \bar{c}_0,
\end{split}
\end{equation}
where $\bar{c}_0$ is a positive constant independent of $\eta$, and the detail proof can be found in Remark \ref{etawuguan}. Therefore, it follows from Theorem \ref{nltm} that for the initial data $(\phi^\eta_0,u^\eta_0,l^\eta_0,\psi^\eta_0)$, the problem \ef{nl} admits a unique strong solution  $(\phi^\eta,u^\eta,l^\eta, \psi^\eta)$ in $[0,T_*]\times \mathbb{R}^3$ satisfying the local estimate in \ef{key1kk} with $c_0$ replaced by $\bar{c}_0$, and the life span $T_*$ is also independent of $\eta$.

Moreover,  $\phi^\eta$ is positive locally (independent of $\eta$) as shown below.
\begin{lemma}\label{phieta}
For any $R_0>0$ and $\eta\in (0,1]$, there exists a constant $a_{R_0}$ independent of  $\eta$ such that
\begin{equation}\label{phieta1}
\phi^\eta(t,x)\geq a_{R_0}>0,\ \ \ \forall (t,x)\in [0,T_*]\times B_{R_0}.
\end{equation}
\end{lemma}
The proof  follows from the same argument for   Lemma 3.9 in \cite{zz2}.

\textbf{Step 2:} Taking limit $\eta\rightarrow 0$.
It follows from the uniform estimates in  \ef{key1kk}, Lemma \ref{phieta} and   Lemma \ref{aubin}  that for any $R> 0$,  there exist a subsequence of solutions (still denoted by) $(\phi^{\eta}, u^{\eta},l^{\eta},  h^{\eta} )$ and the limit functions $(\phi,u^{},l^{},h^{},\psi^{})$, such that   as $\eta \rightarrow 0$, 
the convergences in the weak or weak* sense in  corresponding spaces hold. Moreover, one still has  
\begin{equation}\label{ert1}\begin{split}
(\phi^{\eta}, u^{\eta}, l^{\eta}, h^{\eta}) \rightarrow & (\phi,  u, l, h) \ \ \text{ in } \ C([0,T_*];H^2(B_R)),\\
\psi^{\eta} \rightarrow & \psi \ \ \text{ in } \ C([0,T_*];H^1(B_R)).
\end{split}
\end{equation}

Then by  lower semi-continuity of weak convergences, $(\phi,u,l,  \psi)$ satisfies the  estimates  in \ef{key1kk}.
Moreover,   one can easily verify the following relations: 
\begin{equation}\label{rela}
h=\phi^{2\iota}, \ \ \ \psi=\fr{a\delta}{\delta-1}\nabla h=\fr{a\delta}{\delta-1}\nabla\phi^{2\iota},
\end{equation}
 by the initial assumptions 
 $h_0=\phi^{2\iota}_0$ and $\psi_0=\fr{a\delta}{\delta-1}\nabla h_0=\fr{a\delta}{\delta-1}\nabla\phi^{2\iota}_0$.
Then one has that  $(\phi,u,l,\psi)$ is a weak solution to  \ef{2.3}-\ef{2.5} in the sense of distributions.

\textbf{Step 3.}  The uniqueness follows  from  the same argument as for      Theorem \ref{nltm}.

\textbf{Step 4:} Time continuity. First, the time continuity of $(\phi,\psi)$ can be obtained by a  similar argument as for Lemma \ref{ls}.

Next, note that   \ef{key1kk} and Sobolev embedding theorem imply that
\begin{equation}\label{zheng1}
\begin{split}
 u\in C([0,T_*]; H^2)\cap  C([0,T_*]; \text{weak-}H^3) \quad \text{and} \quad   \phi^{\iota}\nabla u\in  C([0,T_*]; L^2).
 \end{split}
\end{equation}
It then follows from $\ef{2.3}_2$ that
$$
\phi^{-2\iota} u_t \in L^2([0,T_*];H^2),\quad (\phi^{-2\iota} u_t)_t \in L^2([0,T_*];L^2),
$$
which implies that
$
\phi^{-2\iota} u_t \in C([0,T_*];H^1)
$.
This and the regularity  estimates for
\begin{equation*}
\begin{split}
a_2Lu&=-l^{-\nu}\phi^{-2\iota}(u_t+u\cdot\nabla u +a_1\phi \nabla l+l\nabla\phi-a_2\phi^{2\iota}\nabla l^\nu \cdot Q(u)-a_3l^\nu \psi  \cdot Q(u))
\end{split}
\end{equation*}
show that   $
 u\in C([0,T_*]; H^{3})$ immediately.

Moreover,  since
$$
\phi^{2\iota}\nabla^2 u \in L^\infty([0,T_*]; H^1)\cap L^2([0,T_*] ; D^2) \quad \text{and} \quad   (\phi^{2\iota}\nabla^2 u)_t \in  L^2([0,T_*] ; L^2),
$$
  the classical Sobolev embedding theorem implies that
$$
\phi^{2\iota}\nabla^2 u\in C([0,T_*]; H^1).
$$
Then the  time continuity of $u_t$ follows easily.

Similarly, \ef{key1kk}  and Sobolev embedding theorem imply that
\begin{equation}\label{shangzheng1}
\begin{split}
&\nabla l \in C([0,T_*]; H^1) \cap  C([0,T_*]; \text{weak-}H^2), \quad   \phi^{\frac{1}{2}\iota}\nabla l\in  C([0,T_*]; L^2),\\
&l^{-\nu}\phi^{-2\iota} l_t \in L^2([0,T_*];H^2),\quad (l^{-\nu}\phi^{-2\iota} l_t)_t \in L^2([0,T_*];L^2),
 \end{split}
\end{equation}
which implies 
$
l^{-\nu}\phi^{-2\iota} l_t \in C([0,T_*];H^1)
$.
This and the regularity  estimates for
\begin{equation*}
\begin{split}
-a_4 \triangle l
=\phi^{-\iota}l^{-\nu}\big( -\phi^{-\iota}(l_t+u\cdot\nabla l) +a_5l^\nu n\phi^{3\iota}H(u)+a_6l^{\nu+1} \phi^{-\iota}\text{div} \psi+\Theta(\phi,l,\psi)\big)
\end{split}
\end{equation*}
show that   
$
 l-\bar{l}\in C([0,T_*]; D^1_*\cap D^3)$ immediately.
The  time continuity of $l_t$ follows from $l_t\in C([0,T_*];H^1(B_R))$, for $R>0$ and 
$$\sup_{0\leq t\leq T_*}\int_{|x|>R}|\nabla l_t|^2dx\leq \left( \sup_{\substack{0 \leq t \leq T_* \\ |x| > R}} h(t,x)^{-\frac{1}{2}} \right)\Big(\sup_{0\leq t\leq T_*}|h^\frac{1}{4}\nabla l_t(t,\cdot)|_2^2\Big)\longrightarrow 0$$
as $R\longrightarrow \infty$. Thus \ef{b} holds.

In summary,  $(\phi,u,l,\psi)$  is the unique strong  solution in $[0,T_*]\times \mathbb{R}^3$ to the Cauchy problem \ef{2.3}-\ef{2.5}. Hence  the proof of  Theorem \ref{3.1} is complete.
\end{proof}

\subsection{The proof for Theorem \ref{th21}.} Now we are ready to prove Theorem \ref{th21}.

\begin{proof} \textbf{Step 1.} 
It follows from the initial assumptions (\ref{2.7})-(\ref{2.8}) and Theorem \ref{3.1} that there exists  a time $T_{*}> 0$ such that the problem \ef{2.3}-\ef{2.5} has a unique strong  solution $(\phi,u,l,\psi)$ satisfying the regularity (\ref{b}), which implies that
\begin{equation*}
\phi\in C^1([0,T_*]\times\mathbb{R}^3), \ \ (u,\nabla u)\in C([0,T_*]\times\mathbb{R}^3), \ \ (l,\nabla l)\in C([0,T_*]\times\mathbb{R}^3).
\end{equation*}

Set $\rho=(\fr{\gamma-1}{A\gamma}\phi)^{\fr{1}{\gamma-1}}$ with $\rho(0,x)=\rho_0$. It follows from   the relations between $(\varphi, \psi)$  and $\phi$ that 
\begin{equation*}
\varphi=a\rho^{1-\delta}, \ \ \psi=\fr{\delta}{\delta-1}\nabla\rho^{\delta-1}.
\end{equation*}

Then multiplying $\ef{2.3}_1$ by $\fr{\partial\rho}{\partial\phi}$, 
$\ef{2.3}_2$ by $\rho$, and $\ef{2.3}_3$ by $Ac_v\Big(\frac{A\gamma}{\gamma-1}\Big)^{\iota}\rho^{\gamma-\frac{1-\delta}{2}}$ respectively
shows that the equations in \ef{8} are satisfied.

Hence, we have shown  that the triple  $(\rho,u,S)$ satisfied the Cauchy problem \ef{8} with \ef{2} and \eqref{QH}-\eqref{7}  in the sense of distributions and the regularities in Definition 1.1. Moreover, it follows  from the continuity equation that $\rho(t,x)>0$ for $(t,x)\in [0,T_*]\times \mathbb{R}^3$. In summary, the Cauchy problem \ef{8} with \ef{2} and \eqref{QH}-\eqref{7} has a unique regular solution $(\rho,u,S)$.

\textbf{Step 2.}  Now we  show that the regular solution obtained above  is  also a classical one to   the problem   \eqref{1}-\eqref{3} with \eqref{4} and \eqref{6}-\eqref{7}   within its life span.

First, according to the regularities of $(\rho,u,S)$ and the fact that
$$\rho(t,x)>0\quad \text{for}\quad (t,x)\in [0,T_*]\times \mathbb{R}^3,$$
one can obtain  
\begin{equation*}
(\rho,\nabla\rho,\rho_t,u,\nabla u, S,  \nabla S)\in C([0,T_*]\times\mathbb{R}^3).
\end{equation*}

Second, by the  Sobolev embedding theorem:
\begin{equation}\label{qianru1}
L^2([0,T_*];H^1)\cap W^{1,2}([0,T_*];H^{-1})\hookrightarrow C([0,T_*];L^2),
\end{equation}
and the regularity \ef{b}, one gets that
\begin{equation*}
tu_t\in C([0,T_*];H^2), \ \ \text{and}\ \ u_t\in C([\tau,T_*]\times \mathbb{R}^3).
\end{equation*}

Next, note that the following elliptic system holds
\begin{equation*}
\begin{split}
a_2Lu=&-l^{-\nu}\phi^{-2\iota}(u_t+u\cdot\nabla u +a_1\phi \nabla l+l\nabla\phi-a_2\phi^{2\iota}\nabla l^\nu \cdot Q(u)\\
&-a_3l^\nu\psi \cdot  Q(u))\equiv l^{-\nu}\phi^{-2\iota}\mathbb{M}.
\end{split}
\end{equation*}

It follows from the definition of regular solutions and  \ef{b} directly  that
\begin{equation*}
t l^{-\nu}\phi^{-2\iota}\mathbb{M}\in L^\infty([0,T_*];H^2),
\end{equation*}
and 
\begin{equation*}\begin{split}
(t l^{-\nu}\phi^{-2\iota}\mathbb{M})_t
=& l^{-\nu}\phi^{-2\iota}\mathbb{M}+t(l^{-\nu})_t\phi^{-2\iota}\mathbb{M}+t l^{-\nu}(\phi^{-2\iota})_t\mathbb{M}\\
&+t l^{-\nu}\phi^{-2\iota}\mathbb{M}_t\in L^2([0,T_*];L^2),
\end{split}
\end{equation*}
which, along with  the  Sobolev embedding theorem:
\begin{equation}\label{qianru2}
L^\infty([0,T_*];H^1)\cap W^{1,2}([0,T_*];H^{-1})\hookrightarrow C([0,T_*];L^r),
\end{equation}
for any $ r \in [2,6)$, yields  that 
\begin{equation*}
t l^{-\nu}\phi^{-2\iota}\mathbb{M}\in C([0,T_*];W^{1,4}), \ \ t\nabla^2u \in  C([0,T_*];W^{1,4}).
\end{equation*}
These and the standard elliptic regularity theory  yield  that $\nabla^2 u\in C((0,T_*]\times \mathbb{R}^3)$. 

Moreover, it follows from the regularities of $S_t$ and \eqref{qianru2} that 
\begin{equation*}
t S_t\in C([0,T_*];W^{1,4}_{loc}),
\end{equation*}
which, along with $\theta=AR^{-1}\rho^{\gamma-1}e^{\frac{S}{c_v}}$, yields  that 
$$S_t\in C((0,T_*]\times \mathbb{R}^3)\quad \text{and}\quad \theta_t\in C((0,T_*]\times \mathbb{R}^3).$$

Finally, it remains to show that $\nabla^2 \theta \in C((0,T_*]\times \mathbb{R}^3)$. It follows from the  $(1.1)_3$,  (1.2)  and (1.6) that
\begin{equation*}
c_v\rho(\theta_t+u\cdot \nabla \theta)+P\text{div}u=\nabla u : \mathbb{T}+\frac{\daleth}{\nu+1}\triangle \theta^{\nu+1},
\end{equation*}
which implies that
\begin{equation}\label{wendutuoyuan}
\frac{\daleth}{\nu+1}\triangle \theta^{\nu+1}=c_v\rho(\theta_t+u\cdot \nabla\theta)+P\text{div}u-\nabla u: \mathbb{T}=\beth.
\end{equation}

By direct calculations, one can rewrite $\beth$  as
\begin{equation*}
\beth=\frac{A}{\gamma-1}\rho^\gamma(l_t+u\cdot\nabla l)-(\frac{A}{R})^\nu\rho^\delta l^\nu H(u).
\end{equation*}

From \ef{key1kk}, it yields that
\begin{equation}\label{beth}
\begin{split}
\|\beth\|_2\leq &C\big(1+|h^{-\frac{1}{4}}l_t|_2+|h^{\frac{1}{4}}\nabla l_t|_2+|\nabla^2l_t|_2\big),\\
\|\beth_t\|_1\leq &C\big(1+|h^{-\frac{1}{4}}l_t|_2+|h^{\frac{1}{4}}\nabla l_t|_2+|\sqrt{h}\nabla^2l_t|_2\\
&+|h^{-\frac{1}{4}}l_{tt}|_2+|h^{\frac{1}{4}}\nabla l_{tt}|_2+|\nabla^2l_t|_2+\|u_t\|_2\big),
\end{split}
\end{equation}
here one has used
\begin{equation}\label{tbeth}
\begin{split}
|\rho^\gamma\nabla^2l_t|_2\leq &|\rho^\gamma h^{-\frac{1}{2}}|_\infty |\sqrt{h}\nabla^2l_t|_2\leq C|\rho|_\infty^{\gamma+ {\frac{1-\delta}{2}}}|\sqrt{h}\nabla^2l_t|_2,\\
|\rho^\gamma\nabla l_{tt}|_2\leq &|\rho^\gamma h^{-\frac{1}{4}}|_\infty |h^{\frac{1}{4}}\nabla l_{tt}|_2\leq C|\rho|_\infty^{\gamma+ {\frac{1-\delta}{4}}}  |h^{\frac{1}{4}}\nabla l_{tt}|_2,
\end{split}
\end{equation}
$C$ is a constant depending  only on $(A, R, c_v, \alpha, \beta, \gamma, \delta, \daleth,c_0)$.

From \ef{key1kk}, \ef{beth} and  \ef{tbeth}, one has
\begin{equation}\label{*}
t\beth\in L^{\infty}([0,T_*];H^2),\quad (t\beth)_t\in L^2([0,T_*];H^1),\quad  \beth\in L^2([0,T_*];H^2),
\end{equation}

Then it follows from \eqref{qianru2} that
\begin{equation*}
t\beth\in C([0,T_*];W^{1,4}), 
\end{equation*}
which, together with  Lemma \ref{zhenok} and \ef{wendutuoyuan}, shows  that $\nabla ^2\theta^{\nu+1}\in C((0,T_*]\times \mathbb{R}^3)$ and  $\theta^{\nu+1}\in L^2([0,T_*];D^1\cap D^4) $.  Since for $\tau>0$ and the bounded ball $B_R$, $$\rho(x,t)>0,\quad l(x,t)>0, \quad \text{for}\quad (x,t)\in [\tau,T_*]\times \overline{B_R},$$      
$\rho(x,t), l(x,t)\in C((0,T_*]\times \mathbb{R}^3) $ and  $\theta=AR^{-1}\rho^{\gamma-1}l$, one can deduce that  $\nabla^2\theta\in C((0,T_*]\times \mathbb{R}^3)$.

\textbf{Step 3.} We  show finally  that
if   $m(0)<\infty$, then $(\rho,u,S)$ preserves the conservation of   $(m(t),\mathbb{P}(t),E(t))$.
First, we  show that  $(m(t),\mathbb{P}(t),E(t))$ are all finite.
\begin{lemma}
\label{lemmak-1} 
Under the additional assumption,   $0<m(0)<\infty$, it holds that
$$  m(t)+| \mathbb{P}(t)|+E(t)<\infty \quad \text{for} \quad t\in [0,T_*]. $$
\end{lemma}
This lemma  can be proved by  the same argument used in  Lemma 3.13 of \cite{dxz}.

Now we   prove the conservation of  total mass, momentum and total energy.
\begin{lemma}
\label{lemmak}Under the additional assumption,    $0<m(0)<\infty$,  it holds that 
$$  m(t)=m(0),\quad  \mathbb{P}(t)=\mathbb{P}(0),\quad  E(t)=E(0) \quad \text{for} \quad t\in [0,T_*]. $$

\end{lemma}
\begin{proof}
First, $(\ref{1})_2$ and the regularity of the solution imply  that
\begin{equation}\label{deng1}
\mathbb{P}_t=-\int \text{div}(\rho u \otimes u)-\int \nabla P+\int \text{div}\mathbb{T}=0,
\end{equation}
where one has used the fact that
$$
\rho u^{(i)}u^{(j)},\quad \rho^\gamma e^{\frac{S}{c_v}} \quad \text{and} \quad \rho^\delta e^{\frac{S}{c_v}\nu} \nabla u \in W^{1,1}(\mathbb{R}^3)\quad \text{for} \quad i,\ j=1,\ 2,\ 3.
$$

Second,  the energy equation $(\ref{1})_3$ implies  that
\begin{equation}\label{dengn1}
\begin{split}
E_t=&-\int \text{div}(\rho \mathcal{E} u+ Pu-u\mathbb{T}-\kappa(\theta)\nabla \theta)=0,
\end{split}
\end{equation}
where the following facts have been used:
$$
\frac{1}{2}\rho |u|^2u,\quad \rho^\gamma e^{\frac{S}{c_v}} u, \quad \rho^\delta e^{\frac{S}{c_v}\nu} u \nabla u  \quad \text{and} \quad \rho^\delta e^{\frac{S}{c_v}\nu}  \nabla (\rho^{\gamma-1}e^{\frac{S}{c_v}}) \in W^{1,1}(\mathbb{R}^3).
$$

Similarly, one  can show  the conservation of the total mass.
\end{proof}

Hence the proof of Theorem \ref{th21} is complete.
\end{proof}

\section{A non-decay consequence of conservation}

We first prove the non-decay result stated in Theorem  \ref{th25}. Let $T>0$ be any constant, and    $(\rho,u,\theta)\in D(T)$. One easily has 
$$
 |\mathbb{P}(t)|\leq \int \rho(t,x)|u(t,x)|\leq  \sqrt{2m(t)E_k(t)},
$$
which, together with the definition of the solution class $D(T)$, implies that
$$
0<\frac{|\mathbb{P}(0)|^2}{2m(0)}\leq E_k(t)\leq \frac{1}{2} m(0)|u(t)|^2_\infty \quad \text{for} \quad t\in [0,T].
$$
Then one obtains that there exists a positive constant $C_u=\frac{|\mathbb{P}(0)|}{m(0)}$ such that
$$
|u(t)|_\infty\geq C_u  \quad \text{for} \quad t\in [0,T].
$$
Thus one obtains     the desired conclusion  in Theorem \ref{th25}.

Second, let  $ (\rho,u,S)$ be the  regular  solution to the Cauchy problem \eqref{8} with \eqref{2} and \eqref{QH}-\eqref{7}   in $[0,T]\times\mathbb{R}^3$ obtained   in Theorem  \ref{th21}. It follows from Theorem  \ref{th21} that, if one assumes   $0<m(0)<\infty$ and $|\mathbb{P}(0)|>0$ additionally,
$(\rho,u,\theta)$ is a classical solution to   the Cauchy problem   \eqref{1}-\eqref{3} with \eqref{4} and \eqref{6}-\eqref{7} in $(0,T]\times\mathbb{R}^3$, and also  preserves the conservation of   $(m(t),\mathbb{P}(t),E(t))$. Then one has $(\rho,u,\theta)\in D(T)$, which, along with  Theorem \ref{th25}, yields   Corollary \ref{co23}.

\section*{Appendix: some basic lemmas}

For convenience of readers, we  list some basic facts which have been  used frequently in this paper.

\bigskip

The first one is the  well-known Gagliardo-Nirenberg inequality.

\begin{lemma}\cite{oar}\label{lem2as}\
Assume that  $f\in L^{q_1}\cap D^{i,r}(\mathbb{R}^d)$ for   $1 \leq q_1,  r \leq \infty$.  Suppose also that  real numbers $\Xi$ and $q_2$,  and  natural numbers $m$, $i$  and $j$ satisfy
	$$\frac{1}{{q_2}} = \frac{j}{d} + \left( \frac{1}{r} - \frac{i}{d} \right) \Xi + \frac{1 - \Xi}{q_1} \quad \text{and} \quad 
	\frac{j}{i} \leq \Xi \leq 1.
	$$
	Then $f\in D^{j,{q_2}}(\mathbb{R}^d)$, and  there exists a constant $C$ depending only on $i$, $d$, $j$, $q_1$, $r$ and $\Xi$ such that
	\begin{equation}\label{33}
	\begin{split}
	\| \nabla^{j} f \|_{L^{{q_2}}} \leq C \| \nabla^{i} f \|_{L^{r}}^{\Xi} \| f \|_{L^{q_1}}^{1 - \Xi}.
	\end{split}
	\end{equation}
	Moreover, if $j = 0$, $ir < d$ and $q_1 = \infty$, then it is necessary to make the additional assumption that either f tends to zero at infinity or that f lies in $L^s(\mathbb{R}^d)$ for some finite $s > 0$;
	if $1 < r < \infty$ and $i -j -d/r$ is a non-negative integer, then it is necessary to assume also that $\Xi \neq 1$.
\end{lemma}


The second  lemma is on  compactness theory obtained via the Aubin-Lions Lemma.
\begin{lemma}\cite{jm}\label{aubin} Let $X_0\subset X\subset X_1$ be three Banach spaces.  Suppose that $X_0$ is compactly embedded in $X$ and $X$ is continuously embedded in $X_1$. Then the following statements hold.
	
	\begin{enumerate}
		\item[i)] If $J$ is bounded in $L^r([0,T];X_0)$ for $1\leq r < +\infty$, and $\frac{\partial J}{\partial t}$ is bounded in $L^1([0,T];X_1)$, then $J$ is relatively compact in $L^r([0,T];X)$;
		
		\item[ii)] If $J$ is bounded in $L^\infty([0,T];X_0)$  and $\frac{\partial J}{\partial t}$ is bounded in $L^r([0,T];X_1)$ for $r>1$, then $J$ is relatively compact in $C([0,T];X)$.
	\end{enumerate}
\end{lemma}

Finally, one needs  the following regularity theory  for 
\begin{equation}\label{ok}
-\alpha\triangle u-(\alpha+\beta)\nabla \text{div}u =Lu=F, \quad u\rightarrow 0 \quad \text{as} \quad |x|\rightarrow \infty.
\end{equation}
\begin{lemma}\cite{harmo}\label{zhenok}
Let $k\geq 0$ be any fixed  integer and $F\in D^{k,r}$.  If $u\in D^{1,r}$ with $1< r< +\infty$ is a weak solution to  (\ref{ok}), then
$$
|u|_{D^{k+2,r}} \leq C |F|_{D^{k,r}},
$$
where $C$ depends only on $\alpha$,  $\beta$  and $r$.
\end{lemma}

We point out that the regularity theory obtained in the above lemma still works for the scalar Poisson equation.

\bigskip

\noindent{\bf Acknowledgement:} 
This research is partially supported by National Key R$\&$D Program of China (No. 2022YFA1007300),  the Fundamental Research Funds for the Central Universities,  Zheng Ge Ru Foundation, Hong Kong RGC Earmarked Research Grants CUHK-14301421, CUHK-14301023, CUHK-14302819, and CUHK-14300819. Duan's research is also supported in part by National Natural Science Foundation of China under Grants  12271369 and 12371235. Xin's research is also supported in part by the key projects of National Natural Science Foundation of China (No.12131010 and No.11931013) and Guangdong Province Basic and Applied Basic Research Foundation 2020B1515310002. Zhu's research is also  supported in part by  National Natural Science Foundation of China under Grants 12471212, 12101395 and 12161141004, The Royal Society-Newton International Fellowships Alumni AL/201021 and AL/211005.

\bigskip

\noindent{\bf Data availability:}
Data sharing is  not applicable to this article as no data sets were generated or analysed during the current study. 

\bigskip

\noindent{\bf \large Declarations}

\bigskip

\noindent{\bf Conﬂict of interest:} The authors declare  that they have no conflict of
interest.

\end{document}